\documentclass{amsart}
\usepackage[utf8]{inputenc}
\usepackage[T1]{fontenc}
\usepackage{fixme}
\fxsetup{theme=color}
\usepackage{a4wide}
\usepackage{amsmath}
\usepackage{amssymb}
\usepackage{bbold}
\usepackage{graphicx}
\usepackage[portuguese,english]{babel}
\usepackage[T1]{fontenc}
\usepackage{enumitem}
\usepackage{mathtools}
\usepackage{array}
\usepackage{tikz}
\usepackage{tikz-cd}
\usetikzlibrary{arrows}
\usetikzlibrary{shapes}
\usetikzlibrary{positioning}
\usepackage[toc,page]{appendix}
\usepackage{float}
\usepackage{subcaption}
\usepackage{imakeidx}
\usepackage{setspace}
\usepackage{hyperref}

\newcommand\R{\mathbb{R}}
\newcommand\C{\mathbb{C}}
\newcommand\N{\mathbb{N}}
\newcommand\Q{\mathbb{Q}}

\newcommand\Z{\mathbb{Z}}

\newcommand{\mA}{\mathcal{A}}
\newcommand{\mB}{\mathcal{B}}
\newcommand{\mC}{\mathcal{C}}
\newcommand{\mD}{\mathcal{D}}
\newcommand{\mE}{\mathcal{E}}
\newcommand{\mF}{\mathcal{F}}
\newcommand{\mG}{\mathcal{G}}

\newcommand{\mJ}{\mathcal{J}}

\newcommand{\mM}{\mathcal{M}}
\newcommand{\mN}{\mathcal{N}}
\newcommand{\mO}{\mathcal{O}}

\newcommand{\mS}{\mathcal{S}}
\newcommand{\mT}{\mathcal{T}}

\newcommand{\B}[1]{{\mathbf #1}}
\newcommand{\OP}{\operatorname}
\newcommand{\cp}{\B C\B P}

\newcommand{\om}{\omega}

\DeclareMathOperator{\id}{id}
\DeclareMathOperator{\Emb}{Emb}
\DeclareMathOperator{\IEmb}{\Im\!\Emb}
\DeclareMathOperator{\Symp}{Symp}
\DeclareMathOperator{\BSymp}{B\Symp}
\DeclareMathOperator{\ESymp}{E\Symp}
\DeclareMathOperator{\Ham}{Ham}
\DeclareMathOperator{\Diff}{Diff}

\DeclareMathOperator{\Aut}{Aut}
\DeclareMathOperator{\Iso}{Iso}
\DeclareMathOperator{\BG}{B\,\mG}
\DeclareMathOperator{\pt}{pt}
\DeclareMathOperator{\Sp}{Sp}

\DeclareMathOperator{\Stab}{Stab}
\DeclareMathOperator{\BStab}{B\Stab}

\DeclareMathOperator{\Fix}{Fix}
\DeclareMathOperator{\ev}{ev}

\DeclareMathOperator{\PGL}{PGL}
\DeclareMathOperator{\PSL}{PSL}
\DeclareMathOperator{\U}{U}
\DeclareMathOperator{\Fr}{Fr}
\DeclareMathOperator{\PU}{PU}
\DeclareMathOperator{\SU}{SU}
\DeclareMathOperator{\PSU}{PSU}
\DeclareMathOperator{\BU}{B\U}
\DeclareMathOperator{\BSU}{B\SU}

\DeclareMathOperator{\BPU}{B\PU}

\DeclareMathOperator{\Conf}{Conf}
\DeclareMathOperator{\PD}{PD}
\DeclareMathOperator{\Map}{Map}

\newcommand{\w}{\omega}
\newcommand{\into}{\hookrightarrow}

\newtheorem{theorem}{Theorem}[section]
\newtheorem{lemma}[theorem]{Lemma}
\newtheorem{proposition}[theorem]{Proposition}
\newtheorem{corollary}[theorem]{Corollary}
\newtheorem{definition}[theorem]{Definition}
\newtheorem{remark}[theorem]{Remark}

\newtheorem{conjecture}[theorem]{Conjecture}

\newtheorem*{theorem*}{Theorem}
\newtheorem*{lemma*}{Lemma}
\newtheorem*{proposition*}{Proposition}
\newtheorem*{corollary*}{Corollary}
\newtheorem*{definition*}{Definition}
\newtheorem*{remark*}{Remark}

\newcommand{\eoesymbol}{$\between$}

\DeclareRobustCommand{\eoe}{%
  \ifmmode \mathqed
  \else
    \leavevmode\unskip\penalty9999 \hbox{}\nobreak\hfill
    \quad\hbox{\eoesymbol}%
  \fi
}

\makeindex[title=List of symbols, columns=1]

\begin{document}

\title[Embeddings of symplectic balls]{Embeddings of symplectic balls into the complex projective plane }

\author[S. Anjos]{S\'ilvia Anjos}
\address{SA: Center for Mathematical Analysis, Geometry and Dynamical Systems \\ Department of Mathematics \\  Instituto Superior T\'ecnico \\  Av. Rovisco Pais \\ 1049-001 Lisboa \\ Portugal}
\email{sanjos@math.ist.utl.pt}

\author[J. K\k{e}dra]{Jarek K\k{e}dra}
\address{JK: Department of Mathematics \\  University of Aberdeen \\ Fraser Noble Building
Aberdeen AB24 3UE \\
Scotland }
\email{kedra@abdn.ac.uk}

\author[M. Pinsonnault]{Martin Pinsonnault}
\address{MP: Department of Mathematics \\  University of Western Ontario \\ Canada}
\email{mpinson@uwo.ca}

\begin{abstract}
We investigate spaces of symplectic embeddings of $n\leq 4$ balls into the complex projective plane. 
We prove that they are homotopy equivalent to explicitly described algebraic subspaces of the configuration spaces of $n$ points.
We compute the rational homotopy type of these embedding spaces and their cohomology with rational coefficients. 
Our approach relies on the comparison of the action of $\PGL(3,\C)$ on the configuration space of $n$ ordered points in $\cp^2$ with the action of the symplectomorphism group $\Symp(\cp^2)$ on the space of $n$ embedded symplectic balls.
\end{abstract}

\subjclass[2010]{Primary 53D35; Secondary 57R17,57S05,57T20}
\keywords{Symplectic topology; spaces of symplectic embeddings; configuration spaces; group actions}

\maketitle

\begin{spacing}{0.75}
\tableofcontents
\end{spacing}

\section{Introduction}

\paragraph{\bf Background.}
The study of symplectic embeddings is of major importance in symplectic topology, as it directly addresses the question of what it means to be symplectic. An  example of a symplectic embedding result is the Gromov's Nonsqueezing Theorem~\cite{Gr85} which is a fundamental manifestation of symplectic rigidity, and which gives rise to the notion of symplectic capacity. Most of the research on this field has been focused on \emph{existence} problems and much less is known about the topological properties of embedding spaces \emph{per se}. Some notable exceptions are the results about connectedness obtained by Biran~\cite{Bi96} and McDuff~\cite{McD98,McD09}. Other results about the space of symplectic embeddings of one or two balls in rational ruled 4-manifolds can be found in~\cite{LP04,P08i,ALP09}. In particular, in~\cite{ALP09} it is shown that for certain values of the capacity of the ball this space does not have the homotopy type of a finite $CW$-complex. Note that most results have been proven in dimensions $2$ and $4$. On the other hand,  Chaidez and  Munteanu~\cite{CM21}  obtained  interesting results  about the homology groups of spaces of symplectic embeddings between ellipsoids in any dimension.

Consider a symplectic  $4$-manifold $(M,\omega)$. Let $B^4(c)\subset\R^4$ 
\index{B@$B^4(c)$ -- closed standard ball of capacity $c$}
be the closed standard ball of radius $r$ and capacity $c=\pi r^2$ equipped with the restriction of the standard symplectic structure of $\R^4$, and let $\Symp(B^4 (c))$ be the group of symplectomorphisms of $B^{4}(c)$ that extend to some open neighbourhood of $B^4(c)$. Let $\Emb_{n}(\B c;M)$ 
\index{E@$\Emb_{n}(\B c;M)$ -- space of symplectic embeddings of $n$ disjoint balls}
denote the space, equipped with the $C^\infty$-topology, of symplectic embeddings of $n$ disjoint balls $\mB_{\B c}:= B^4(c_1) \, \sqcup \hdots  \sqcup \,  B^4(c_n)$
\index{B@$\mB_{\B c}:= B^4(c_1) \, \sqcup \hdots  \sqcup \,  B^4(c_n)$}
of capacities $\B c:= (c_1,\ldots, c_n)$ into $(M,\omega)$. Let $\IEmb_{n}(\B c; M)$
\index{I@{$\IEmb_{n}(\B c, M)$ -- space of subsets of $M$ that are images of maps belonging to $\Emb_{n}(\B c, M) $}}
be the space of subsets of $M$ that are images of maps belonging to $\Emb_{n}(\B c;M)$, which we topologize as the quotient
\begin{equation*} 
\IEmb_{n}(\B c, M):= \Emb_{n}(\B c, M) / \prod_i \Symp(B^4 (c_i)),
\end{equation*}
where the group $\Symp(B^4 (c_i))$ acts by reparametrizations of the ball $B^4(c_i)$. We say $\IEmb_{n}(\B c, M)$ is the space of unparametrized symplectic balls of capacities $c_1,\dots, c_n$ in $M$. Finally, let $\Conf_n(M)$ denote the configuration space of $n$ distinct and ordered points in $M$, that is,
\[\Conf_n(M) = \{ (x_1,\hdots, x_n) \in M^n  ~|~ x_i \neq x_j \mbox{ for } i\neq j \}.\]  
\index{C@$\Conf_n(M)$ --  space of $n$ distinct and ordered points in $M$}
The homotopy type of $\IEmb_{n} (\B c, \cp^2)$, for $n \in \{1,2\}$ was computed by Pinsonnault in~\cite{P08i}: $\IEmb_1 (\B c, \cp^2)\simeq \cp^2$ while $\IEmb_2 (\B c, \cp^2)\simeq \Conf_2(\cp^2)$. In particular, in these two cases, it is independent of the choice of the capacities. Moreover, the case $n=1$ is a particular instance of a recent result in~\cite{ALLP23}, where the authors showed that for a rational manifold $M$ with $\chi (M) \leq 11$ the space $\IEmb_1( c, M)$ is weakly homotopy equivalent to $M$ whenever the capacity $c$ of the ball is smaller than the symplectic area of any embedded symplectic sphere of negative self-intersection in the blow up $(\widetilde M_c, \tilde \omega_c)$.

\paragraph{\bf Main results.} In this paper we consider the symplectic manifold $(\cp^2, \omega_{FS})$ and compute the rational homotopy type and the cohomology with rational coefficients of the space  $\IEmb_{n}(\B c, \cp^2)$, for $n \in \{3,4 \}$. We  assume that $c_1 \geq \hdots \geq c_n$ and we normalize $\cp^2$ so that the area of a line is~1. The spaces $\Emb_{n}(\B c; \cp^2)$, with $n \leq 4$, are non-empty if and only if $c_i+c_j < 1$ for all $i \neq j$ and, by the work of McDuff~\cite{McD98}, are connected. 

The configuration space $\Conf_{3}(\cp^{2})$ decomposes into $2$ disjoint strata
\[\Conf_{3}(\cp^{2})=F_{0}\sqcup F_{123},\]
where
\begin{align*}
F_{0} & = \{P\in \Conf_{3}(\cp^{2})~|~\text{the 3 points are in general position}\}\\
F_{123} & = \{P\in \Conf_{3}(\cp^{2})~|~\text{the 3 points are colinear}\}.
\end{align*}
Notice that $F_0=\U(3)/\B T^3$.

\newpage

\begin{theorem}\label{3balls}
Consider $\cp^2$ endowed with its standard Fubini-Study symplectic form  and  let $c_1,c_2,c_3 \in (0,1)$ be such that $c_i+c_j < 1$. 
\begin{enumerate}[label=(\roman*)]
\item If $c_1+c_2+c_3 \geq 1$, the space $\IEmb_{3}(\B c, \cp^2)$ is homotopy equivalent to the flag manifold $F_0=\U(3)/\B T^3$.  
\item If $c_1+c_2+c_3<1$, the space $\IEmb_{3}(\B c, \cp^2)$ is homotopy equivalent to the full configuration space $\Conf_3(\cp^2)=F_{0}\sqcup F_{123}$. 
\end{enumerate}
\end{theorem}

Similarly, the configuration space $\Conf_{4}(\cp^{2})$ decomposes into $6$ disjoint strata
\[\Conf_{4}(\cp^{2})=F_{0}\sqcup F_{234}\sqcup F_{134}\sqcup F_{124}\sqcup F_{123}\sqcup F_{1234}\]
where
\begin{align*}
F_{0} & = \{P\in \Conf_{4}(\cp^{2})~|~\text{no three points of $P$ are colinear}\}\\
F_{ijk} & = \{P\in \Conf_{4}(\cp^{2})~|~\text{only the points~}p_{i},p_{j},p_{k}\text{~are colinear}\}\\
F_{1234} & = \{P\in \Conf_{4}(\cp^{2})~|~\text{all four points of $P$ are colinear}\}
\end{align*}
\index{F@$F_{\bullet}$ -- a stratum in a configuration space of
points of $\cp^2$}

\begin{theorem}\label{homotopy type 4balls}
Consider $\cp^2$ endowed with its standard Fubini-Study symplectic form. Let $0<c_4\leq c_3\leq c_2\leq c_1<1$ be such that $c_i+c_j < 1$. Then the following table gives the homotopy type of the space $\IEmb_{4}(\B c, \cp^2)$ depending on the capacities $c_i$. 
\begin{center}
\begin{tabular}{ ||c |c|| } 
\hline
Capacities & Homotopy type of $\IEmb_{4}(\B c, \cp^2)$ \\  [0.5ex] 
\hline\hline
$c_2+c_3+c_4 \geq 1$ &  $F_{0} \simeq \PU(3)$  \\  [0.5ex] 
\hline
$c_2+c_3+c_4 < 1$ and $c_1+c_3 +c_4 \geq 1$  & $F_1:=F_{0} \sqcup F_{234}$   \\ [0.5ex] 
\hline
$c_1+c_3+c_4 < 1$ and $c_1+c_2 +c_4 \geq 1$  & $F_2:=F_{0} \sqcup F_{234} \sqcup F_{134}$   \\ [0.5ex] 
\hline
$c_1+c_2+c_4 < 1$ and $c_1+c_2 +c_3 \geq 1$  & $F_3:=F_{0} \sqcup F_{234} \sqcup F_{134}\sqcup F_{124}$   \\ [0.5ex] 
\hline
$c_1+c_2+c_3 < 1$ and $c_1+c_2 +c_3 +c_4  \geq 1$  & $F_4:=F_{0} \sqcup F_{234} \sqcup F_{134}\sqcup F_{124} \sqcup F_{123}$   \\ [0.5ex] 
\hline
$c_1+c_2 +c_3 +c_4  <  1$  & $F_5:=\Conf_{4}(\cp^{2})$   \\ [0.5ex] 
\hline
\end{tabular}
\end{center}
\end{theorem}

\begin{remark}
The homotopy type of embedding spaces can be seen as a refinement of the Gromov capacity and, as such, should reveal a finer interplay between symplectic rigidity and flexibility. In both Theorem~\ref{3balls} and Theorem~\ref{homotopy type 4balls}, we observe an interesting gradation from a most rigid situation (the case of "big" balls) to a completely flexible regime (the case of "small" balls). 
\end{remark}

Our geometric methods allow for the computation of the rational cohomology ring of the embedding spaces. For $3$ balls, the rational cohomology algebra of the spaces $\Conf_3(\cp^2)$ and $\U(3)/\B T^3$ is known from previous works that rely on algebraic methods (see~\cite{AB14}). We recover the same results in Section~\ref{section:cohomology3balls}. In the case of $4$ balls, we obtain the following description of the cohomology ring $H^*(\IEmb_{4}(\B c, \cp^2);\Q)$.

\begin{theorem}\label{cohomology ring 4 balls}
Consider $\cp^2$ endowed with its standard Fubini-Study symplectic form. Let $0<c_4\leq c_3\leq c_2\leq c_1<1$. Then the  rational cohomology ring of the space $\IEmb_{4}(\B c, \cp^2)$ of $4$ unparametrized balls in the projective plane  is given in the following table, depending on the capacities $c_i$: 
\begin{center}
\begin{tabular}{ |c |c| } 
 \hline 
 Capacities &  $H^*(\IEmb_{4}(\B c, \cp^2);\Q)$ \\  [0.4ex] 
 \hline\hline
 $c_2+c_3+c_4 \geq 1$  &  $\Lambda (\beta, \eta)$ \\  [0.4ex] 
 \hline
$c_2+c_3+c_4 < 1$ and $c_1+c_3 +c_4 \geq 1$   &  $\Lambda(\alpha_1,\eta) / (\alpha_1^2)$ \\ [0.4ex] 
 \hline
$c_1+c_3+c_4 < 1$ and $c_1+c_2 +c_4 \geq 1$   &  $\Lambda(\alpha_1, \alpha_2,\eta) / (\alpha_1^2+\alpha_1^2, \, \alpha_1\alpha_2)$ \\ [0.4ex] 
 \hline
 $c_1+c_2+c_4 < 1$ and $c_1+c_2 +c_3 \geq 1$   &  $\Lambda(\alpha_1, \alpha_2, \alpha_3, \eta) / (\alpha_1^2+\alpha_2^2 + \alpha_3^2, \, \alpha_i\alpha_j)$, $i \neq j$ \\ [0.4ex] 
 \hline
  $c_1+c_2+c_3 < 1$ and $c_1+c_2 +c_3 +c_4 \geq 1$   &  $\Lambda(\alpha_1, \alpha_2, \alpha_3, \alpha_4,  \eta) / (\alpha_1^2+\alpha_2^2 + \alpha_3^2+ \alpha_4^2, \, \alpha_i\alpha_j)$, $i \neq j$  \\ [0.4ex] 
 \hline
  $c_1+c_2 +c_3 +c_4 < 1$   &  $H^*(\Conf_{4}(\cp^{2});\Q)$ \\
 \hline
 \end{tabular}
\end{center}
Here  $\Lambda$ denotes  an exterior algebra, $\deg \beta =3$, $\deg \eta =5$ and $\deg \alpha_i=2$ for all $1 \leq i \leq 4$. 
\end{theorem}

\begin{remark}
The rational cohomology of $\Conf_{4}(\cp^2)$ can be understood from an algebraic model constructed independently by  Kriz~\cite{Kr94} and Totaro~\cite{Tot96} for configuration spaces $\Conf_n(M)$ whenever $M$ is a projective manifold. However, their model does not give an explicit presentation of the cohomology algebra $H^*(\Conf_{4}(\cp^{2});\Q)$. This is explained in Section \ref{section:small4balls}. 
\end{remark}

It is important to point out that the homotopy equivalences between configurations spaces and embedding spaces given in Theorem~\ref{3balls} and in Theorem~\ref{homotopy type 4balls} are not obtained by constructing explicit maps. Instead, our approach consists in comparing the action of the complex automorphism group $\PGL(3,\C)$ on $\Conf_n(\cp^2)$ with the action of the symplectomorphism group $\Symp(\cp^2)$ on $\IEmb_n(\B c,\cp^2)$. To this end, we use the symplectic blow-up construction to replace balls in $\cp^2$ by exceptional curves in the $n$-fold blow-up $\widetilde{M}_n$ of $\cp^2$. We then replace the  action of $\Symp(\cp^2)$ on $\IEmb_n(\B c,\cp^2)$ by the action of the diffeomorphism group $\Diff_h(\widetilde{M}_n)$ on a certain space $\mA(n,[\Sigma])$ of compatible almost complex structures on $\widetilde{M}_n$ which admits a partition analogous to the stratification of the configuration space $\Conf_n(\cp^2)$. 
Both Theorem \ref{3balls} and \ref{homotopy type 4balls} are proven in Section \ref{SS:main-geometric} and are consequences of the following result which offers a more conceptual explanation as to why they hold.
\begin{theorem}\label{claim introduction}
The natural action of $\PGL(3)$ on the strata of the configuration space $\Conf_n(\cp^2)$ and the action of $\Diff_h(\widetilde{M}_n)$ on the strata of $\mA(n,[\Sigma])$ have equivalent homotopy orbits. 
\end{theorem} 
\noindent
A more precise statement of the above theorem is given as Theorem \ref{claim} which is proven in Section~\ref{section:proof of main thm}.
Since this approach applies equally well for the cases $n=1$ and $n=2$, it provides a uniform treatment of the embedding spaces $\IEmb_n(\B c, \cp^2)$ for $n\leq 4$.

Our results exhibit an interesting duality between complex geometry and symplectic topology on $\cp^2$. More precisely, for $n\leq 4$, genericity conditions on sets of $n$ distinct points on $\cp^2$ translate into numerical conditions on the symplectic capacites of the $n$ disjoint balls that determine the homotopy type of $\IEmb_n(\B c,\cp^2)$. It is an interesting question to see whether this duality still holds for $n\geq 5$. Our work suggests the following conjecture.
\begin{conjecture}\label{conjecture}
For $n\leq 8$, and for any admissible capacities $\B c=c_1\geq\cdots\geq c_n>0$, the space $\IEmb_n(\B c,\cp^2)$ is homotopy equivalent to a union of strata in $\Conf_n(\cp^2)$ defined by the relative positions of $n$ points with respect to generic immersed holomorphic spheres in $\cp^2$. 
\end{conjecture}
The main difficulties in proving this conjecture are twofold. Firstly, the space of admissible capacities has a more complicated structure giving rise to combinatorial difficulties (see Section~\ref{chambers5} for an example). Secondly, our approach relies crucially on the transitivity of the actions of symplectomorphism groups on certain configurations of symplectic spheres of negative self-intersections (as discussed in Section~\ref{section:Actions on ACS}). It is not clear that analogous results hold for the more general configurations that occur when $n\geq 5$.

From a homotopy theoretic point of view, a duality between points and balls is somewhat expected in view of the homotopy equivalence $\Symp(\cp^2)\simeq\PGL(3)$. Indeed, the homotopy fiber $F_{\C}$ of the complex evaluation map $\ev_{\C}\colon \PGL(3) \to \Conf_n(\cp^2)$ can be computed by looking at the symplectic evaluation fibration 
\begin{equation}
\begin{tikzcd}
F_{\C}\simeq \Symp(\cp^2,\B p) \ar[r]&\Symp(\cp^2) \arrow{r}{\ev_{\om}} & \Conf_n(\cp^2) \\
& \PGL(3) \arrow{u}{\simeq}\arrow{ru}[swap]{\ev_{\C}}& 
\end{tikzcd}
\end{equation}
and one may ask what should be the symplectic analogue of the stabilizer $\Symp(\cp^2,\B p)$ when we restrict the action of $\PGL(3)$ to invariant subspaces of $\Conf_n(\cp^2)$. In essence, this is what we investigate in the present paper.

It is worth mentioning that, as a byproduct of our methods, we obtain a homotopy equivalence between the space $\mJ_{\B c}([\Sigma])$ of compatible almost complex structures on the blow-up $(\widetilde{M}_{\B c},\widetilde{\omega}_{\B c})$ for which there exist the classes of exceptional divisors have $J$-holomorphic representatives and the subspace $\mJ_{\B c}^{int}([\Sigma])$ of integrable ones. See Proposition \ref{prop:equivalence with integrable} for details.

\paragraph{\bf Organization of the paper.} 
In Section~\ref{section: framework} we explain a general framework for the study of symplectic balls. We then characterize the stability chambers of the set of admissible capacities, describe the stratification of the configuration space of points and analogous stratifications of various spaces of almost complex structures. At the end of this section we state our main geometric result Theorem~\ref{claim} in full detail. Assuming this, we then give the proofs of Theorem~\ref{3balls} and Theorem~\ref{homotopy type 4balls}. In Section~\ref{section: homotopy orbits PGL} we analyse the action of $\PGL(3)$ on $\Conf_n(\cp^2)$, while Section \ref{section:Actions on ACS} is devoted to analysis of the actions of symplectomorphism and diffeomorphism groups on spaces of complex structures, obtaining homotopy decompositions of the classifying space of stabilizers of balls.  We then combine the results obtained in Sections~\ref{section: homotopy orbits PGL} and~\ref{section:Actions on ACS} to  prove Theorem~\ref{claim} in Section~\ref{section: homotopy orbits Diff}. Then, in Section~\ref{section rational model}, we construct a rational algebraic model for the space of symplectic embeddings of balls into $\cp^2$. Using this model, in Section \ref{section cohomology ring}, we compute the cohomology with rational coefficients of $\IEmb_{n}(\B c, \cp^2)$ proving Theorem~\ref{cohomology ring 4 balls}.

\paragraph{\bf Acknowledgements.}
The first author is grateful to Pedro Brito and Gustavo Granja for helpful and enlightening conversations. The third author is grateful to Siyuan Yu for useful comments on early drafts of this paper and to Denis Auroux for helpful discussions. All three authors would like to thank the support of FCT/Portugal, through projects UID/MAT/04459/2020 and PTDC/MAT-PUR/29447/2017, and the Mathematics Department of Instituto Superior Técnico, where part of this work was completed. 

\section{Symplectic balls in rational \texorpdfstring{$4$}{4}-manifolds} \label{section: framework}

\subsection{General setting}\label{section General Setting} 

In \cite{LP04} Lalonde and Pinsonnault proposed a general framework to study embedding spaces of symplectic balls in symplectic $4$-manifolds through natural actions of symplectomorphism groups. We recall the main points in the special case of rational $4$-manifolds 
for which a number of simplifications occur. 

In the following, let $(M,\om)$ be a symplectic rational $4$-manifold. Given a compatible almost complex structure $J$ which is integrable near $n$ distinct points let $\widetilde{M}_{n}:=M\#\,n\overline{\cp}\,\!^2$
\index{M@$\widetilde{M}_{n}$ -- $n$-fold complex blow-up of $M$}
be the $n$-fold complex blow-up of $M$ at these points. Let $\Sigma_1,\ldots, \Sigma_{n}$ denote the exceptional divisors, let $E_1,\ldots, E_n$ be their homology classes, and let $\PD(E_i)$ denote their Poincaré duals. If $K:=K(M,J)\in H_2(M,\Z)$ is the anti-canonical class associated to $J$, then $K-E_1-\cdots-E_n$ is the anti-canonical class of $\widetilde{M}_{n}$. Note that these anti-canonical classes only depend on the deformation class of the symplectic form $\om$. A homology class $E\in H_2(M,\Z)$ is \emph{exceptional} if $E\cdot E=-1$, $K\cdot E=1$, and if it is represented by a smooth embedded sphere. Let $\mE(M,\om)\subset H_2(M,\Z)$
\index{E@$\mE(M,\om)$ -- set of exceptional classes}
be the set of all such exceptional classes. The set $\mE(M,\om)$ characterizes the symplectic cone of any symplectic rational $4$-manifold of Euler number $\chi(M)\geq 5$.

Let $\mC(M,\om,n)\subset(0,c_{M})^n$
\index{C@$\mC(M,\om,n)$ -- set of capacities $\B c = (c_1,\ldots,c_n)$ for which there exists a symplectic embedding $B^4(c_1)\sqcup\cdots\sqcup B^4(c_n) \into (M,\om)$}
be the set of capacities $\B c = (c_1,\ldots,c_n)$ for which there exists a symplectic embedding $B^4(c_1)\sqcup\cdots\sqcup B^4(c_n) \into (M,\om)$. In the next theorem we collect a few results about symplectic rational $4$-manifolds that will be useful in what follows. 
\begin{theorem}\label{thm:GeneralResultsOnRationalSurfaces}
Let $M$ be a symplectic rational $4$-manifold. 
\begin{enumerate}
\item  Any two cohomologous symplectic forms on $M$ are diffeomorphic \cite{LM96,LL95}.
\item  There exists a symplectic embedding $B^4(c_1)\sqcup\cdots\sqcup B^4(c_n) \into (M,\om)$ if, and only if, the cohomology class
\[
[\widetilde{\om}_{\B c}]:=[\om]-c_1\PD(E_1)-\cdots-c_n\PD(E_n)\in H^2(\widetilde{M}_{n} ,\R)
\]
pairs strictly positively with all exceptional classes in $\mE(\widetilde{M}_{n})$, and if it satisfies the volume condition $\langle[\widetilde{\om}_{\B c}]^2,[\widetilde{M}_{n}]\rangle>0$; see \cite{LL01}. 
\item  If $M$ is a symplectic rational $4$-manifold, then for each $\B c\in \mC(M,\om,n)$, the embedding space $\Emb_{n}(\B{c},M)$ is path-connected \cite[Corollary 1.5]{McD98}.
\end{enumerate}
\end{theorem}

Note that the theorem implies that a symplectic form in $\widetilde{M}_{n}$ is uniquely defined by the cohomology class $[\om]$ and the capacities $\B c =(c_1, \hdots, c_n)$. 

Given capacities $\B c=(c_1,\ldots,c_n)\in \mC(M,\om,n)$, the connectedness of $\Emb_{n}(\B c, M)$ implies that the natural action of the identity component of the symplectomorphism group $\Symp_0(M,\omega)$
\index{S@$\Symp_0(M,\omega)$ -- identity component of the symplectomorphism group $\Symp(M,\omega)$}
on $\Emb_{n}(\B c, M)$ is transitive. After choosing a fixed embedding $\iota:\mB_{\B c}\into M$ we get a ladder of evaluation fibrations 
\begin{equation}\label{fibration ladder}
\begin{tikzcd}
\Symp(M,\iota(\mB_{\B c})_{\id}) \arrow[r]\arrow[d] & \Symp_0(M,\omega) \arrow[r]\arrow[d] & \Emb_{n}(\B c,M) \arrow[d]\\
\Symp(M,\iota(\mB_{\B c})) \arrow[r] & \Symp_0(M,\omega) \arrow[r] & \IEmb_{n}(\B c,M),
\end{tikzcd}
\end{equation}
where $\Symp(M,\iota(\mB_{\B c})_{\id})$
\index{S@$\Symp(M,\iota(\mB_{\B c})_{\id})$ -- subgroup of symplectomorphisms which restrict to the identity on the image $\iota(\mB_{\B c})$}
is the subgroup of symplectomorphisms which restrict to the identity on the image $\iota(\mB_{\B c})$, and where $\Symp(M,\iota(\mB_{\B c}))$
\index{S@$\Symp(M,\iota(\mB_{\B c}))$ -- symplectomorphisms which map the image of each ball $\mB_{\B c}$ to itself}
consists of symplectomorphisms which map the image of each ball $B^4(c_i)\subset\mB_{\B c}$ to itself. Note that it follows from Theorem \ref{thm:GeneralResultsOnRationalSurfaces} that  $\Emb_{n}(\B c,M)$ does not depend on the choice of the fixed embedding $\iota$. As the group $\Symp(B^4(c))$ is  homotopically equivalent to $\U(2)$ (independently of the capacity $c$), the above two fibrations are essentially equivalent.

Let $\widetilde{M_{\B c}}:=(M\#\,n\overline{\cp}\,\!^2, \widetilde{\omega}_{\B c})$
\index{M@$\widetilde{M_{\B c}}$ -- $n$-fold symplectic blow-up of $M$ at the balls $\iota(\mB_{\B c})$}
be the $n$-fold symplectic blow-up of $M$ at the balls $\iota(\mB_{\B c})$. Let $\Sigma=\Sigma_1\sqcup\cdots\sqcup\Sigma_{n}$ be the union of the exceptional divisors. By definition, $\Sigma$ is an (ordered) configuration of disjoint exceptional symplectic spheres in classes $E_1,\ldots,E_n$. Let's write $\mC_{\B c}(\Sigma)$
\index{C@$\mC_{\B c}(\Sigma)$ -- space of configurations of disjoint symplectic spheres}
for the space of all such configurations, and let $\mJ_{\B c}(\Sigma)$
\index{J@$\mJ_{\B c}(\Sigma)$ -- space  of all compatible complex structures on $\widetilde{M_{\B c}}$ for which there are embedded $J$-holomorphic representatives corresponding to $\Sigma$}
be the space of all compatible complex structures on $\widetilde{M_{\B c}}$ for which there are embedded $J$-holomorphic representatives of each class $E_i$. By positivity of intersections, there is exactly one configuration for each $J\in\mJ_\B c(\Sigma)$. Arguing as in~\cite[Section 4.1]{LP04}, one proves that the map $\mJ_{\B c}(\Sigma)\to\mC_{\B c}(\Sigma)$ is a fibration with contractible fibers, so that it is a homotopy equivalence. In particular, since $\mJ_{\B c}(\Sigma)$ is open, dense, and connected in the contractible space of all tamed almost-complex structures, $\mC_{\B c}(\Sigma)$ is itself connected. It follows that $\Symp_0(\widetilde{M}_{\B c})$ acts transitively on $\mC_{\B c}(\Sigma)$ and that there is a fibration
\begin{equation}\label{fibration exceptional configurations}
 \Symp(\widetilde{M}_{\B c},\Sigma)\to \Symp_0(\widetilde{M}_{\B c})\to \mC_{\B c}(\Sigma)\simeq \mJ_{_{\B c}}(\Sigma),
\end{equation}
\index{S@$\Symp(\widetilde{M}_{\B c},\Sigma)$ -- stabilizer of the exceptional divisor $\Sigma$ }
where $\Symp(\widetilde{M}_{\B c},\Sigma)$ is the subgroup of symplectomorphisms sending each exceptional divisor $\Sigma_i$ to itself. As explained in~Section 2 of~\cite{LP04}, there is a homotopy equivalence
\begin{equation}\label{homotopy equivalence ball-fiber}
 \Symp(\widetilde{M}_{\B c},\Sigma) \simeq \Symp(M,\iota(\mB_{\B c})). 
\end{equation}
This yields a homotopy fibration 
\begin{equation}\label{second fibration unparametrized}
 \Symp(\widetilde{M}_{\B c},\Sigma) \to \Symp_0(M,\omega)\to\IEmb_{n}(\B c,M)
\end{equation}
which can be used to compute the homotopy types of $\IEmb_{n}(\B c,M)$ and of $\Emb_{n}(\B c,M)$.

Given two $n$-tuples of capacities $\B c=(c_1,\ldots,c_n)$ and $\B {c'} =\{c_1',\ldots,c_n'\}$, we declare $\B c\leq \B{c'}$ iff $c_i\leq c_i'$ for all $1\leq i\leq n$. Given a non-negative $n$-tuple $\boldsymbol{\epsilon}=\{\epsilon_1,\ldots, \epsilon_n\}$, we write $\B c+\boldsymbol{\epsilon}$ for $\{c_1+\epsilon_1,\ldots,c_n+\epsilon_n$\}. For each pair $\B c$, $\B{c}+\boldsymbol{\epsilon}$, there is a restriction map
\begin{equation}\label{eqn:Fibration frames}
i_{\B c}^{\B{c}+\boldsymbol{\epsilon}}:\Emb_{n}(\B{c}+\boldsymbol{\epsilon},M)\to \Emb_{n}(\B c,M).
\end{equation}
Let $\Sp\Fr(n,M)$
\index{S@$\Sp\Fr(n,M)$ -- space of symplectic frames at $n$ ordered points in $M$}
be the space of symplectic frames at $n$ ordered points in $M$. Evaluation of the derivatives at the centers of the $n$ balls defines a fibration
\begin{equation}\label{eqn:Fibration framed embeddings}
\Emb_{n}^{\B f}(\B c,M)\to\Emb_{n}(\B c,M) \xrightarrow{j_{\B c}} \Sp\Fr(n,M),
\end{equation}
where $\Emb_{n}^{\B f}(\B c,M)$
\index{E@$\Emb_{n}^{\B f}(\B c,M)$ -- embeddings with a fixed framing $\B f$ at the centers}
consists of embeddings with a fixed framing $\B f$ at the centers. Since the evaluation maps commute with restrictions, there is a map
\begin{equation}\label{eqn:j_infty}
\varinjlim \Emb_{n}(\B c,M) \xrightarrow{j_\infty} \Sp\Fr(n,M),
\end{equation}
where the direct limit is taken with respect to the partial order given by reverse inclusions. The next two results are valid for symplectic manifolds $(M^{2m},\om)$ of any dimension.

\begin{lemma}\label{lemma:Thickening} 
For $(M^{2m},\om)$ compact, and for any $n\geq1$, 
\begin{enumerate}
\item there exists capacities $\B {c_0}=(c_1,\ldots,c_n)$ such that the induced map 
\[\pi_*(j_{\B c}):\pi_*(\Emb_{n}(\B c,M)) \to \pi_*(\Sp\Fr(n,M))\]
is surjective for all $\B c\leq \B{c_0}$,
\item the map $j_\infty:\varinjlim \Emb_{n}(\B c,M) \to \Sp\Fr(n,M)$ is a weak homotopy equivalence.
\end{enumerate}
\end{lemma}
\begin{proof}
Let $J$ be a compatible almost complex structure, and let $g$ be the Riemannian metric associated to the pair $(\om,J)$. Let $\exp$ be the corresponding exponential map, and let $r>0$ be the injectivity radius of $(M,g)$. 

Let $\tau:\Sp\Fr(n,M)\to\R_{>0}$ be the function that assign to a frame the minimal distance between any two of its points. Given $\epsilon>0$, let $\Sp\Fr^{\geq\epsilon}(n,M)=\tau^{-1}[\epsilon,\infty)$. For $\epsilon$ small enough, the closed subspace $\Sp\Fr^{\geq\epsilon}(n,M)$ is a strong deformation retract of $\Sp\Fr(n,M)$. The subspace $\Sp\Fr^{\geq\epsilon}(n,M)$ itself deformation retracts onto the compact subspace of unitary frames $\U\Fr^{\geq\epsilon}(n,M)$.

Given a unitary frame $f_{m}$ at a point $m\in M$, we can construct a symplectic embedding of a ball $\phi:B^{4}(c(m))\to M$ of some small capacity $c(m)$ such that $d\phi_{0}(f_{0})=f$, where $f_{0}$ is the standard unitary frame of $\R^{2m}\simeq \C^{m}$. To see this, let $\B f_{m}:\R^{2m}\to T_{m}M$ be the identification of the tangent space given by the symplectic frame $f_{m}$. For $0<\delta<r$, the smooth map $\psi_m(z) = \exp_{m}(\B f_m(z))$ is an embedding of $B^{2m}(\pi\delta^{2})$. Since $d\psi_{m}=\B f$ at the origin, Moser's theorem implies that $\psi_m$ is isotopic to a symplectic embedding $\psi_m'$ of a smaller ball of radius $\delta'$. Let $\delta(f_{m})$ be the supremum of all $\delta'>0$ that are obtained this way. Since Moser's isotopy is continuous with respect to the parameters involved, $\delta(f_{m})$ is continuous in $f_{m}$ and, by compactness of $\U\Fr(1,M)$, we can set
\[\delta_{M}=\frac{1}{3}\min_{f_{m}\in \U\Fr(1,M)}\{\delta_{m},\epsilon\}>0.\]
This procedure defines a continuous map
\[\Psi:\U\Fr^{\geq \epsilon}(n, M)\to\Emb_{n}(\pi\delta_{M}^{2},\ldots,\pi\delta_{M}^{2},M)\]
which is a section of the composition
\[\Emb_{n}(\B c,M) \xrightarrow{j_{\B c}} \Sp\Fr(n,M)\xrightarrow{\simeq} \U\Fr(n,M).\]
This proves the first statement. For the second statement, we look at the fiber of the fibration~\eqref{eqn:Fibration framed embeddings} under restriction maps. Choose a frame $\B f$ and a collection $\psi_i:B^{2m}(\delta)\to M$ of $n$ disjoint framed Darboux charts. Let $A$ be a class in $\pi_q(\Emb_{n}^{\B f}(\B c,M))$ represented by the map $\phi:S^q\to \Emb_{n}^{\B f}(\B c,M)$. By restricting the embeddings to small enough balls, we can ensure that for each $z\in S^q$, the image of the restriction is contained in the chosen charts. Consequently, after restriction, the class $A$ is represented by a collection of $n$ framed embeddings of a single ball $B^{2m}(\epsilon_i)$ into another ball $B^{2m}(\delta)$ whose derivative at the origin is the identity. By further restricting to smaller balls, we can apply Alexander trick to show that each restriction is isotopic to the inclusion. This shows that the class $A$ becomes trivial after restriction to small enough balls. This concludes the proof of the second statement.
\end{proof}

\begin{corollary}\label{cor:RationalHyperbolicity}
Let $(M^{2m},\om)$ be a compact, simply connected, symplectic manifold without boundary. Let $n$ be a positive integer satisfying the following conditions:
\begin{enumerate}
\item $n\geq 4$ if $M=S^2$,  
\item $n\geq 3$ if $H(M;\Q)\simeq \Q[x]/x^{k}$,
\item $n\geq 2$ in all other cases.
\end{enumerate}
Then there exists capacities $\B {c_0}=(c_1,\ldots,c_n)$ such that, for every $\B c\leq \B {c_0}$, the embedding space $\Emb_{n}(\B c,M)$ is rationally hyperbolic and the rational cohomology ring $H^*(\Symp(M,\iota(\mB_{\B c})_{\id});\Q)$ is not finitely generated.
\end{corollary}
\begin{proof}
A result of Y. Felix and J.-C. Thomas asserts that under the conditions of the statement, the configuration space $\Conf_{n}(M)$ is rationally hyperbolic, see~\cite[Corollary p.~562]{FT94}. Looking at the fibration
\[\Sp(2m)^n\to\Sp\Fr(n,M)\to\Conf_{n}(M)\]
and recalling that finite dimensional Lie groups are rationally elliptic (see~\cite{FOT08} p.~86), it follows that $\Sp\Fr(n,M)$ is also rationally hyperbolic. On the other hand, Theorem~1.1 in~\cite{Ke05} asserts that if a topological group $G$ acts transitively on a simply connected rationally hyperbolic space $X$, then the rational cohomology ring of the isotropy subgroup of a point $H_*(G_{\pt};\Q)$ is infinitely generated. Together with Lemma~\ref{lemma:Thickening}, this proves the statement.
\end{proof}

\begin{remark}
In general, it is not known whether there exist capacities $\B {c_0}=(c_1,\ldots,c_n)$ such that the map $\Emb_{n}(\B{c},M) \to \Sp\Fr(n,M)$ is a weak homotopy equivalence for every $\B{c}\leq\B {c_0}$. See the discussion of stability in Section~\ref{section:stability}.
\end{remark}

\subsection{Configurations of balls and points}
Unlike $\Emb_{n}(\B c,M)$, the spaces $\IEmb_{n}(\B c,M)$ do not form a directed system. However,
due to the stability properties described in the next section they are homotopy equivalent to configuration spaces of points for suitable $\B c$. In what follows we make a preparatory observation.

Let $\Symp(B^4(\B c)) = \prod_i \Symp(B^4(c_i))$
\index{S@$\Symp(B^4(\B c))$ -- group of symplectic reparametrisations of the disjoint union of balls }
be the group of symplectic reparametrisations of the disjoint union of balls $B^4(c_1)\sqcup \dots \sqcup B^4(c_n)$ and let $\Symp(B^4(\B c),0)$ be the subgroup fixing the origins $0_i\in B^4(c_i)$. Notice that the orbit of $(0_1,\ldots,0_{n})$ with respect to the natural action of $\Symp(B^4(\B c))$ on the product $B^4(c_1)\times \dots \times B^4(c_n)$ is the whole interior of the product and the isotropy subgroup is $\Symp(B^4(\B c),0)$. This shows that the inclusion $\Symp(B^4(\B c),0)\subseteq \Symp(B^4(\B c))$ is a homotopy equivalence. In particular, the projection
\[
\IEmb_{n}^*(\B c,M):=\Emb_{n}(\B c,M)/\Symp(B^4(\B c),0)\to \Emb_{n}(\B c,M)/\Symp(B^4(\B c))=\IEmb_{n}(\B c,M)
\]
is a homotopy equivalence.
Consider the following diagram,
\begin{equation}
\begin{tikzcd}
\Symp(B^4(\B c),0) \ar[r] \ar[d] & \Sp(2m)^n\ar[d]\\
\Emb_{n}(\B c,M) \ar[r] \ar[d]   & \Sp\Fr(n,M)\ar[d]\\
\IEmb_{n}^*(\B c,M):=\Emb_{n}(\B c,M)/\Symp(B^4(\B c),0) \ar[r]\ar[d,leftrightarrow] & \Conf_{n}(M)\\
\IEmb_{n}(\B c,M), & \\
\end{tikzcd}
\label{Eq:Emb-Conf}
\end{equation}
in which the rightmost column is the symplectic frame bundle over the configuration space and the two first horizontal maps are defined by evaluating the differential at the center. Observe that an element of the space $\IEmb_{n}^*(\B c,M):=\Emb_{n}(\B c,M)/\Symp(B^4(\B c),0)$ is a configuration of $n$ symplectic balls {\it with centers}. Therefore the bottom horizontal map is well defined. Composing it with the bottom homotopy equivalence we get the map
\begin{equation}
\IEmb_{n}(\B c,M) \to \Conf_{n}(M)
\label{Eq:IEmb-Conf}
\end{equation}
defined for every choice of capacities.

\begin{remark}
Corollary~\ref{cor:RationalHyperbolicity} implies that for generic symplectic manifolds and capacities, part of the homotopical complexity of the embedding space $\IEmb_{n}(\B{c},M)$ is of purely topological nature and does not encode symplectic information. What really matters are the values of capacities $\B c$ at which the homotopy type of $\IEmb_{n}(\B{c},M)$ changes. The simplest way to factor out the contribution of the configuration space $\Conf_n(M)$ to the homotopy of $\IEmb_{n}(\B{c},M)$ is to study the space $\IEmb_{n}^{\B p}(\B c,M)$ of balls with fixed centers $\B p:=\{p_1,\ldots,p_n\}$. However, looking at the fibration
\begin{equation}
\Symp(\widetilde{M}_{\B c},\Sigma) \simeq \Symp(M,\iota(\mB_{\B c}),\B p)
\to
\Symp_0(M,\B p)
\to
\IEmb_n^{\B p}(\B c, M)
\end{equation}
we see that this is still equivalent to analysing the symplectic stabilizer $\Symp(\widetilde{M}_{\B c},\Sigma)$.
\end{remark}

\subsection{Stability for embedding spaces}\label{section:stability}
\begin{definition}
Let $(M,\om)$ be a rational $4$-manifold and let $\B c_0$, $\B c_1$ be two sets of capacities in $\mC(M,\om,n)$. We say that $\B c_0$ and $\B c_1$ are in the same stability component if there exists a continuous family of capacities $\B c_t\subset\mC(M,\om,n)$ interpolating $\B c_0$ and $\B c_1$ for which the homotopy type of $\Emb_{n}(\B c_t,M)$ is constant.
\end{definition}

The dependence of the homotopy type of $\Symp(\widetilde{M}_{\B c},\Sigma)$ on the capacities $\B c=(c_1,\cdots, c_{n})$ can be investigated through a fibration whose total space only depends on the deformation class of the symplectic forms. Let
$\Diff_{[\B c]}(\widetilde{M},\Sigma)$
\index{D@$\Diff_{[\B c]}(\widetilde{M},\Sigma)$ -- group of diffeomorphisms of the blow-up $\widetilde{M}_{\B c}$ that preserves the class $[\widetilde{\om}_{\B c}]$ and that leave the exceptional divisor $\Sigma$ invariant}
be the group of diffeomorphisms of the blow-up $\widetilde{M}_{\B c}$ that preserves the class $[\widetilde{\om}_{\B c}]$ and that leave the exceptional divisor $\Sigma$ invariant. Let $\Omega_{\B c}:=\Omega([\widetilde{\om}_{\B c}])$ 
\index{O@$\Omega_{\B c}$ -- space of symplectic forms cohomologous to $\widetilde{\om}_{\B c}$}
be the space of symplectic forms cohomologous to $\widetilde{\om}_{\B c}$ and let $\Omega_{\B c}(\Sigma)$
\index{O@$\Omega_{\B c}(\Sigma)$ -- subspace of $\Omega_{\B c}$ for which $\Sigma$ is symplectic}
be the subspace of forms for which $\Sigma$ is symplectic. By applying a relative version of Moser's lemma to each component of $\Omega_{\B c}(\Sigma)$, and using Part (1) of Theorem~\ref{thm:GeneralResultsOnRationalSurfaces}, we get that there is an evaluation fibration
\begin{equation}\label{relativeUFIB}
\Symp(\widetilde{M_{\B c}}, \Sigma)
\to 
\Diff_{[\B c]}(\widetilde{M},\Sigma) 
\to
\Omega_{\B c}(\Sigma).
\end{equation}
Similarly, we define the space of pairs 
\begin{equation}\label{eq:space of pairs}
P_{\B c}(\Sigma) =\left\{ (\om',J)~|~\om'\in\Omega_{\B c}(\Sigma),~J\text{~is compatible with~}\om',~\Sigma\text{~is~}J\text{-holomorphic} \right\}
\end{equation}
and the space of compatible almost-complex structures
\[\mA_{\B c}(\Sigma) = \left\{J\text{~is compatible with some~}\om'\in\Omega_{\B c}(\Sigma)\text{~and~}\Sigma\text{~is~}J\text{-holomorphic}\right\}.\]
\index{A@$\mA_{\B c}(\Sigma)$ -- space of almost complex structures $J$ which are compatible with some form in $\Omega_{\B c}(\Sigma)$ and for which $\Sigma$ is $J$-holomorphic}
Then the projection maps 
\begin{equation}\label{eq:two projections}
\Omega_{\B c}(\Sigma)\leftarrow P_{\B c}(\Sigma)\to \mA_{\B c}(\Sigma)
\end{equation}
are homotopy equivalences. As explained in~\cite[Section~4.1]{ALLP23}, the homotopy fiber of the evaluation map $\Diff_{[\B c]}(\widetilde{M},\Sigma)\to\mA_{\B c}(\Sigma)$ is homotopy equivalent to $\Symp(\widetilde{M_{\B c}}, \Sigma)$. In particular, the sequence of maps
\begin{equation} \label{relative homotopy fibration}
\Symp(\widetilde{M_{\B c}}, \Sigma)\into \Diff_{[\B c]}(\widetilde{M},\Sigma) \to \mA_{\B c}(\Sigma).
\end{equation}
induces a long exact sequence of homotopy groups. Consequently, as the capacities $\B c$ vary, the homotopy types of $\Symp(\widetilde{M_{\B c}}, \Sigma)$ and of $\IEmb_{n}(\B c,M)$ change precisely when the homotopy type of the evaluation map $\Diff_{[\B c]}(\widetilde{M},\Sigma) \to\mA_{\B c}(\Sigma)$ changes.

\subsection{Stability chambers for symplectic balls in \texorpdfstring{$\cp^2$}{CP2}}\label{section: stability chambers} 
We now apply the framework of the previous sections to the specific case of symplectic balls in $\cp^2$. 

Any diffeomorphism $\phi\in \Diff_{[\B c]}(\widetilde{M}_n,\Sigma)$ fixes the classes $E_1,\ldots,E_n$, and $\PD [\widetilde{\om}_{\B c}]=L-\sum_i c_iE_i$. It follows that the group $\Diff_{[\B c]}(\widetilde{M}_n,\Sigma)$ is equal to the group $\Diff_{h}(\widetilde{M}_n,\Sigma)$ of all diffeomorphisms acting trivially on homology and leaving the exceptional divisor $\Sigma$ invariant. In particular, it is independent of the capacities~$\B c$. In order to simplify the notation, we will write
\[
\mD_h:=\Diff_{h}(\widetilde{M}_n),\quad\mD_h(\Sigma):=\Diff_{h}(\widetilde{M}_n,\Sigma),\quad\mG_{\B c}:=\Symp(\widetilde{M_{\B c}}),\quad\text{and}\quad\mG_{\B c}(\Sigma):=\Symp(\widetilde{M}_{\B c},\Sigma).
\]
\index{D@$\mD_h=\Diff_h(\widetilde{M}_n)$ -- group of diffeomorphisms of $\widetilde{M}_n$ acting trivially on homology}
\index{D@$\mD_h (\Sigma)=\Diff_h(\widetilde{M}_n,\Sigma)$ -- group of diffeomorphisms of $\widetilde{M}_n$ acting trivially on homology leaving the exceptional divisor $\Sigma$ invariant}
\index{G@$\mG_{\B c}:=\Symp(\widetilde{M_{\B c}})$}
\index{G@$\mG_{\B c}(\Sigma)$ -- stabilizer of the exceptional divisor $\Sigma$, $\Symp(\widetilde{M}_{\B c},\Sigma)$}
Stability components in the set $\mC(n):=\mC(\cp^2,\om_{FS},n)$
\index{C@$\mC(n)$ -- admissible capacities  $\mC(\cp^2,\om_{FS},n)$}
of admissible capacities can be explicitly described for embeddings of up to 8 balls in $\cp^2$. Given $\B c\in\mC(n)$, let $\mS_{\B c}^{\leq-1}(\Sigma)\subset H_2(M_{n},\Z)$ denote the set of homology classes of embedded $\widetilde{\om}_{\B c}$-symplectic spheres of self-in\-ter\-sec\-tion $\leq -1$ that intersect non-negatively with the exceptional classes $E_1,\ldots,E_{n}$. As explained in~\cite[Proposition~2.14(ii) and Lemma~3.2]{ALLP23}, for any $1\leq n\leq 8$ and any $\B c\in\mC(n)$ the set $\mS_{\B c}^{\leq-1}(\Sigma)$ is finite. Let $\mS_{n}^{\leq-1}(\Sigma)$ be the union of the sets $\mS_{\B c}^{\leq-1}(\Sigma)$ over all capacities $\B c\in\mC(n)$. 
To each class $A\in \mS_{n}^{\leq-1}(\Sigma)$ correspond a linear functional $H^2(M_{n},\R)\to\R$ and an associated map $\ell_A:\mC(n)\to\R$ defined by setting $\ell_A(\B c):=\langle [\widetilde{\w}_{\B c}],\, A\rangle$. The wall corresponding to $A\in\mS_{n}^{\leq-1}(\Sigma)$ is the set of capacities $\B c\in\mC(n)$ for which $\ell_A(\B c)=0$. It follows from~\cite[Corollary 3.3]{ALLP23} that the set of walls is locally finite, that is, given any $\B c\in\mC(n)$, $0\leq n\leq8$, there exists an open neighbourhood $U\subset\mC(n)$ that meets at most finitely many walls.  

\begin{theorem}[{Stability, \cite[Theorem 1.3]{ALLP23}}]\label{thm:StabilityOfEmbeddings}
For each integer $1\leq n\leq 8$, the set $\mC(n)$ of admissible capacities admits a partition into convex regions, called \emph{stability chambers}, such that:
\begin{enumerate}[label=(\roman*)]
\item each chamber is a convex polyhedron characterized by the signs of the functionals $\ell_A$, $A\in \mS_{n}^{\leq-1}(\Sigma)$;  
\item if two sets of capacities $\B c$ and $\B {c'}$ belong to the same stability chamber, then we have equality $\mA_{\B c}(\Sigma)=\mA_{\B{c'}}(\Sigma)$;
\item if two sets of capacities $\B c$ and $\B {c'}$ belong to the same stability chamber, then the embedding spaces $\Emb_{n}(\B c, \cp^2)$ and $\Emb_{n}(\B {c'}, \cp^2)$ are homotopy equivalent.
\end{enumerate}
\end{theorem}

It follows that to describe the stability chambers for embeddings of balls in $\cp^2$, it suffices to describe the sets $\mS_{n}^{\leq-1}(\Sigma)$.
\begin{proposition}[{\cite[Proposition 4.6]{Zhang17}}]\label{prop:ZhangPossibleNegativeCurves}
Let $J$ be a tamed almost complex structure on $M_{n}=M\#\,n\overline{\cp}\,\!^2$ with $n\leq 8$, and let $C=aL-\sum r_i E_i$ be an irreducible curve with $C\cdot C\leq-1$ and $a>0$. Then the homology class $[C]$ is one of the following:
\begin{enumerate}
\item $L-\sum E_{i_j}$,
\item $2L-\sum E_{i_j}$,
\item $3L-2E_m-\sum_{i_j\neq m} E_{i_j}$,
\item $4L-2E_{m_1}-2E_{m_2}-2E_{m_3}-\sum_{i_j\neq m_i} E_{i_j}$,
\item $5L-E_{m_1}-E_{m_2}-\sum_{i_j\neq m_i} 2E_{i_j}$,
\item $6L-3E_{m}-\sum_{i_j\neq m} 2E_{i_j}$.
\end{enumerate}
\end{proposition}

\begin{corollary}\label{cor:NegativeCurves}
Given a tamed almost-complex structure $J$ on the symplectic blow-up of $\cp^2$ at~$n$ disjoint balls of capacities $c_1,\ldots,c_n$, $0\leq n\leq 8$, an embedded $J$-holomorphic sphere of self-intersection $C\cdot C\leq -2$ that intersects each of the exceptional classes $E_1,\ldots,E_{n}$ non-negatively must represent one of the classes listed in Proposition~\ref{prop:ZhangPossibleNegativeCurves}.
\end{corollary}
\begin{proof}
Let $[C]=aL-\sum r_i E_i$. Then $0\leq E_i\cdot C = r_i$. Since any $J$-holomorphic representative of $[C]$ must have positive symplectic area, and since $\langle[\widetilde{\om}_{\B c}],\,[C]\rangle = a-\sum c_i r_i$, the coefficient $a$ must be strictly positive.
\end{proof}

\begin{theorem}[Stability chambers]\label{thm:ClassesDefiningStabilityChambers}
For each integer $1\leq n\leq 8$, the stability chambers of the set $\mC(n)$ of admissible capacities are the convex polyhedral regions defined by the linear functionals $\ell_A$ where $A$ is one of the homology classes of self-intersection $A\cdot A\leq -2$ listed in Proposition~\ref{prop:ZhangPossibleNegativeCurves}. 
\end{theorem}
\begin{proof}
Exceptional classes in $\mS_{n}^{\leq-1}(\Sigma)$ define the boundary of the set $\mC(n)$ of admissible classes while classes of self-intersection $\leq -2$ define  interior walls in $\mC(n)$. The proposition then follows directly from Theorem~\ref{thm:StabilityOfEmbeddings} and Corollary~\ref{cor:NegativeCurves}.
\end{proof}

\begin{remark}
Using the adjunction inequality for immersed curves, it is easy to see that the classes listed in Proposition~\ref{prop:ZhangPossibleNegativeCurves}  are precisely the spherical classes represented by blow-ups of immersed holomorphic spheres of degree $d\leq 6$ in $\cp^2$. This supports Conjecture~\ref{conjecture}.
\end{remark}

\subsubsection{Stability chambers for $n=1$ or $n=2$ balls in $\cp^2$}\label{section:stability chambers n=1,2}
For $n=1$, the space of admissible capacities is the interval $(0,1)$, while for $n=2$, it is the polygon $0<c_2\leq c_1<c_1+c_2<1$. Since none of the classes in Proposition~\ref{prop:ZhangPossibleNegativeCurves} have self-intersection less than or equal to $-2$ when $n\leq 2$, the entire space of admissible capacities is itself a stability chamber. 

\subsubsection{Stability chambers for $n=3$ balls in $\cp^2$}\label{section:stability chambers n=3}
The space of admissible capacities consists of triples $\B c = (c_1, c_2, c_3)$ satisfying $0<c_3\leq c_2 \leq c_1<c_1+c_2\leq 1$. The $-2$ class $A=L-E_1-E_2-E_3$ is the only negative homology class of self-intersection $A\cdot A\leq-2$ contained in the list of Proposition~\ref{prop:ZhangPossibleNegativeCurves}. The linear functional $\ell_A$ separates the space of admissible capacities into two chambers, namely $c_1+c_2+c_3<1$ and $c_1+c_2+c_3\geq 1$. 

\subsubsection{Stability chambers for $n=4$ balls in $\cp^2$}\label{section:stability chambers n=4}
For $n=4$, the chambers are defined by the $5$ classes $L-E_i-E_j-E_k$ and $L-E_1-E_2-E_3-E_4$. Because of the normalization $0<c_4\leq\cdots\leq c_1<1$, the symplectic areas of these classes are linearly ordered:
\[1-c_1-c_2-c_3-c_4<1-c_1-c_2-c_3\leq 1-c_1-c_2-c_4\leq 1-c_1-c_3-c_4\leq 1-c_2-c_3-c_4\]
Consequently, there are exactly $6$ stability chambers that we label and order accordingly:
\[C_5 \prec C_4 \prec C_3 \prec C_2\prec C_1\prec C_0\]
Note that these chambers correspond to the conditions on the capacities that are listed in Theorem~\ref{homotopy type 4balls} and Theorem~\ref{cohomology ring 4 balls}, with $C_0$ being the chamber in which $c_2+c_3+c_4\geq 1$, and $C_5$ being the chamber for which $c_1+c_2+c_3+c_4<1$. From now on, in order to simplify notation, we will write $\B c \prec \B{c'}$ whenever $\B c$ belongs to a chamber that preceeds the chamber of $\B{c'}$.

\subsubsection{Stability chambers for $n=5$ balls in $\cp^2$}\label{chambers5}
The space of admissible capacities consists of $5$-tuples $\B c = (c_1, \ldots, c_5)$ satisfying $0<c_5\leq \cdots \leq c_1<c_1+c_2\leq 1$ and $\sum c_i<2$. There are 33 chambers defined by the $16$ classes $L-E_{i_1}-E_{i_2}-E_{i_3}$, $L-E_{i_1}-E_{i_2}-E_{i_3}-E_{i_4}$, and $L-E_1-E_2-E_3-E_4-E_5$. This time, the symplectic areas of these classes cannot be linearly ordered, so that the cell decomposition of the space of admissible capacities has a more complicated structure.

\subsection{Stratifications of \texorpdfstring{$\Conf_{n}(\cp^2)$, $\mJ_{\B c}(\Sigma)$, $\mJ_{\B c}([\Sigma])$, $\mA(n,\Sigma)$, and $\mA(n,[\Sigma])$}{Conf, J, A}}\label{section: stratifications Conf A J}

\subsubsection{Stratification of configuration spaces}\label{section:stratification of Conf}
Consider $\cp^{2}$ with its usual complex structure. For $n\leq 4$, the configuration space $\Conf_n(\cp^2)$ decomposes as a disjoint union of strata dtermined by genericity conditions. For $n=1$ and $n=2$, there is only one stratum as any two points lie on a single line. As explained in the Introduction, $\Conf_{3}(\cp^{2})$ decomposes into two disjoint strata
\[\Conf_{3}(\cp^{2}) = F_{0}\sqcup F_{123}\]
while $\Conf_{4}(\cp^{2})$ decomposes into six strata
\[\Conf_{4}(\cp^{2}) =F_{0}\sqcup F_{234}\sqcup F_{134}\sqcup F_{124}\sqcup F_{123}\sqcup F_{1234}\]
where
\begin{align*}
F_{0} & = \{P\in \Conf_{4}(\cp^{2})~|~\text{no three points of $P$ are colinear}\}\\
F_{ijk} & = \{P\in \Conf_{4}(\cp^{2})~|~\text{only the points~}p_{i},p_{j},p_{k}\text{~are colinear}\}\\
F_{1234} & = \{P\in \Conf_{4}(\cp^{2})~|~\text{the four points of $P$ are colinear}\}.
\end{align*}
In the case $n=3$, we define $F_1:=F_0\sqcup F_{123}$ and for $n=4$, we recursively define 
\[F_{1}=F_{0}\sqcup F_{234},~F_{2}=F_{1}\sqcup F_{134},~F_{3}=F_{2}\sqcup F_{124}, ~F_{4}=F_{3}\sqcup F_{123}, \text{~and~}F_5=\Conf_4(\cp^2)\]
so that we have inclusions $F_0\subset F_1\subset\cdots\subset F_5$.\label{Fi}

\subsubsection{Stratification of spaces of almost complex structures}\label{section:stratification of mA}
Let $\mJ(\widetilde{\om}_{\B c})$ 
\index{J@$\mJ({\om})$ -- space of all compatible almost complex structures on $(M, \om)$}
be the space of all compatible almost complex structures on $\widetilde{M_{\B c}}$.
Recall that given an admissible set of capacities $\B c=(c_1,\ldots,c_n)$, we defined in Sections~\ref{section General Setting} and \ref{section:stability}, respectively, the spaces
\[\mJ_{\B c}(\Sigma)=\{J\in\mJ(\widetilde{\om}_{\B c})~|~ \text{the submanifold~} \Sigma \text{~is holomorphic}\}\]
\[\mA_{\B c}(\Sigma) = \left\{J\text{~is compatible with some~}\om'\in\Omega_{\B c}(\Sigma)\text{~and~}\Sigma\text{~is~}J\text{-holomorphic}\right\}.\]

In order to compare the $\PGL(3)$ action on $\Conf_n(\cp^2)$ with the action of $\mD_h(\Sigma)$ on $\mA_{\B c}(\Sigma)$, it will be convenient to consider larger spaces of almost-complex structures by relaxing the conditions on the $J$-holomorphic representatives of the classes $[\Sigma]$. We define:
\begin{align*}
\mJ_{\B c}([\Sigma]) & :=\{J\in\mJ(\widetilde{\om}_{\B c})~|~ \text{the classes~}[\Sigma]=\{E_1,\ldots, E_n\} \text{~have~} J\text{-holomorphic representatives}\},\\
\mA_{\B c} & := \left\{J\text{~is compatible with at least one symplectic form~}\om'\in\Omega_{\B c}\right\},\\
\mA_{\B c}([\Sigma]) & :=\{J\in \mA_{\B c}~|~ \text{the classes~}[\Sigma]=\{E_1,\ldots, E_n\} \text{~have~} J\text{-holomorphic representatives}\}.
\end{align*}
\index{J@$\mJ_{\B c}([\Sigma])$ -- space of compatible almost complex structures on $\widetilde{M_{\B c}}$ such that the classes $[\Sigma]$ have $J$-holomorphic representatives}
\index{A@$\mA_{\B c}$ -- space of compatible almost complex structures on $\widetilde{M_{\B c}}$ with at least one symplectic form in $\Omega_{\B c}$}
\index{A@$\mA_{\B c}([\Sigma])$ -- space of almost complex structures in $\mA_{\B c}$ such that classes $[\Sigma]$ have $J$-holomorphic representatives}
Taking the union over all admissible capacities, we also define
\[
\mA(n,\Sigma) := \bigcup_{\B c\in\mC(n)} \mA_{\B c}(\Sigma), \quad
\mA(n,[\Sigma]) := \bigcup_{\B c\in\mC(n)} \mA_{\B c}([\Sigma]), \quad \text{and}\quad
\mA(n) := \bigcup_{\B c\in\mC(n)} \mA_{\B c}.
\]
\index{Ab@$\mA(n,\Sigma) := \bigcup_{\B c\in\mC(n)} \mA_{\B c}(\Sigma)$}
\index{Ab@$\mA(n,[\Sigma]) := \bigcup_{\B c\in\mC(n)} \mA_{\B c}([\Sigma])$}
\index{Ab@$\mA(n) := \bigcup_{\B c\in\mC(n)} \mA_{\B c}$}
For $n\leq 4$, these spaces decompose into disjoint strata characterized by the existence of curves representing the classes of self-intersection $\leq -2$ listed in Proposition~\ref{prop:ZhangPossibleNegativeCurves}, namely, $L_{ijk}:=L-E_i-E_j-E_k$ and $L_{1234}:=L-E_1-E_2-E_3-E_4$. For $n=1$ and $n=2$, there is only one stratum as no such classes exist. For instance, for $n=3$ we have
\[\mA(3,\Sigma) = \mA_{0}(\Sigma)\sqcup \mA_{123}(\Sigma),\]
where 
\begin{align*}
\mA_{0}(\Sigma) & = \{J\in \mA(3,\Sigma)~|~\text{there is no embedded $J$-holomorphic representative of $L_{123}$}\},\\
\mA_{123}(\Sigma) & = \{J\in \mA(3,\Sigma)~|~\text{$L_{123}$ is represented by an embedded $J$-holomorphic sphere}\}.
\end{align*}
Similarly, for $n=4$, we have
\[\mA(4,\Sigma) =\mA_{0}(\Sigma)\sqcup \mA_{234}(\Sigma)\sqcup \mA_{134}(\Sigma)\sqcup \mA_{124}(\Sigma)\sqcup \mA_{123}(\Sigma)\sqcup \mA_{1234}(\Sigma),\]
where
\begin{align*}
\mA_{0}(\Sigma) & = \{J\in \mA(4,\Sigma)~|~\text{$L_{123}$ and $L_{1234}$ have no embedded $J$-holomorphic representatives}\},\\
\mA_{ijk}(\Sigma) & = \{J\in \mA(4,\Sigma)~|~\text{$L_{ijk}$ is represented by an embedded $J$-holomorphic sphere}\},\\
\mA_{1234}(\Sigma) & = \{J\in \mA(4,\Sigma)~|~\text{$L_{1234}$ is represented by an embedded $J$-holomorphic sphere}\}.
\end{align*}
\index{A@$\mA_{i}(\Sigma), \mA_{i}([\Sigma])$ -- unions of strata in $\mA(n,\Sigma)$ }
In the case $n=3$, we further define $\mA_1(\Sigma):=\mA_0(\Sigma)\sqcup \mA_{123}(\Sigma)$, and for $n=4$, we recursively define
\begin{multline}\mA_{1}(\Sigma)=\mA_{0}(\Sigma)\sqcup \mA_{234}(\Sigma),~ \mA_{2}(\Sigma)=\mA_{1}(\Sigma)\sqcup \mA_{134}(\Sigma),~\mA_{3}(\Sigma)=\mA_{2}(\Sigma)\sqcup \mA_{124}(\Sigma),\\
~ \mA_{4}(\Sigma)=\mA_{3}(\Sigma)\sqcup \mA_{123}(\Sigma),\text{~and~} \mA_5=\mA(4,\Sigma)
\end{multline}
so that we have inclusions $\mA_0(\Sigma)\subset\mA_1(\Sigma)\subset\cdots\subset\mA_5(\Sigma)$.

The spaces $\mJ_{\B c}(\Sigma)$, $\mJ_{\B c}([\Sigma])$, and $\mA(n,[\Sigma])$ admit analogous stratifications that we index similarly. In particular, for any index $I$ we have inclusions
\[
\mJ_I(\Sigma)\subset\mJ_I([\Sigma])\subset\mA_I(\Sigma)\subset\mA_I([\Sigma])\subset \mA_I\subset\mA(n).
\] 
\index{J@$\mJ_I(\Sigma)\subset\mJ_I([\Sigma])\subset\mA_I(\Sigma)\subset\mA_I([\Sigma])\subset \mA_I\subset\mA(n)$ -- various strata of the spaces of almost complex structures}
Note that for the spaces $\mJ_{\B c}(\Sigma)$ and $\mJ_{\B c}([\Sigma])$ some strata may be empty as the symplectic area of the defining classes may not be strictly positive for some choices of capacities $\B c$.

\subsubsection{The equivalence between homotopy orbits}\label{section:equivalence homotopy orbits}
If $G$ is a topological group and $X$ is a $G$-space, we write 
\[
X_{hG} := \OP{E}G\times_G X
\]
for the homotopy orbit (i.e. the Borel construction) of the $G$-action. 
\index{h@$\heartsuit_{hG}$ -- homotopy orbit of $\heartsuit$ with respect to the action of $G$}
The complex automorphism group $\PGL(3)$ acts on $\Conf_{n}(\cp^{2})$ preserving its stratification. The diffeomorphism group $\mD_h(\Sigma)=\Diff_{h}(\widetilde{M}_n,\Sigma)$ acts on $\mA(n,\Sigma)$ while $\mD_h=\Diff_{h}(\widetilde{M}_n)$ acts on $\mA(n,[\Sigma])$. Similarly, the symplectomorphism group $\mG_{\B c}(\Sigma)=\Symp(\widetilde{M_{\B c}}, \Sigma)$ acts on $\mJ_{\B c}(\Sigma)$ while $\mG_{\B c}=\Symp(\widetilde{M_{\B c}})$ acts on $\mJ_{\B c}([\Sigma])$. Each of these actions preserves the corresponding stratification.

\begin{proposition}\label{prop:homotopy orbit symp}
For any admissible capacities $\B c$, the action of $\mG_{\B c}(\Sigma)$ on $\mJ_{\B c}(\Sigma)$ yields a homotopy decomposition of the symplectic stabilizer $\OP{B}\mG_{\B c}(\Sigma)$, that is, $\big(\mJ_{\B c}(\Sigma)\big)_{h\mG_{\B c}(\Sigma)} \simeq \OP{B}\Symp(\widetilde{M_{\B c}}, \Sigma)$.
\end{proposition}
\begin{proof}
This follows from the fact that the space $\mJ_{\B c}(\Sigma)$ is contractible.
\end{proof}

\begin{proposition}\label{prop:equivalence S and [S]}
Given a set of capacities $\B c$, let $I$ be a multi-index representing a union of non-empty strata in $\mJ_{\B c}(\Sigma)$, and let $\mJ_I(\Sigma)\into\mJ_I([\Sigma])$ be the corresponding inclusion. The canonical map between homotopy orbits
\[
\mJ_{I}(\Sigma)_{h\Symp(\widetilde{M}_{\B c}, \Sigma)}
\to 
\mJ_{I}([\Sigma])_{h\Symp(\widetilde{M}_{\B c})}
\]
is a weak equivalence.
\end{proposition}
\begin{proof}
Let $\mG_{\B c}(\Sigma)=\Symp_{h}(\widetilde{M}_n,\Sigma)$ and $\mG_{\B c}=\Symp_h(\widetilde{M}_{\B c})$. Consider the locally trivial fibration
\[\mJ_{I}(\Sigma)\to\mJ_{I}([\Sigma])\to \mC_{\B c}([\Sigma])\]
that assigns to $J\in\mJ_{I}([\Sigma])$ the unique $J$-holomorphic configuration of disjoint exceptional spheres in classes $E_1,\ldots,E_n$. Since $\mG_{\B c}$ acts transitively on $\mC_{\B c}([\Sigma])$ with stabilizer $\mG_{\B c}(\Sigma)$, taking homotopy orbits gives another fibration
\[\mJ_I(\Sigma)\to \mJ_{I}([\Sigma])_{h\mG_{\B c}}\to\mC_{\B c}([\Sigma])_{h\mG_{\B c}}\simeq \OP{B}\mG_{\B c}(\Sigma)\]
that extends to the longer Puppe sequence
\[\Omega\big( \mJ_{I}([\Sigma])_{h\mG_{\B c}} \big) \to
\mG_{\B c}(\Sigma)\to
\mJ_I(\Sigma)\to 
\mJ_{I}([\Sigma])_{h\mG_{\B c}}\to
\OP{B}\mG_{\B c}(\Sigma).\]
Since the homotopy fiber of the evaluation map $\mG_{\B c}(\Sigma)\to\mJ_I(\Sigma)$ is $\Omega\big(\mJ_I(\Sigma)_{h\mG_{\B c}(\Sigma)}\big)$, we conclude that $\mJ_{I}(\Sigma)_{h\mG_{\B c}(\Sigma)}
\simeq
\mJ_{I}([\Sigma])_{h\mG_{\B c}}$.
\end{proof}

\begin{proposition}\label{prop:equivalence G and D actions}
Given a set of capacities $\B c$, let $I$ be a multi-index representing a union of non-empty strata in $\mJ_{\B c}(\Sigma)$, and let $\mJ_I(\Sigma)\into\mA_I(\Sigma)\into\mA_I([\Sigma])$ be the corresponding inclusions. The canonical maps between homotopy orbits
\[
\mJ_{I}(\Sigma)_{h\Symp(\widetilde{M}_{\B c}, \Sigma)}
\to 
\mA_{I}(\Sigma)_{h\Diff(\widetilde{M}_{n}, \Sigma)}
\quad
\text{~and~}
\quad
\mA_{I}(\Sigma)_{h\Diff(\widetilde{M}_{n}, \Sigma)}
\to 
\mA_{I}([\Sigma])_{h\Diff_h(\widetilde{M}_{n})}
\]
are weak equivalences.
\end{proposition}
\begin{proof}
Let $\mD_h(\Sigma)=\Diff_{h}(\widetilde{M}_n,\Sigma)$ and $\mG_{\B c}(\Sigma)=\Symp(\widetilde{M}_{\B c}, \Sigma)$. Consider the space of pairs
\[P_I(\Sigma)=\{(\om,J)\in P_{\B c}(\Sigma)~|~\om\in\Omega_{\B c}(\Sigma),~J\in \mJ_I(\Sigma)\}.\]
The projection onto the first factor is a $\mD_h(\Sigma)$-equivariant fibration 
\[\mJ_I(\Sigma)\to P_I(\Sigma)\to\Omega_{\B c}(\Sigma).\]
Since $\mD_h(\Sigma)$ acts transitively on $\Omega_{\B c}(\Sigma)$, taking homotopy orbits gives another fibration
\[\mJ_I(\Sigma)\to P_I(\Sigma)_{h\mD_h(\Sigma)}\to\Omega_{\B c}(\Sigma)_{h\mD_h(\Sigma)}\simeq \BSymp(\widetilde{M}_{\B c}, \Sigma)\]
that extends to the longer Puppe sequence
\[\Omega\big(P_I(\Sigma)_{h\mD_h(\Sigma)}\big) \to\Symp(\widetilde{M}_{\B c}, \Sigma)\to\mJ_I(\Sigma)\to P_I(\Sigma)_{h\mD_h(\Sigma)}\to\BSymp(\widetilde{M}_{\B c}, \Sigma).\]
Since the homotopy fiber of the evaluation map $\Symp(\widetilde{M}_{\B c}, \Sigma)\to\mJ_I(\Sigma)$ is $\Omega\big(\mJ_I(\Sigma)_{h\mG_{\B c}}\big)$, we conclude that $\mJ_I(\Sigma)_{h\mG_{\B c}}\simeq P_I(\Sigma)_{h\mD_h(\Sigma)}$. Finally, since the projection $P_I(\Sigma)\to\mA_I(\Sigma)$ is a $\mD_h(\Sigma)$-equivariant fibration with convex fibers, we also have the equivalence $P_I(\Sigma)_{h\mD_h(\Sigma)}\simeq \mA_I(\Sigma)_{h\mD_h(\Sigma)}$. This establishes the first homotopy equivalence.

For the second equivalence, let $\mC([\Sigma])$
\index{C@$\mC([\Sigma])$ -- space of all configurations of $n$ disjoint, embedded, symplectic spheres representing the exceptional  classes $[\Sigma]$}
be the space of all configurations of $n$ disjoint, embedded, symplectic spheres representing the exceptional  classes $[\Sigma]$. There is a fibration
\[
\mA_I(\Sigma)\to\mA_I([\Sigma])\to\mC([\Sigma])
\]
The group $\mD_h$ acts transitively on $\mC([\Sigma])$ with stabilizer $\mD_h(\Sigma)$. Taking the homotopy orbits, we get another fibration
\[
\mA_I(\Sigma)\to \mA_I([\Sigma])_{h\mD_h}\to \OP{B}\mD_h(\Sigma)
\]
that extends to the longer Puppe sequence
\[\Omega\big(\mA_I([\Sigma])_{h\mD_h}\big) \to\mD_h(\Sigma)\to\mA_I(\Sigma)\to \mA_I([\Sigma])_{h\mD_h}\to \OP{B}\mD_h(\Sigma).\]
We conclude that $\mA_I([\Sigma])_{h\mD_h}\simeq \mA_I(\Sigma)_{h\mD(\Sigma)}$.
\end{proof}

\subsection{The main geometric result and a proof of Theorems \ref{3balls} and \ref{homotopy type 4balls}}\label{SS:main-geometric}
As explained in the Introduction, for $n\leq 4$, the equivalences between spaces of configurations and spaces of embeddings, and in particular Theorem~\ref{3balls} and Theorem~\ref{homotopy type 4balls}, are consequences of Theorem~\ref{claim introduction} that we now reformulate in a more precise way.
\begin{theorem}\label{claim}
The action of $\PGL(3)$ on the subspace $F_i\subset\Conf_n(\cp^2)$ and the action of $\mD_h$ on the subspace $\mA_i([\Sigma])\subset\mA(n,[\Sigma])$ have homotopy equivalent homotopy orbits. 
\end{theorem}

The proof of Theorem~\ref{claim} and of its corollaries are broken into several steps. The analysis of the $\PGL(3)$ action on $\Conf(\B c,\cp^2)$ is done in Section~\ref{section: homotopy orbits PGL}, while the $\mG_{\B c}$ action on $\mJ_{\B c}([\Sigma])$ is described in Section~\ref{section:Actions on ACS}. This is used in Section~\ref{section: homotopy orbits Diff} to describe the homotopy orbit of $\mD_h$ acting on $\mA(n,[\Sigma])$. These results are put together in Section~\ref{section:proof of main thm} where the proof of Theorem~\ref{claim} is given.

\begin{proof}[Proof of Theorem \ref{3balls} and Theorem \ref{homotopy type 4balls}]
It follows from \eqref{second fibration unparametrized} there is a homotopy fibration
\[
\IEmb_n(\B c,\cp^2)\to \OP{B}\Symp(\widetilde{M}_{\B c},\Sigma) \to \OP{B}\Symp(\cp^2).
\]
Moreover, we have the following chain of homotopy equivalences for the set of
capacities $\B c$ belonging to a single chamber $C_i$:
\begin{align*}
\OP{B}\Symp(\widetilde{M}_{\B c},\Sigma) 
&\simeq (\mJ_{\B c}(\Sigma))_{h\mG_{\B c}(\Sigma)} &\text{by Proposition \ref{prop:homotopy orbit symp}}\\
&\simeq (\mJ_{\B c}([\Sigma]))_{h\mG_{\B c}} &\text{by Proposition \ref{prop:equivalence S and [S]}}\\
&\simeq (\mA_{\B c}([\Sigma]))_{h\mD_h} &\text{by Proposition \ref{prop:equivalence G and D actions}}\\
&\simeq (\mA_{i}([\Sigma]))_{h\mD_h} &\text{by Theorem \ref{thm:StabilityOfEmbeddings}}\\
&\simeq (F_i)_{h\PGL} &\text{by Theorem \ref{claim}}
\end{align*}
as well as the equivalence $\OP{B}\Symp(\cp^2)\simeq \OP{B}\PGL$. This implies that
$\IEmb_n(\B c,\cp^2)$ is homotopy equivalent to the fiber of the fibration
\[
F_i \to (F_i)_{h\PGL} \to \OP{B}\PGL,
\]
where $F_i$ are defined on page \pageref{Fi} and in the statement of Theorem \ref{homotopy type 4balls}.
\end{proof}

\section{Homotopy orbits of the \texorpdfstring{$\PGL(3)$}{PGL(3)} action on \texorpdfstring{$\Conf_n(\cp^2)$}{Conf(n,CP2)}}\label{section: homotopy orbits PGL}

\subsection{The cases \texorpdfstring{$n=1$}{n=1} and \texorpdfstring{$n=2$}{n=2}}
The group $\PGL(3)$ acts transitively on $\Conf_n(\cp^2)$ for $n=1,2$ so that we have fibrations
\[\U(2)\simeq \Stab(p_1)\to\PGL(3)\to\cp^2\quad\text{and}\quad\B{T}^2\simeq\Stab\big((p_1,p_2)\big)\to\PGL(3)\to\Conf_2(\cp^2)\]
and homotopy orbits
\[\big(\Conf_1(\cp^2)\big)_{h\PGL(3)}\simeq \OP{B}\U(2)\quad\text{and}\quad\big(\Conf_2(\cp^2)\big)_{h\PGL(3)}\simeq \OP{B}\B{T}^2.\]

\subsection{The case \texorpdfstring{$n=3$}{n=3}}
The stratum $F_{0}$ is open and dense in $\Conf_{3}(\cp^{2})$, while $F_{123}$ has codimension $2$. The group $\PGL(3)$ acts transitively on each stratum with stabilisers homotopy equivalent to K\"ahler actions $\B T^2$ and $\B S^1$, respectively, so that $F_0 \simeq \PGL(3)/\B T^2 $ and $F_{123}\simeq \PGL(3)/\B S^1$. The isotropy representation of $\B S^{1}$ on the normal bundle of $F_{123}$ is the standard representation of weight $1$. It follows that there is an invariant neighborhood $\mN_{123}$
\index{N@$\mN_I$ -- invariant neighborhood of $F_I$}
of $F_{123}$ of the form $\PGL(3)\times_{\B S^{1}}\C$. Applying the Borel construction to the orbit diagram, we obtain
\[
\begin{tikzcd}
\mN_{123}\setminus F_{123} \arrow{r}{} \arrow{d}{} & \mN_{123} \arrow{d}{} \\
F_{0} \arrow{r}{} & \Conf_{3}(\cp^{2})
\end{tikzcd}
\implies
\begin{tikzcd}
* \arrow{r}{} \arrow{d}{} & \OP{B}\B S^{1}_{123} \arrow{d}{} \\
 \OP{B} \B T^2 \arrow{r}{} & \OP{B}\B S^{1} \vee B \B T^2 \simeq\big(\Conf_3(\cp^2)\big)_{h\PGL(3)}.
\end{tikzcd}
\]

\subsection{The case \texorpdfstring{$n=4$}{n=4}}
The stratum $F_{0}$ is open and dense in $\Conf_{4}(\cp^{2})$, while $F_{ijk}$ and $F_{1234}$ are, respectively, of codimensions $2$ and $4$. The action of $\PGL(3)$ on $F_{0}$ is transitive with trivial stabilizers, so that $F_{0}\simeq \PGL(3)$. Similarly,  $\PGL(3)$ acts transitively on $F_{ijk}$ with stabilizer $\B S^{1}_{ijk}$, so that $F_{ijk}\simeq \PGL(3)/\B S^{1}_{ijk}$. The isotropy representation of $\B S^{1}_{ijk}$ on the normal bundle of $F_{ijk}$ is the standard representation (of weight $1$) and there is an invariant neighborhood $\mN_{ijk}$ of $F_{ijk}$ in $F\setminus F_{1234}$ of the form $\PGL(3)\times_{\B S^{1}_{ijk}}\C$. Recall that we defined $F_{1}=F_{0}\sqcup F_{234}, \ldots, F_5=\Conf_4(\cp^2)$. Applying inductively the Borel construction to the associated diagrams, we obtain
\begin{equation}\label{eq:pushout diagram F-1}
\begin{tikzcd}
\mN_{ijk}\setminus F_{ijk} \arrow{r}{} \arrow{d}{} & \mN_{ikj} \arrow{d}{} \\
F_{r-1} \arrow{r}{} & F_r
\end{tikzcd}
\quad\implies\quad
\begin{tikzcd}
* \arrow{r}{} \arrow{d}{} & \OP{B}\B S^{1}_{ijk} \arrow{d}{} \\
P_{r-1} \arrow{r}{} & P_r := \big(F_{r}\big)_{h\PGL(3)}
\end{tikzcd}
\end{equation}
so that $\big(F_{r}\big)_{h\PGL(3)}\simeq \OP{B}\B S^1 \vee\cdots\vee \OP{B}\B S^1$ $r$ times, $1\leq r\leq 4$. 

We now consider the pushout diagrams obtained by adding the last stratum $F_{1234}$
\[
\begin{tikzcd}[cramped, column sep=tiny]
\mN_{1234}\setminus F_{1234} \arrow{r}{} \arrow{d}{} & \mN_{1234} \arrow{d}{} \\
F_{4} \arrow{r}{} & \Conf_4(\cp^2)
\end{tikzcd}
~\implies~
\begin{tikzcd}[column sep=tiny]
\left(\mN_{1234}\setminus F_{1234}\right)_{h\PGL(3)} \arrow{r}{} \arrow{d}{} & \left(\mN_{1234}\right)_{h\PGL(3)} \arrow{d}{} \\
P_{4} \arrow{r}{} & P_5:=\big(\Conf_4(\cp^2)\big)_{h\PGL(3)}
\end{tikzcd}
\]
Recall that the cross-ratio
\[\chi(z_{1},z_{2},z_{3},z_{4}):=(z_{1},z_{2}:z_{3},z_{4})=\frac{(z_{3}-z_{1})(z_{4}-z_{2})}{(z_{3}-z_{2})(z_{4}-z_{1})}\]
is a complete invariant of the $\PSL(2,\C)$ action on the configuration space $\Conf_{4}(\cp^1)$ of $4$ distinct ordered points on the complex line. It follows that the moduli space is
\[\mM_{0,4}:=\Conf_4(\cp^1)/\PSL(2)=\C\setminus\{0,1\}.\] 
The action of $\PGL(3)$ on $F_{1234}$ is not transitive as is preserves the cross-ratio of any $4$ distinct points on the same line. However, $\PGL(3)$ does act transitively on the level sets of the cross-ratio with stabilizers $H=\C^{2}\ltimes\C^{*}\simeq \B S^{1}$. Consequently, there is a fibration
\[\PGL(3)/H\to F_{1234}\xrightarrow{\chi} \mM_{0,4}\]
in which the projection $\chi$ is $\PGL(3)$ equivariant (with respect to the trivial action of $\PGL(3)$ on $\mM_{0,4}$). Choosing any parametrized line $u:\cp^{1}\to\cp^{2}$ yields a global equivariant section 
\[s_{u}(w)=(u(w),u(0),u(1),u(\infty)).\]
Let $\mM(L)$ be the moduli space of parametrized lines in $\cp^{2}$. The group $\PSL(3)$ acts transitively on $\mM(L)$ with stabilizer $H_L=\C^2\times \C^*\simeq\B S^1$. There is an equivariant diffeomorphism
\[\phi:F_{1234}\to\mM(L)\times \mM_{0,4}\]
given by $\phi(\B p)=(u_{\B p},\chi(\B p))$ where $u_{\B p}$ is the only curve sending $(0,1,\infty)$ to $(p_{2},p_{3},p_{4})$, and whose inverse is $\phi^{-1}(u,w)=s_{u}(w)$. Applying the Borel construction to $\mM(L)\times \mM_{0,4}$
gives 
\[(F_{1234})_{h\PGL(3)}\simeq \OP{B}H\times \mM_{0,4}\simeq \OP{B}\B S^{1}\times \mM_{0,4}.\]
Similarly, the isotropy representation on the normal bundle of any orbit in $F_{1234}$ splits as the sum of the trivial representation and two copies of the standard representation of weight $1$. Applying the Borel construction to 
\[S^{3}\to S\mN_{1234}\to F_{1234}\]
yields a fibration
\[S^{3}\to \left(\mN_{1234}\setminus F_{1234}\right)_{h\PGL(3)}\to (F_{1234})_{h\PGL(3)}\simeq \OP{B} \B S^{1}\times \mM_{0,4}\]
which is homotopy equivalent to the product bundle of the Hopf fibration with $\mM_{0,4}$,
\[S^{3}\times \{*\}\to S^{2}\times \mM_{0,4}\to \OP{B} \B S^{1}\times\mM_{0,4}.\]
We thus obtain a pushout diagram 
\begin{equation}\label{eq:pushout diagram F-2}
\begin{tikzcd}
S^{2}\times \mM_{0,4} \arrow{r}{} \arrow{d}{} & \OP{B} \B S^{1}\times\mM_{0,4} \arrow{d}{} \\
P_{4}\simeq \OP{B} \B S^{1}\vee \OP{B} \B S^{1}\vee \OP{B} \B S^{1}\vee \OP{B} \B S^{1} \arrow{r}{} & P_5:=\big(\Conf_4(\cp^2)\big)_{h\PGL(3)}.
\end{tikzcd}
\end{equation}

\section{Actions of symplectomorphism groups on almost complex structures}\label{section:Actions on ACS}

In this section, we study the action of symplectomorphisms on spaces of almost complex structures. In the cases $n\leq3$, these statements can also be deduced from scattered results found in~\cite{AbGrKi,AbMcD00,AP13,P08i}. This section provides a uniform treatment for all $n\leq 4$.

\subsection{Almost complex structures and configurations}\label{section:ACS and configurations}
We first recall some notation: $n\in\{1,2,3,4\}$, $\widetilde{M}_n$ is the smooth manifold underlying the complex blow-up of $\cp^2$ at $n$ distinct points, $\Sigma=\Sigma_1\sqcup\cdots\sqcup \Sigma_n\subset \widetilde{M}_n$ is the configuration of exceptional curves, $\B c=(c_1,\ldots,c_n)$ is an admissible set of capacities with $0<c_n\leq\cdots\leq c_1<1$,  $\widetilde{M}_{\B c}$ is the $n$-fold symplectic blow-up with capacities $\B c$, and   $\mJ(\widetilde{\omega}_{\B c})$ is the space of compatible almost complex structures on $\widetilde{M}_{\B c}$. 

Given a configuration of embedded symplectic spheres $S:=S_1\cup\cdots\cup S_m$ in $\widetilde{M}_{\B c}$, we set
\[
\mJ_{\B c}(S)=\{J\in\mJ(\widetilde{\omega}_{\B c})~|~\text{the configuration $S$ is $J$-holomorphic} \}.
\]
\index{J@$\mJ_{\B c}(S)$ -- space of compatible almost complex structures $J$ such that the configuration $S$ is $J$-holomorphic}
Similarly, given a set of spherical homology classes $A= \{A_1,\cdots,A_k\} \subset  H_2(\widetilde{M}_n;\Z)$, we define
\begin{multline*}
\mJ_{\B c}([A])=\{J\in\mJ(\widetilde{\omega}_{\B c})~|~\text{there exists a $J$-holomorphic configuration~} C:=C_1\cup\cdots\cup C_m \\ \text{~of embedded spheres in~} \widetilde{M}_{\B c}  \mbox{ such that } [C_i]= A_i\}.
\end{multline*}
\index{J@$\mJ_{\B c}([A])$ -- space of compatible almost complex structures $J$ such that there exists a $J$-holomorphic configuration  of embedded spheres with homological configuration $A$}
In particular, for a configuration $S$, writing $[S]:=\{[S_1],\cdots, [S_m]\}$ for the corresponding set of homology classes, we have
\begin{multline*}
\mJ_{\B c}([S])=\{J\in\mJ(\widetilde{\omega}_{\B c})~|~\text{there exists a $J$-holomorphic configuration~} C:=C_1\cup\cdots\cup C_m \\ \text{~of embedded spheres in~} \widetilde{M}_{\B c}  \mbox{ such that } [C_i]= [S_i]\}.
\end{multline*}
\index{J@$\mJ_{\B c}([S])$ -- space of compatible almost complex structures $J$ such that there exists a $J$-holomorphic configuration  of embedded spheres with homological configuration $[S]$}
Given a configuration of spheres $S\subset\widetilde{M}_{\B c}$ and an homological configuration $A\subset H_2(\widetilde{M}_n;\Z)$ as above, we define the space of almost complex structures realizing $A$ relative to $S$ as
\[\mJ_{\B c}([A],S) := \mJ_{\B c}([A])\cap\mJ_{\B c}(S).\]
\index{J@$\mJ_{\B c}([A],S) := \mJ_{\B c}([A])\cap\mJ_{\B c}(S)$}
By definition, the strata introduced in Section~\ref{section: stratifications Conf A J}, can be written as
\[
\mJ_{0}(\Sigma)=\mJ_{\B c}([\mE_n],\Sigma), \qquad 
\mJ_{ijk}(\Sigma)=\mJ_{\B c}([L_{ijk}],\Sigma), \qquad 
\mJ_{1234}(\Sigma)=\mJ_{\B c}([L_{1234}],\Sigma),
\]
where $\mE_n\subset H_2(\widetilde{M}_n;\Z)$ is the set of all symplectic exceptional classes.

The above spaces of almost complex structures are closely related to spaces of symplectic spherical configurations. Let $\mC_{\B c}([S])$
\index{C@$\mC_{\B c}([S])$ -- space of  configurations of embedded symplectic spheres in classes $[S]$ that are holomorphic for some $J\in \mJ_{\B c}([S])$}
be the space of  configurations of embedded symplectic spheres in classes $[S]$ that are simultaneously holomorphic for some $J\in \mJ_{\B c}([S])$, and let $\mC_{\B c}^\circ([S])$ 
\index{C@$\mC_{\B c}^\circ([S])$ -- subspace of $\mC_{\B c}([S])$ whose spheres intersect $\widetilde{\om}_{\B c}$-orthogonaly}
be the subspace of configurations whose spheres intersect $\widetilde{\om}_{\B c}$-orthogonaly. Similarly, we write $\mC_{\B c}([A],S)$ 
\index{C@$\mC_{\B c}([A],S)$ -- space of configurations of homological type $A\cup[S]$ that are holomorphic for some $J\in \mJ_{\B c}([A],S)$}
for the space of configurations of homological type $A\cup[S]$ that are holomorphic for some $J\in \mJ_{\B c}([A],S)$, and $\mC_{\B c}^\circ([A],S)$
\index{C@$\mC_{\B c}^\circ([A],S)$ -- subspace of $\mC_{\B c}([A],S)$ whose components intersect $\widetilde{\om}_{\B c}$-orthogonaly}
for the subspace of configurations whose components intersect $\widetilde{\om}_{\B c}$-orthogonaly.

\begin{proposition}\label{prop:Stratification mJ} 
The space $\mJ_{\B c}([\Sigma])$ is an open submanifold of $\mJ(\widetilde{\om}_{\B c})$. The stratum $\mJ_0([\Sigma])$ is open and dense in $\mJ_{\B c}([\Sigma])$. If non-empty, the stratum $\mJ_{ijk}([\Sigma])$ is a codimension $2$ submanifold. In the case $n=4$, the stratum $\mJ_{1234}([\Sigma])$ is a submanifold of codimension $4$ whenever nonempty.
\end{proposition}
\begin{proof} 

By the general theory of $J$-holomorphic curves in $4$-manifolds, the space $\mJ_{\B c}([\Sigma])$ is open and dense in $\mJ(\widetilde{\om}_{\B c})$. By ~\cite[Proposition B1]{AP13}, the subspace $\mJ_{\B c}([\Sigma]\cup [L_{ijk}])$ is a submanifold of $\mJ(\widetilde{\om}_{\B c})$ of codimension $2$ contained in the open set $\mJ_{\B c}([\Sigma])$. Likewise, in the case $n=4$, the stratum $\mJ_{\B c}([\Sigma]\cup [L_{1234}])$ is a codimension $4$ submanifold of $\mJ_{\B c}([\Sigma])$.
\end{proof}

\begin{corollary}\label{cor:the open stratum is connected}
The stratum $\mJ_0([\Sigma])$ is connected. \qed
\end{corollary}

Each stratum of $\mJ_{\B c}([\Sigma])$ can be characterized by the existence of extended $J$-holomorphic configurations whose orbits under the action of symplectomorphism groups can be described geometrically. 

\subsubsection{Case $n=1$} For any compatible $J$, the one-point blow-up of $\cp^2$ is a ruled surface in which the exceptional class $[\Sigma]$ is represented by a $J$-holomorphic section. Let $F=L-E_1$ be the homology class of a fiber. Pick a point $p\in\widetilde{M}_1$. The structures in the unique stratum $\mJ_0([\Sigma])$ are characterized by the existence of configurations consisting of a $J$-holomorphic curve in the class $[\Sigma]$ together with the unique $J$-holomorphic fiber passing through $p$. We say that these configurations are of homological type $\mT_0$.
\index{T@$\mT_{\bullet}$ -- homological type of configurations}
Note that this kind of configurations admits a toric model.

\subsubsection{Case $n=2$} The structures in the unique stratum $\mJ_0([\Sigma])$ are characterized by the existence of configurations consisting of representatives of the exceptional classes $[\Sigma]=[\Sigma_1]\sqcup[\Sigma_2]$ together with the $J$-holomorphic exceptional curve representing the class $L_{12}:=L-E_1-E_2$. We again say that these configurations are of type $\mT_0$. Again, that this kind of configurations admits a toric model.

\begin{figure}[H]
\centering
\begin{subfigure}{0.4\textwidth}
\centering{
\begin{tikzpicture}[scale=0.265]

\filldraw[fill=lightgray,opacity=0.5,draw=none] (0,8) -- (0,0) -- (13,0) -- (5,8) -- cycle;

\draw[thick,black] (0,8) -- (0,0) -- (13,0);
\draw[line width=0.5ex,black] (0,8) -- (5,8) -- (13,0);

\draw[dotted,thick] (0,8) -- (0,13) -- (5,8);

\draw[fill] (0,0) circle (1.5ex);
\node at (-0.7,-0.7) {$O$};

\node at (2.5,7.2) {$E_1$};
\node at (7,4.5) {$F$};
\end{tikzpicture}}

\caption{Moment image for~$n=1$.}
\label{fig:toric blow-up n=1}
\end{subfigure}
\hfill
\begin{subfigure}{0.4\textwidth}
\centering{
\begin{tikzpicture}[scale=0.265]

\filldraw[fill=lightgray,opacity=0.5,draw=none] (0,8) -- (0,0) -- (10,0) -- (10,3) -- (5,8) -- cycle;

\draw[thick,black] (0,8) -- (0,0) -- (10,0);
\draw[line width=0.5ex,black] (0,8) -- (5,8) -- (10,3) -- (10,0);

\draw[dotted,thick] (0,8) -- (0,13) -- (5,8);
\draw[dotted,thick] (10,0) -- (13,0) -- (10,3);

\draw[fill] (0,0) circle (1.5ex);
\node at (-0.7,-0.7) {$O$};

\node at (2.5,7.2) {$E_1$};
\node at (9,1.5) {$E_2$};
\node at (9,6) {$L_{12}$};
\end{tikzpicture}}

\caption{Moment image for~$n=2$.}
\label{fig:toric blow-up n=2}
\end{subfigure}
\caption{Toric configurations for $n=1$ and $n=2$ balls.}
\label{fig:figures n=1 and n=2}
\end{figure}
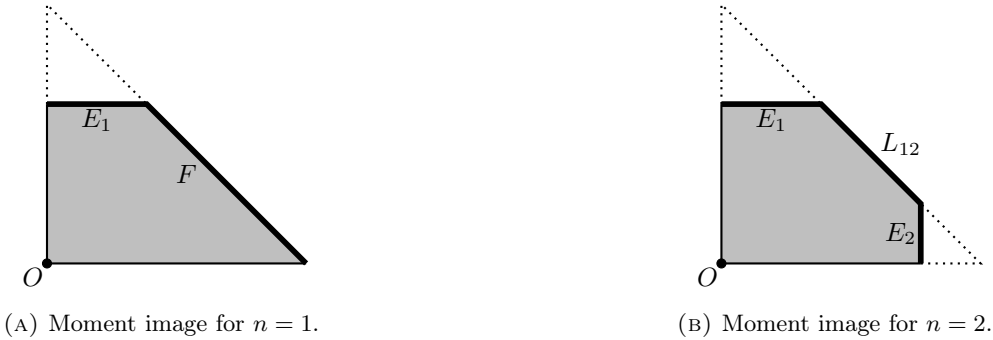

\subsubsection{Case $n=3$} An almost complex structure $J\in\mJ(\Sigma)$ belongs to the stratum $\mJ_0([\Sigma])$ if and only if none of the classes $L_{ijk}$ is represented by a $J$-holomorphic embedded sphere. Since each class $L_{123}$ intersects negatively with either $L_{12}=L-E_1-E_2$ or $L_{23}=L-E_2-E_3$, the existence of $J$-holomorphic spheres representing the classes $L_{12}$ and $L_{23}$ ensures that no $J$-holomorphic representatives of $L_{123}$ exist. Consequently, $J$-holomorphic configurations of type $\mT_{0}$ consisting of representatives of the exceptional classes $[\Sigma]$ together with representatives of the classes $L_{12}\cup L_{23}$ characterize the stratum $\mJ_0(\Sigma)$. Similarly, the stratum $\mJ_{123}$ is characterized by the existence of $J$-holomorphic configurations of type $\mT_{123}$ consisting of representatives of $[\Sigma]$ together with an embedded sphere in class $L_{123}$. For $n=3$, we thus write
\begin{align*}
\mJ_0([\Sigma])=\mJ_{\B c}(\mT_0) &= \mJ_{\B c}([L_{12}]\cup[L_{23}]\cup[\Sigma]) &\mJ_{123}([\Sigma])=\mJ_{\B c}(\mT_{123}) &= \mJ_{\B c}([L_{123}]\cup[\Sigma])\\
\mC_{\B c}(\mT_0) &= \mC_{\B c}([L_{12}]\cup[L_{23}]\cup[\Sigma]) &\mC_{\B c}(\mT_{123}) &= \mC_{\B c}([L_{123}]\cup[\Sigma])
\end{align*}
Note that a configuration of type $\mT_0$ can be realized as a union of invariant holomorphic spheres under a Hamiltonian K\"ahler action of $\B T^2$ on $(\widetilde{M}, \widetilde{\om}_{\B c})$, see Figure~\ref{fig:toric blow-up}. On the other hand, there are no toric configuration of type $\mT_{123}$. However, there exists an almost toric fibration and a configuration $S_{123}$ whose projection on the base diagram form a subset of the boundary, see Figure~\ref{fig:almost-toric blow-up n=3}.

\begin{figure}[H]
\centering
\begin{subfigure}{0.4\textwidth}
\centering{
\begin{tikzpicture}[scale=0.3]

\filldraw[fill=lightgray,opacity=0.5,draw=none] (0,4) -- (4,0) -- (9,0) -- (9,3) -- (2,9) -- (0,9)  -- cycle;

\draw[thick,black] (0,4) -- (4,0) -- (9,0) -- (9,3) -- (2,9) -- (0,9)  -- cycle;
\draw[line width=0.5ex,black] (0,4) -- (4,0) -- (9,0) -- (9,3) -- (2,9) -- (0,9);

\draw[dotted,thick] (0,4) -- (0,0) -- (4,0);

\draw[dotted,thick] (9,0) -- (12,0) -- (9,3);

\draw[dotted,thick] (2,9) -- (0,11) -- (0,9);

\draw[fill] (0,4) circle (1.5ex);
\node[left] at (0,4) {$O$};

\node at (2.5,2.5) {$E_1$};
\node at (8.25,1.5) {$E_2$};
\node at (1.2,8.3) {$E_3$};
\node[right] at (3.7,-1) {$L-E_1-E_2$};
\node[right] at (6,6) {$L-E_2-E_3$};
\end{tikzpicture}}

\caption{Moment image of a toric configuration~$S_0$.}
\label{fig:toric blow-up}
\end{subfigure}
\hfill
\begin{subfigure}{0.4\textwidth}
\centering{
\begin{tikzpicture}[scale=0.265]

\filldraw[fill=lightgray,opacity=0.5,draw=none] (0,8) -- (0,0) -- (10,0) -- (10,3) -- (9,4) -- (7,4) -- (7,6) -- (5,8) -- cycle;

\draw[thick,black] (0,8) -- (0,0) -- (10,0);
\draw[line width=0.5ex,black] (0,8) -- (5,8) -- (7,6);
\draw[line width=0.5ex,black] (9,4) -- (10,3) -- (10,0);

\draw[dotted,thick] (0,8) -- (0,13) -- (5,8);
\draw[dotted,thick] (10,0) -- (13,0) -- (10,3);

\draw[dotted,line width=0.5ex] (9,4) -- (7,4) -- (7,6);
\node at (7,4) {\(\times\)};

\draw[fill] (0,0) circle (1.5ex);
\node at (-0.7,-0.7) {$O$};

\node at (2.5,7.2) {$E_1$};
\node at (9,1.5) {$E_2$};
\node at (6.3,3.3) {$E_3$};
\end{tikzpicture}}

\caption{An almost toric construction of a configuration $S_{123}$}
\label{fig:almost-toric blow-up n=3}
\end{subfigure}
\caption{Configurations for $n=3$ balls.}
\label{fig:figures}
\end{figure}
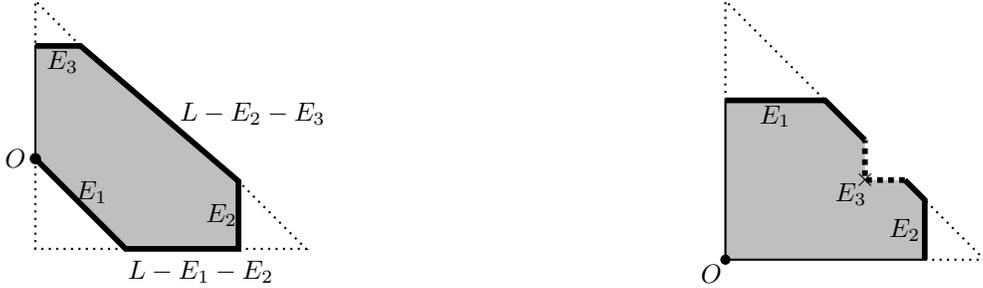

\subsubsection{Case $n=4$} An almost complex structure $J\in\mJ([\Sigma])$ belongs to the stratum $\mJ_0([\Sigma])$ if and only if none of the classes $L_{ijk}$ or $L_{1234}$ is represented by a $J$-holomorphic embedded sphere. Since these classes intersect negatively with either $L_{12}$ or $L_{34}$, the existence of $J$-holomorphic spheres representing the classes $L_{12}$ and $L_{34}$ ensures that no $J$-holomorphic representatives of $L_{ijk}$ or $L_{1234}$ exist. Consequently, configurations of $J$-holomorphic spheres of type $S_0:=[\Sigma] \cup [L_{12}]\cup [L_{34}]$ characterizes the almost complex structures in the stratum $\mJ_0([\Sigma])$, see Figure~\ref{fig:configuration S0 n=4}. Likewise, the strata $\mJ_{ijk}([\Sigma])$ are characterized by the existence of configurations of type $S_{ijk}=[\Sigma] \cup [L_{m\ell}] \cup [L_{ijk}]$, where $m$ is  chosen from $\{i,j,k\}$, as shown in Figure~\ref{fig:configuration Sijk n=4}. For technical reasons, it is convenient to take $m=1$ when $\ell\neq 1$, and $m=4$ when $\ell=1$. These configurations can be constructed from almost toric fibrations as shown in Figure~\ref{fig:almost-toric blow-up n=4}. Finally, the stratum $\mJ_{1234}([\Sigma])$ is characterized by the existence of $J$-holomorphic configurations of type $[\Sigma]\cup [L_{1234}]$ as shown in Figure~\ref{fig:configurationS1234} and Figure~\ref{fig:almost-toric blow-up S1234 n=4}. For $n=4$, we will thus set
\begin{align*}
\mJ_0([\Sigma])=\mJ_{\B c}(\mT_0) &= \mJ_{\B c}([L_{12}]\cup[L_{34}]\cup[\Sigma]) &\mJ_{234}([\Sigma])=\mJ_{\B c}(\mT_{234}) &= \mJ_{\B c}([L_{234}]\cup[L_{14}]\cup[\Sigma])\\
\mC_{\B c}(\mT_0) &= \mC_{\B c}([L_{12}]\cup[L_{34}]\cup[\Sigma]) &\mC_{\B c}(\mT_{234}) &= \mC_{\B c}([L_{234}]\cup[L_{14}]\cup[\Sigma])
\end{align*}
\begin{align*}
\mJ_{ijk}([\Sigma])=\mJ_{\B c}(\mT_{ijk}) &= \mJ_{\B c}([L_{ijk}]\cup[L_{1\ell}]\cup[\Sigma]) & \mJ_{1234}([\Sigma])=\mJ_{\B c}(\mT_{1234}) &= \mJ_{\B c}([L_{1234}]\cup[\Sigma])\\
\mC_{\B c}(\mT_{ijk}) &= \mC_{\B c}([L_{ijk}]\cup[L_{1\ell}]\cup[\Sigma]) &\mC_{\B c}(\mT_{1234}) &= \mC_{\B c}([L_{1234}]\cup[\Sigma])
\end{align*}
where $\{i,j,k\}\neq \{2,3,4\}$.

\begin{lemma}\label{lemma:CharaterizationStrataConfigurations}
Let $\mT$ be any of the above types of configurations characterizing a stratum. 
\begin{enumerate}
\item There is a natural homotopy equivalence $\mJ_{\B c}(\mT)\to \mC_{\B c}(\mT)$. 
\item The inclusion $\mC_{\B c}^\circ(\mT)\into \mC_{\B c}(\mT)$ is a homotopy equivalence.
\end{enumerate}
\end{lemma}
\begin{proof}
For $n=1$, there is a unique configuration of type $\mT_0$ for each $J\in\mA_0$. For $n\geq 2$, since the homological configuration types only contain classes of negative self-intersections, positivity of intersection implies that for $J\in \mJ_{\B c}(\mT)$, there is a unique $J$-holomorphic configuration of type $\mT$. We thus get a well-defined fibration $\mJ_{\B c}(\mT)\to \mC_{\B c}(\mT)$. Now, given a configuration, the set of all $J$ for which it is $J$-holomorphic is contractible. This proves the first statement. The second statement follows from the Gompf isotopy lemma, see~\cite[Lemma~5.3]{Evans}.
\end{proof}


\begin{figure}[H]
\begin{tikzpicture}[scale=0.1, roundnode/.style={circle, draw=black!80, thick, minimum size=8mm}, font=\small]
\draw (0,-20) -- (30,0);
\draw (18,0) -- (48,-20);
\draw (0,-15) -- (10,-24);
\draw (7,-10) -- (17,-19);
\draw (48,-15) -- (38,-24);
\draw (41,-10) -- (31,-19);
\node at (12,-26) {$E_1$};
\node at (19,-21) {$E_2$};
\node at (30,-21) {$E_3$};
\node at (37,-26) {$E_4$};
\node at (14,-6) {$L_{12}$};
\node at (34,-6) {$L_{34}$};
\end{tikzpicture}
\caption{A configuration $S_0$ for $k=4$}
\label{fig:configuration S0 n=4}
\end{figure}
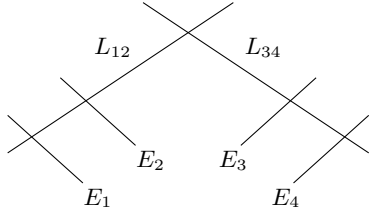

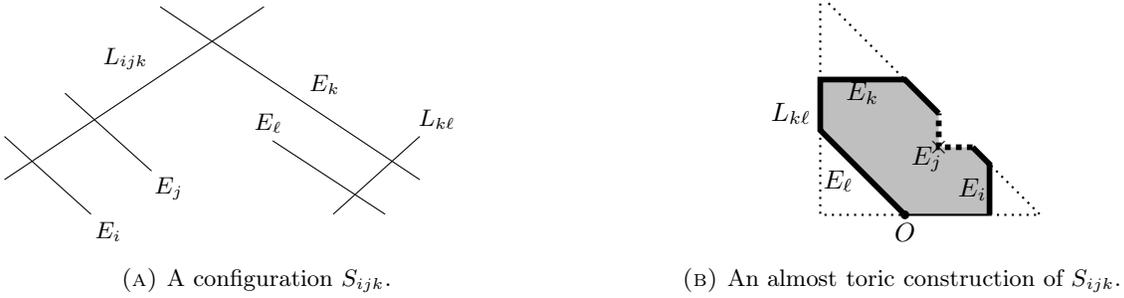
\begin{figure}[H]
\begin{subfigure}{0.45\textwidth}
\begin{tikzpicture}[scale=0.115, roundnode/.style={circle, draw=black!80, thick, minimum size=8mm}, font=\small]
\draw (0,-20) -- (30,0);
\draw (18,0) -- (48,-20);
\draw (0,-15) -- (10,-24);
\draw (7,-10) -- (17,-19);
\draw (48,-15) -- (38,-24);
\draw (44,-24) -- (31,-15.5);
\node at (12,-26) {$E_i$};
\node at (19,-21) {$E_j$};
\node at (37,-9) {$E_k$};
\node at (50,-13) {$L_{k\ell}$};
\node at (14,-6) {$L_{ijk}$};
\node at (30.5,-13.5) {$E_{\ell}$};
\end{tikzpicture}
\caption{A configuration $S_{ijk}$.}
\label{fig:configuration Sijk n=4}
\end{subfigure}
\hfill
\begin{subfigure}{0.4\textwidth}
\centering{
\begin{tikzpicture}[scale=0.225]

\filldraw[fill=lightgray,opacity=0.5,draw=none] (0,8) -- (0,5) -- (5,0) -- (10,0) -- (10,3) -- (9,4) -- (7,4) -- (7,6) -- (5,8) -- cycle;

\draw[thick,black] (5,0) -- (10,0);
\draw[line width=0.5ex,black] (5,0) -- (0,5) -- (0,8) -- (5,8) -- (7,6);
\draw[line width=0.5ex,black] (9,4) -- (10,3) -- (10,0);

\draw[dotted,thick] (0,8) -- (0,13) -- (5,8);
\draw[dotted,thick] (10,0) -- (13,0) -- (10,3);
\draw[dotted,thick] (0,5) -- (0,0) -- (5,0);

\draw[dotted,line width=0.5ex] (9,4) -- (7,4) -- (7,6);
\node at (7,4) {\(\times\)};

\draw[fill] (5,0) circle (1.5ex);
\node at (5,-1) {$O$};

\node[left] at (2.5,2) {$E_\ell$};
\node at (2.5,7.2) {$E_k$};
\node at (9,1.5) {$E_i$};
\node at (6.3,3.3) {$E_j$};
\node[left] at (0,6) {$L_{k\ell}$};
\end{tikzpicture}}

\caption{An almost toric construction of $S_{ijk}$.}
\label{fig:almost-toric blow-up n=4}
\end{subfigure}
\caption{Configurations of type $\mT_{ijk}$ for $n=4$ balls.}
\label{fig:Sijk for n=4}
\end{figure}

\begin{figure}[H]
\begin{subfigure}{0.45\textwidth}
\centering{
\begin{tikzpicture}[scale=0.115, roundnode/.style={circle, draw=black!80, thick, minimum size=8mm}, font=\small]
\draw (10,15) -- (10,30); 
\draw (20,15) -- (20,30); 
\draw (30,15) -- (30,30); 
\draw (40,15) -- (40,30); 
\draw (0,25) -- (45,25); 
\node at (10,32) {$E_1$};
\node at (20,32) {$E_2$};
\node at (30,32) {$E_3$};
\node at (40,32) {$E_4$};
\node at (50,25) {$L_{1234}$};
\end{tikzpicture}}

\caption{A configuration $S_{1234}$}
\label{fig:configurationS1234}
\end{subfigure}
\hfill
\begin{subfigure}{0.4\textwidth}
\centering{
\begin{tikzpicture}[scale=0.225]

\filldraw[fill=lightgray,opacity=0.5,draw=none] (0,8) -- (0,0) -- (11,0) -- (11,2) -- (10,3) -- (8,3) -- (8,5) -- (7,6) -- (6,6) -- (6,7) -- (5,8) -- cycle;

\draw[thick,black] (0,8) -- (0,0) -- (11,0);
\draw[line width=0.5ex,black] (0,8) -- (5,8) -- (6,7);
\draw[line width=0.5ex,black] (8,5) -- (7,6);
\draw[line width=0.5ex,black] (10,3) -- (11,2) -- (11,0);

\draw[dotted,thick] (0,8) -- (0,13) -- (5,8);
\draw[dotted,thick] (11,0) -- (13,0) -- (11,2);

\draw[dotted,line width=0.5ex] (10,3) -- (8,3) -- (8,5);
\node at (8,3) {\(\times\)};

\draw[dotted,line width=0.5ex] (7,6) -- (6,6) -- (6,7);
\node at (6,6) {\(\times\)};

\draw[fill] (0,0) circle (1.5ex);
\node at (0.7,0.7) {$O$};

\node at (2.5,7.2) {$E_1$};
\node at (10,1) {$E_2$};
\node at (7.3,2.3) {$E_3$};
\node at (5.3,5.3) {$E_4$};
\end{tikzpicture}}

\caption{An almost toric construction of a configuration $S_{1234}$.}
\label{fig:almost-toric blow-up S1234 n=4}
\end{subfigure}
\caption{Configurations of type $\mT_{1234}$ for $n=4$ balls.}
\label{fig:configurations of type 1234}
\end{figure}
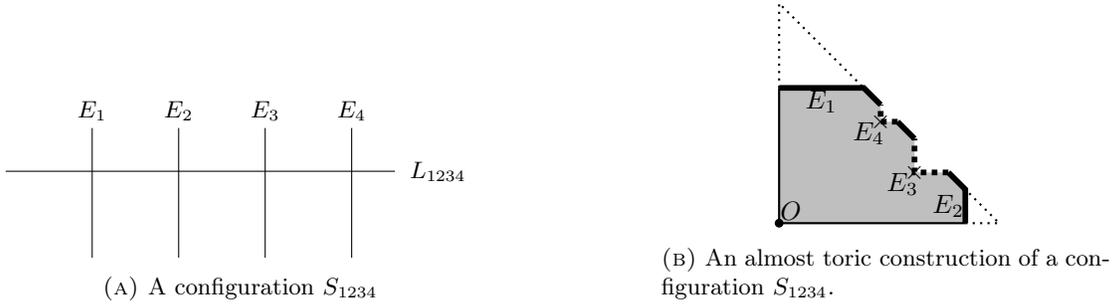


\begin{definition}
A configuration of embedded symplectic spheres $S=S_1\cup\cdots\cup S_k$ in $(M,\om)$ is said to be standard if   \begin{enumerate}
\item $S$ is almost complex for some compatible $J\in\mJ(\om)$,
\item the intersections between the components $S_i$ are transverse and symplectically orthogonal,
\item $[S_i]\cdot [S_j]\leq 1$ for $i\neq j$,
\item $S_i\cap S_j\cap S_k=\emptyset$ for every triple of distinct indices $i,j,k$,
\item $S$ is simply connected. 
\end{enumerate}
\end{definition}

\begin{proposition}\label{prop:BanyagaExtensionOfIsotopies}
Let $S_t=S_{1t}\cup\cdots\cup S_{kt}$, $t\in[0,1]$, be a family of standard configurations in $(M,\om)$. Then there exist symplectomorphisms $\phi_t:M\to M$ depending continuously on $t$, such that $\phi_0=\id$ and $\phi_t(S_0) = S_t$.
\end{proposition}
\begin{proof}
From Moser’s stability theorem, there exists a continuous family of symplectic embeddings $
\psi_{t} \colon (S_0,\om\vert_{S_0})\to (S_t,\om\vert_{S_t})$. By the generalization of Weinstein's symplectic neighborhood theorem to orthogonal configurations given in~\cite[Theorem 2]{Gua18}, we can extend these maps to symplectic embeddings $\psi_t \colon U_0\to U_t$, where $U_t$ is some tubular neighborhood of $S_t$ modelled on a symplectic plumbing of the symplectic normal bundles $N(S_{it})$. By Banyaga's extension theorem~\cite[Theorem~II.2.3]{Ban78}, there exists a symplectic isotopy $\phi_t$ that coincides with $\psi_t$ on $S_0$.
\end{proof}

Recall that $\mE(M)\subset H_2(M,\Z)$ is the set of classes that can be represented by exceptional spheres, i.e. symplectically embedded spheres of self-intersection $-1$.

\begin{proposition}[{\cite[Theorem 1.2.7 (iii)]{McDOp15}}]\label{prop:McDuff-Opshtein}
Let $(M,\om)$ be a symplectic $4$-manifold with an orthogonal configuration $S$ of $k$ embedded symplectic spheres $S=S_1\cup\cdots\cup S_k$, and suppose that $A\in \mE$ satisfies $A\cdot S_i\geq 0$ for all $1\leq i\leq k$.
Then there is an open, dense, and connected subset $\mJ([A],S)\subset \mJ(S)$ such that $A$ is represented by an embedded $J$-holomorphic curve.\qed
\end{proposition}

\begin{proposition}[{\cite[Proposition 2.1]{BLW14}}]\label{ConnectednessEmbeddingsInComplementZ}
Let $(M,\om)$ be a closed rational or ruled symplectic $4$-manifold and let $Z\subset M$ be a closed symplectic sphere. Then, for any choice of capacities $\B c=c_{1},\ldots, c_{n}$, the space $\Emb(\B c, M\setminus Z)$ of symplectic embeddings $B(c_{i})\sqcup\cdots\sqcup B(c_{n})\into (M\setminus Z, \om)$ in the complement of $Z$ is connected whenever it is non-empty.\qed
\end{proposition}

\begin{proposition}\label{Abreu-McDuff-HamiltonianIsotopies}
Let $(M,\om)$ be a rational, ruled, symplectic $4$-manifold and let $Z$ be an embedded symplectic sphere in the class of a section of nonpositive self-intersection. Then any other embedded symplectic sphere $Z'$ homologous to $Z$ is Hamiltonian isotopic to $Z$.
\end{proposition}
\begin{proof}
By \cite[Corollary 2.8]{AbMcD00}, the space $\mJ([Z])$ is connected, which implies that the space $\mC([Z])$ of embedded symplectic spheres in class $[Z]$ is also connected. By Proposition~\ref{prop:BanyagaExtensionOfIsotopies}, the group $\Symp_h(M,\om)$ acts transitively on $\mC([Z])$, and it follows from \cite[Corollary 2.7]{AbMcD00} that we have equalities $\Symp_h(M,\om)=\Symp_0(M,\om)=\Ham(M,\om)$.
\end{proof}

\begin{proposition}\label{prop:NegativeCurvesHamiltonianIsotopic}
Let $(M,\om)$ be a closed rational or ruled symplectic $4$-manifold. Let $Z$ be an embedded symplectic sphere of self-intersection $-2$ or $-3$. Then the group $\Symp_h(M,\om)$ acts transitively on the set of embedded symplectic spheres homologous to $Z$.
\end{proposition}
\begin{proof} The proof depends on whether $A=[Z]$ is a characteristic element of the intersection lattice $H_{2}(M,\Z)$. (Recall that an element $A\in H_{2}(M,\Z)$ is characteristic if for all $B\in H_{2}(M;\Z)$, $A\cdot B = B\cdot B \mod 2$.)

When the class $A$ is not characteristic, it follows from \cite[Proposition 5.15]{DoLiWu18} and \cite[Proposition 5.17]{DoLiWu18} that $A$ is Cremona equivalent to 
\begin{enumerate}
\item $B-F$ if $A\cdot A=-2$ and $M=S^{2}\times S^{2}$,
\item $E_{1}-E_{2}$ if $A\cdot A=-2$ and $M=\cp^{2}\#n\overline{\cp^{2}}$, $n\geq 2$,
\item $2E_{1}-L$ if $A\cdot A=-3$ and $M=\cp^{2}\#n\overline{\cp^{2}}$, $n\geq 1$.
\end{enumerate}
In the first case, the statement of the proposition follows directly from Proposition~\ref{Abreu-McDuff-HamiltonianIsotopies}. In the two other cases, let $Z_0$ and $Z_1$ be two symplectic, embedded spheres representing the class $A$. We can find $(n-1)$ orthogonal exceptional classes $V_1,\ldots, V_{n-1}$ that have zero intersection with $A$. By an inductive application of Proposition~\ref{prop:McDuff-Opshtein}, we can find an almost structure $J_i$, $i=0,1$, for which the classes $V_j$ are represented by holomorphic spheres $\Sigma_{ij}$ and for which $Z_i$ is also holomorphic. We can find a Hamiltonian isotopy $f_t$, $t\in[0,1]$, that takes the exceptional curves $\Sigma_{1j}$ to $\Sigma_{0j}$, so that there is no lost of generality in assuming $\Sigma_{0j}=\Sigma_{1j}$. By positivity of intersections, the spheres $\Sigma_j$ in classes $V_j$ are disjoint from the curves $Z_i$ and can be blown down to obtain a rational ruled surface $(\overline{M},\overline{\om})$ that is diffeomorphic to $S^2\times S^2$ if $A\cdot A=-2$ or to $\cp^2\#\overline{\cp^2}$ if $A\cdot A=-3$. In both cases, the blow-down curves $\overline{Z_i}$ are sections that are disjoint from a collection of $(n-1)$ symplectic balls $B_{j}\subset \overline{M}$. By Proposition~\ref{Abreu-McDuff-HamiltonianIsotopies}, there is an Hamiltonian isotopy $\psi_t$, $t\in[0,1]$, taking $\overline{Z}_1$ to $\overline{Z}_0$. By Proposition~\ref{ConnectednessEmbeddingsInComplementZ}, there is an Hamiltonian isotopy $\phi_t$, $t\in[0,1]$, with support in $\overline{M}\setminus \overline{Z}_0$, taking the balls $\psi_1(B_{j})$ back to the balls $B_{j}$. Moreover, we can arrange that $\phi_1\circ\psi_1$ is the identity near the balls $B_j$, so that $\phi_1\circ\psi_1$ lifts to a symplectomorphism $g\in\Symp_h(M,\om)$ that takes $Z_1$ to~$Z_0$.

When the class $A$ is characteristic, it follows again from \cite[Proposition 5.15]{DoLiWu18} and \cite[Proposition 5.17]{DoLiWu18} that $A$ is Cremona equivalent to either
\begin{enumerate}
\item $L-E_1-E_2-E_3$ if $A\cdot A=-2$ and $M=\cp^{2}\#3\overline{\cp^{2}}$,
\item $L-E_1-E_2-E_3-E_4$ if $A\cdot A=-3$ and $M=\cp^{2}\#4\overline{\cp^{2}}$.
\end{enumerate}
Let $(\widetilde{M},\widetilde{\om})$ be the symplectic blow-up of $(M,\om)$ at a ball $B_0$ of capacity $\epsilon>0$ away from two spheres $Z_0$ and $Z_1$ in class $A$, and let $E_0$ be the class of the new exceptional divisor $\Sigma_0$. Then the exceptional classes $L_{i0}:=L-E_0-E_i$, $i\neq0$, are pairwise orthogonal and have zero intersection with $A$. In particular, $A$ is no longer characteristic. Note that blowing down disjoint divisors representing the classes $L_{0i}$ yields a rational ruled surface $(\overline{M},\overline{\om})$ whose fiber is in the homology class $L-E_0$ so that the blown down curves $\overline{Z_i}$ are in the class of a section. The previous argument then shows the existence of some symplectomorphism $\widetilde{g}\in\Symp_h(\widetilde{M},\widetilde{\om})$ taking $\widetilde{Z_1}$ to $\widetilde{Z_0}$. Composing with a Hamiltonian isotopy supported away from $\widetilde{Z_0}$, we can assume that $\widetilde{g}$ is the identity near the exceptional divisor $\Sigma_0$. Blowing down $\Sigma_0$ gives the required symplectomorphism $g\in\Symp_h(M,\om)$ that takes $Z_1$ to~$Z_0$.
\end{proof}

\begin{proposition}\label{prop:HomogeneityConfigurations}
Let $\mT$ be any of the above types of configurations characterizing a stratum. In the case $n=1$, the subgroup $\mG_{\B c,p}:=\OP{Symp}_h (\widetilde{M}_{\B c}, p)$ fixing $p\in\Sigma$ acts transitively on the space $\mC_{\B c}^\circ(\mT)$ of orthogonal configurations of type~$\mT$. For $2\leq n\leq 4$, the full group $\mG_{\B c}=\OP{Symp}_h (\widetilde{M}_{\B c})$ acts transitively on the space $\mC_{\B c}^\circ(\mT)$ of orthogonal configurations of type~$\mT$. 
\end{proposition}
\begin{proof} We first consider the action of $\mG_{\B c,p}$ on orthogonal configurations of type $\mT_{0}$ in the case $n=1$. Since the stratum $\mJ_{\B c}(\mT_0)$ is path-connected, Lemma~\ref{lemma:CharaterizationStrataConfigurations} implies that $\mC_{\B c}^\circ(\mT_{0})$ is also connected. By Proposition~\ref{prop:BanyagaExtensionOfIsotopies}, it follows that the group $\mG_{\B c,p}$ acts transitively on~$\mC_{\B c}^\circ(\mT_0)$. The exact same argument shows that for $2\leq n\leq 4$, the full group $\mG_{\B c}$ acts transitively on~$\mC_{\B c}^\circ(\mT_0)$.

Next, we consider two orthogonal configurations $C_i$, $i=0,1$, of type $\mT_{123}$ in $\widetilde{M}_3$, each consisting of a set of disjoint exceptional divisors $\Sigma_i$ and of some embedded sphere $Z_i$ in class $L-E_1-E_2-E_3$. By Proposition~\ref{prop:NegativeCurvesHamiltonianIsotopic}, there is an element $g\in\Symp_h(\widetilde{M}_3,\widetilde{\om}_{\B c})$ such that $g(Z_1)=Z_0$ and which takes $\Sigma_1$ to $g(\Sigma_1)$. By Proposition~\ref{prop:McDuff-Opshtein}, the space $\mJ_{\B c}([\Sigma], Z_0)$ is connected. It follows that the spaces $\mC_{\B c}([\Sigma],Z_0)$ and $\mC_{\B c}^\circ([\Sigma],Z_0)$ are connected. Again, by Proposition~\ref{prop:BanyagaExtensionOfIsotopies}, the group $\Symp_h(\widetilde{M}_3,\widetilde{\om}_{\B c})$ acts transitively on $\mC_{\B c}^\circ([\Sigma],Z_0)$. Consequently, we can find $f\in\Symp_h(\widetilde{M}_3,\widetilde{\om}_{\B c})$ taking $g(\Sigma_1)$ back to $\Sigma_0$ and leaving $Z_0$ invariant, so that $f\circ g$ is the required symplectomorphism sending $C_1$ to $C_0$. 

Note that the argument for configurations of type $\mT_{123}$ applies mutatis mutandis to orthogonal configurations of type $\mT_{1234}$ in $(\widetilde{M}_4,\om_{\B c})$.

We are left with orthogonal configurations of type $\mT_{ijk}=[\Sigma] \cup [L_{ijk}]\cup [L_{m\ell}] $ in $(\widetilde{M}_4,\om_{\B c})$. Let $C_0$ and $C_1$ be two such configurations, and let $Z_0$ and $Z_1$ be their components in class $L_{ijk}$. By Proposition~\ref{prop:NegativeCurvesHamiltonianIsotopic}, there is an element $g\in\Symp_h(\widetilde{M}_3,\widetilde{\om}_{\B c})$ such that $g(Z_1)=Z_0$. Since $[L_{m\ell}]$ is an exceptional class, Proposition~\ref{prop:McDuff-Opshtein} still implies that the space $\mC_{\B c}^\circ([\Sigma]\cup[L_{m\ell}],Z_0)$ is connected. Applying Proposition~\ref{prop:BanyagaExtensionOfIsotopies} finishes the proof.
\end{proof}

\begin{proposition}\label{prop:CompactlySupportedSymplectomorphisms}
Let $S$ be a standard configuration of type $\mT_0$, $\mT_{ijk}$, or $\mT_{1234}$. In the case $S$ is of the type $\mT_{234}$ assume further that $c_1<1/2$. Let $U=\widetilde{M}_4\setminus S$ be its complement. Then  group $\Symp_c(U)$ of symplectomorphisms with compact support in $U$ is contractible.
\end{proposition}
\begin{proof}
We will only prove the statement for $n=4$ balls as the cases $1\leq n\leq 3$ are similar and simpler.

Since $\Symp_h$ acts transitively on standard configurations of types $\mT_0$, $\mT_{ijk}$, and $\mT_{1234}$, it suffices to prove the statement for one configuration of each types.

If $S$ is of type $\mT_0$, then this is Proposition~3.3 in~\cite{LLW15}. As we explain below, the argument can be readily adapted to the other configuration types except $\mT_{234}$ for which we have to impose the extra assumption $c_1<1/2$. 

Let's assume $[\widetilde{\om}_{\B c}]$ is rational so that there is $k\in\N_*$ such that $k[\om]$ is integral. The first step is to write some positive multiple $mk\PD[\widetilde{\om}_{\B c}]$ as a strictly positive integral combination of \emph{all} the homology classes represented by the components of $S$, so that $mk\PD[\widetilde{\om}_{\B c}]$ is represented by a positive divisor $D$ whose underlying set is $S$. For the case $S$ is of type $\mT_{1234}$, we have
\begin{align*}
\PD[\widetilde{\om}_{\B c}] &= L-c_1E_1-c_2E_2-c_3E_3-c_4E_4\\
&= (L-E_1-E_2-E_3-E_4) + (1-c_1)E_1 + (1-c_2)E_2 + (1-c_3)E_3 + (1-c_4)E_4.
\end{align*}

Now, let $S$ be a holomorphic configuration of type $\mT_{234}=\Sigma\cup[L_{234}]\cup [L_{14}]$. For two nonnegative integers  $a,b\in\N$,
\begin{align*}
(a+b)\PD[\widetilde{\om}_{\B c}] &= (a+b)(L-c_1E_1-c_2E_2-c_3E_3-c_4E_4)\\
\begin{split} &= a(L-E_2-E_3-E_4) + b(L-E_1-E_4) \\
& \quad + (b-(a+b)c_1)E_1 + (a-(a+b)c_2)E_2 + (a-(a+b)c_3)E_3 + (a+b)(1-c_4)E_4
\end{split}
\end{align*}
and because $0<c_4\leq\cdots\leq c_1<1$ and $c_i+c_j<1$, all coefficients are strictly positive iff $a=b=1$ and $c_1<1/2$, in which case
\begin{align*}
2\PD[\widetilde{\om}_{\B c}] &= 2(L-c_1E_1-c_2E_2-c_3E_3-c_4E_4)\\
\begin{split} &= (L-E_2-E_3-E_4) + (L-E_1-E_4) \\
& \qquad + (1-2c_1)E_1 + (1-2c_2)E_2 + (1-2c_3)E_3 + 2(1-c_4)E_4.
\end{split}
\end{align*}
Note that for the other types of configurations $\mT_{ijk}=\Sigma\cup[L_{ijk}]\cup [L_{1\ell}]$, the coefficient of $E_1$ is $2(1-c_1)>0$ and the coefficients of the other classes $E_i$ are $(1-2c_i)>0$, so that we do not have to impose the extra condition $c_1<1/2$.

In all cases, since $mk\PD[\widetilde{\om}_{\B c}]=[D]$, there exists a K\"ahler potential $\phi:U\to\R_+$ so that $mk\widetilde{\om}_{\B c}=d\overline{d}\phi$, making the complement $U$ a Stein domain biholomorphic to $\C$ when $S=S_{1234}$ or to $\C\times \C^*$ when $S=S_{ijk}$. From there on, the argument is identical to the proof of Proposition~3.3 in~\cite{LLW15}.
\end{proof}

\begin{remark}\label{remark:Symplectically Convex Complement}
It is likely that the extra condition $c_1<1/2$ for configurations of type $\mT_{234}$ in Proposition~\ref{prop:CompactlySupportedSymplectomorphisms} can be removed. For instance, a standard configuration $S$ of type $\mT_{ijk}$ or $\mT_{1234}$ is always almost-toric, and its complement $U$ can be identified with a toric domain in $\C^4$. In the case of $\mT_{1234}$, $U$ is always symplectically star-shaped, that is, there exists a Liouville vector field $\overline X$ lifting the radial vector field $X$ on the base diagram, and whose negative flow contracts $U$ into an $\B T^2$-invariant symplectic ball. The contractibility of $\Symp_c(U)$ then follows by adapting the proof of Theorem~9.5.2 in~\cite{MS17}. For configurations of type $\mT_{ijk}$, under any of the conditions $c_\ell+c_i+c_j<1$, $c_\ell+c_i+c_k<1$, or $c_\ell+c_j+c_k<1$, there also exists a Liouville vector field contracting $U$ into $D^2\times D^2_*$. In particular, for the type $\mT_{234}$, the statement holds even if $c_1\geq 1/2$ as long as $c_1+c_i+c_j<1$ for at least one pair of indices $i,j$.
\end{remark}

\begin{proposition}\label{prop:SymplecticStabilisers}
Let $S\in \mC_{\B c}^\circ(\mT_I)$ be a standard configuration characterizing a stratum $\mJ_I([\Sigma])$ in $\mJ_{\B c}([\Sigma])$. Suppose that $S$ is holomorphic with respect to an integrable structure $J_S$. Let $\Stab_h(S)$
\index{S@$\Stab_h(S)$ -- symplectic stabilizer of the configuration $S$ under the action of $Symp_h(\widetilde{M}_{\B c})$}
be the symplectic stabilizer of $S$ under the action of $Symp_h(\widetilde{M}_{\B c})$, $\Aut_h(J_S)$
\index{A@$\Aut_h(J_S)$ -- group of complex automorphisms of the complex structure $J_S$ preserving homology}
be the group of complex automorphisms of $J_S$ preserving homology and let $\Iso_h(\widetilde{\om}_{\B c},J_S)$
\index{I@$\Iso_h(\widetilde{\om}_{\B c},J_S)$ -- group of Kahler isometries acting trivially on homology}
be the subgroup of Kahler isometries acting trivially on homology. We have the following homotopy equivalences.
\begin{enumerate}[itemsep=1em]
\item In the case $n=1$, $\Stab_h(S)\simeq\Aut_h(J_S)\simeq\Iso_h(\widetilde{\om}_{\B c},J_S)\simeq \U(2)$.
\item In the case $n=2$, $\Stab_h(S)\simeq\Aut_h(J_S)\simeq\Iso_h(\widetilde{\om}_{\B c},J_S)\simeq \B T^2$.
\item In the case $n=3$,
\begin{enumerate}
\item for $S$ of type $\mT_0$, $\Stab_h(S)\simeq\Aut_h(J_S)\simeq\Iso_h(\widetilde{\om}_{\B c},J_S)\simeq \B{T}^2$,
\item for $S$ of type $\mT_{123}$, $\Stab_h(S)\simeq\Aut_h(J_S)\simeq\Iso_h(\widetilde{\om}_{\B c},J_S)\simeq \B{S}^1_{ijk}$.
\end{enumerate}
\item In the case $n=4$,
\begin{enumerate}
\item for $S$ of type $\mT_0$, $\Stab_h(S)\simeq\Aut_h(J_S)=\Iso_h(\widetilde{\om}_{\B c},J_S)= \mathbb{1}$,
\item for $S$ of type $\mT_{ijk}$, $\Stab_h(S)\simeq\Aut_h(J_S)\simeq\Iso_h(\widetilde{\om}_{\B c},J_S)\simeq \B{S}^1_{ijk}$,
\item for $S$ of type $\mT_{1234}$, $\Stab_h(S)\simeq \B{S}^1_{1234}\times\B{F}_2$, where $\B F_2$ denotes the free group on two generators, and $\Aut_h(J_S)\simeq\Iso_h(\widetilde{\om}_{\B c},J_S)\simeq \B{S}^1_{1234}$.
\end{enumerate}
\end{enumerate}
In all cases, the identity component $\Stab_0(S)$
\index{S@$\Stab_0(S)$ -- identity component of $\Stab_h(S)$}
of the symplectic stabilizer is homotopy equivalent to the isometry group $\Iso_h(\widetilde{\om}_{\B c},J_S)$. 
\end{proposition}
\begin{proof}
Again, we will only prove the statement for $n=4$ balls as the $n\leq3$ cases are similar. Note that the statement for $n=1$ and $n=2$ follows, respectively, from the proofs of \cite[Proposition 2.6 (iii)]{AbMcD00} and \cite[Proposition 3.1]{P08i}, while for $n=3$ it follows from~\cite[Corollary~B.9 and Corollary~B.12]{AP13}.

(4a) We first consider the open stratum. We can assume $J_S$ is a complex structure obtained by blowing up $\cp^2$ at $4$ generic points. It follows that the group $\Aut(J_S)$ of complex automorphisms is trivial. Now, the fact that $\Stab_h(S)$ is contractible is proven by looking at a sequence of fibrations that reduce the problem to symplectomorphism groups of pointed surfaces and automorphisms of symplectic plumbings as described in~\cite[Section~4]{Evans}.

Let $C_0=C_{12}\cup C_{34}$ be the two curves in $S$ in classes $[L_{12}]$ and $[L_{34}]$, so that $S=\Sigma\cup C_0$. Restricting an element $\phi\in\Stab_h(\Sigma\cup C_0)$ to $C_0$ gives
\[\Fix(C_0)\to\Stab_h(\Sigma\cup C_0)\to\Symp(C_0)\simeq *,\]
where $\Fix(C_0)$ is the subgroup of elements fixing $C_0$ pointwise and leaving $\Sigma$ invariant. For $\phi\in\Fix(C_0)$, its differential $d\phi$ along $C_0$ is a symplectic automorphism of the normal bundle $\mN(C_0)$ so that we have a second fibration
\[H\to\Fix(C_0)\to\Aut(\mN(C_0))\simeq \Map\big((S^2,*),(S^1,*)\big)\times\Map\big((S^2,*),(S^1,*)\big) \simeq *,\]
whose fiber $H$ is homotopy equivalent to the group $H_0$ of symplectomorphisms $\phi$ acting as the identity near $C_0$ and leaving $\Sigma$ invariant. Restricting $\phi\in H_0$ to the $4$ curves in $\Sigma$ yields a third fibration
\[H_1\to H_0\to \Symp(\Sigma,~\id\text{~near~}p_i)\simeq *,\]
where $p_i=\Sigma_i\cap C_0$. Taking differentials of elements $\phi\in H_1$ along $\Sigma$, we get a fourth fibration
\[\Symp_c(\widetilde{M}\setminus S)\simeq H_2\to H_1\to \Aut(\mN(\Sigma),\, \id \text{~near~}p_i)\simeq *,\]
whose fiber is homotopy equivalent to symplectomorphisms of $\widetilde{M}$ supported away from the configuration $S$. By Proposition~\ref{prop:CompactlySupportedSymplectomorphisms}, this latter group is contractible, so that $\Stab(\Sigma\cup C_0)$ is also contractible. This completes the argument for the open stratum $\mJ_0$.

(4b) We now consider a configuration $S$ of type $\mT_{ijk}$. We can assume that $S$ is obtained by blowing up $\cp^2$ at $3$ points $p_i$, $p_j$, $p_k$ in the line $[0:z_1:z_2]$, and at $p_\ell=[1:0:0]$. It follows that the group $\Aut(J_S)$ of complex automorphisms acting trivially on homology is the subgroup of $\PU(3)$ fixing pointwise the line at infinity together with the origin: 
\[
\Aut(J_S)= \left\{A=\begin{pmatrix}\lambda&0&0\\0&a&0\\0&0&a\end{pmatrix}\in\PU(3)~|~\lambda,a\in\C^*\right\}.
\] 
In particular, it is homotopy equivalent to the circle $\B S^1\subset \PU(3)$ acting on $\cp^2$ by \[t\cdot[z_0:z_1:z_2]=[tz_0:z_1:z_2].\]
To show that the symplectic stabilizer $\Stab_h(S)$ of the configuration is homotopy equivalent to the same circle $\B S^1$, let $C_{ijk}$ be the component of $S$ in the class $[L_{ijk}]$, and consider the restriction fibration
\[\Fix(C_{ijk})\to\Stab_h(\Sigma\cup C_{ijk})\to\Symp(C_{ijk},~ \id\text{~on~}\{p_i,p_j,p_k\})\simeq *.\]
Restricting the differential of $\phi\in\Fix(C_{ijk})$ to the normal bundle $\mN(C_{ijk})$ yields
\[H_0\to\Fix(C_{ijk})\to \Aut(\mN(C_{ijk}))\simeq \B{S}^1,\]
where $H_0\subset\Stab(C_{ijk})$ is the subgroup of symplectomorphisms acting trivially near $C_{ijk}$. Proceeding as before, it is easy to show that this subgroup is homotopy equivalent to $\Symp_c(\widetilde{M}_{\B c}\setminus S)$, which is contractible. Consequently, 
$\Stab_h(S)\simeq \B{S}^1$.

(4c) A configuration $S$ of type $\mT_{1234}$ is obtained by blowing up $4$ points on the line at infinity $L:=\{[0:z_1:z_2]\}$. The automorphism group $\Aut_h(J_S)$ is the subgroup of $\PSL(3,\C)$ that fixes a line pointwise, namely,
\[
\Aut(J_S)=\left\{A=\begin{pmatrix}1&0&0\\b&a&0\\c&0&a\end{pmatrix}\in\PU(3)~|~a\in\C^*,~b,c\in\C\right\}\simeq \B{S}^1.
\] 
Let $C_{1234}\subset S$ be the curve in class $[L_{1234}]$. Restricting an element $\phi\in\Stab_h(\Sigma\cup C_{1234})$ to $C_{1234}$ defines a map
\begin{equation}\label{eq:restriction 4 points}
\Stab_h(\Sigma\cup C_{1234})\to\Symp(C_{1234},~4)\simeq \B{F}_2,
\end{equation}
where $\Symp(C_{1234},~4)$ is the group of symplectomorphisms fixing the four points $\{p_1,p_2,p_3,p_4\}$, and where $\B{F}_2$ is the free group on two generators. We claim that this map is a fibration. To see this, we first note that any Hamiltonian diffeomorphism $\phi\in \Symp(C_{1234},~4)$ that is isotopic to the identity through a Hamiltonian isotopy $\phi_t$ can be lifted to $\Stab_h(\Sigma\cup C_{1234})$ by extending $\phi_t$ in a Weinstein neighborhood of $C_{1234}$. It is thus sufficient to show that the restriction map~\eqref{eq:restriction 4 points} is surjective at the $\pi_0$ level.

For this, we look at the action of $\Symp(S^2,~3)$ on a fourth point $p_4\in S^2\setminus\{p_1,p_2,p_3\}$. This yields a fibration
\[\Symp(S^2,~4)\to\Symp(S^2,~3)\to S^2\setminus\{p_1,p_2,p_3\}\]
whose monodromy induces the isomorphism $\B F_2\simeq \pi_1(S^2\setminus\{p_1,p_2,p_3\})\simeq\pi_0(\Symp(S^2,~4))$. Given a class $[\phi_1]\in \pi_0(\Symp(S^2,~4))$, let $\phi_t$ be a path in $\Symp(S^2,~3)$ connecting the identity to $\phi_1$ and representing $[\phi_1]$. We can suppose that $\phi_t$ is the identity on $3$ discs $D_i$ centered at the first three points $p_1$, $p_2$, $p_3$, and that $\phi=\phi_1$ is also the identity on a disc $D_4$ centered at $p_4$. Identify $S^2$ with a line $L\subset\cp^2$, and let $N$ be a Weinstein neighborhood of $L$. Let $B(\epsilon_1)\sqcup\cdots\sqcup B(\epsilon_4)$ be four symplectic balls of capacities $\epsilon_i$ in $N$ whose intersections with $L$ is contained in the interior of the discs $D_i$. We extend the Hamiltonian isotopy $\phi_t$ to a isotopy $\overline{\phi_t}$ of $N$ that is the identity near the first three balls $B(\epsilon_i)$, and such that $\overline{\phi_1}$ is also the identity near $B(\epsilon_4)$. It follows that $\overline{\phi_1}$ lifts to a symplectomorphism $\psi$ on the symplectic blow-up of the balls $B(\epsilon_i)$. By construction, this symplectomorphism leaves the four exceptional divisors $\Sigma_i$ fixed and sends the proper transform $C_{1234}$ of $L$ to itself, that is, $\psi\in\Stab(S)$. Moreover, its restriction to $C_{1234}$ represents the class $[\phi_1]\in \pi_0(\Symp(S^2,~4))$. This shows that the restriction map~\eqref{eq:restriction 4 points} is surjective whenever the capacities $c_i$ are small enough. By the description of the stability chambers given in Section~\ref{section:stability chambers n=4}, we can assume that the capacities $c_1,\ldots,c_4$ are as small as we want, which concludes the proof of the claim. 

We now look at the fibration
\[\Fix(C_{1234})\to\Stab_h(\Sigma\cup C_{1234})\to\Symp(C_{1234},~4)\simeq \B{F}_2.\]
Restricting the differential of $\phi\in\Fix(C_{1234})$ to the normal bundle $\mN(C_{1234})$ yields
\[H_0\to\Fix(C_{1234})\to \Aut(\mN(C_{1234}))\simeq \B{S}^1,\]
where $H_0\subset\Stab(C_{1234})$ is the subgroup of symplectomorphisms acting trivially near $C_{1234}$. The same arguments as in the previous cases show that this subgroup is homotopy equivalent to $\Symp_c(\widetilde{M}_{\B c}\setminus S)$, which is contractible. Consequently, as topological spaces,
\[\Stab_h(S)\simeq \B{S}^1\times \B{F}_2.\]
In particular, $\Stab_0(S)\simeq \B{S}^1$.
\end{proof}

\begin{proposition}\label{prop:structure of the stabilizer J_1234}
Let $S$ be a configuration of type $\mT_{1234}$. There is an injective homomorphism $\B S^1\times \B F_2 \into \Stab_h(S)$ that induces a homotopy equivalence $\OP{B}\B S^1\times \OP{B}\B F_2\simeq \BStab_h(S)$. 
\end{proposition}
\begin{proof}
From Proposition~\ref{prop:SymplecticStabilisers}~(4c), we know that $\Stab_h(S)/\Stab_0(S)=\pi_0(\Stab_h(S))\simeq \B F_2$, and because $\B F_2$ is free, there is a section $\sigma:\B F_2\into \Stab_h(S)$. Consider the diagram of exact sequences 
\begin{equation}
\begin{tikzcd}
\B S^1 \arrow{d}{\simeq} \ar[r] & \B S^1 \rtimes \B F_2 \arrow{d}{\simeq} \ar[r] & \B F_2\arrow{d}{\simeq} \\
\Stab_0(S) \ar[r] & \Stab_h(S) \ar[r] & \pi_0(\Stab_h(S))\\
\Fix(C_{1234}) \arrow{u}[swap]{\simeq} \ar[r] & \Stab_h(S) \arrow[equal]{u} \ar[r] & \Symp(C_{1234},4) \arrow{u}[swap]{\simeq}
\end{tikzcd}
\end{equation}
Looking at the lifts of the elements of $\B F_2$ constructed in the proof of Proposition~\ref{prop:SymplecticStabilisers} (4c), it follows immediately that given $a\in\B F_2$ and $t\in\B S^1\subset\Stab_0(S)$, the conjugate $ata^{-1}$ restricts to the image of $t$ in the bundle automorphisms group $\Aut(\mN(C_{1234}))$. Consequently, $\B F_2$ acts trivially on $\B S^1$ and there is an injective homomorphism $\B S^1\times \B F_2\to \Stab_h(S)$ that is a homotopy equivalence.
\end{proof}

\begin{corollary}\label{cor:HomogeneityStrataJ}
The group $Symp_h(\widetilde{M}_{\B c})$ acts on $\mJ_{\B c}([\Sigma])$ preserving the stratification. Each of the strata $\mJ_0([\Sigma])$ and $\mJ_{ijk}([\Sigma])$ is homotopy equivalent to the orbit of an integrable complex structure. For a complex structure $J_S$ in the stratum $\mJ_{1234}([\Sigma])$, the homotopy fiber of the inclusion of its orbit $(\Symp_h(\widetilde{M}_{\B c})\cdot J_S)\into\mJ_{1234}([\Sigma])$ is homotopy equivalent to the free group $\B{F}_2$, while the homotopy fiber of the evaluation map $\Symp_h(\widetilde{M}_{\B c})\to\mJ_{1234}([\Sigma])$ is equivalent to $\B S^1\times \B F_2$.
\end{corollary}
\begin{proof} Given a complex structure $J$ in a stratum $\mJ_I([\Sigma])$ and the corresponding $J$-holomorphic configuration $S$, this follows from the $\Symp_h(\widetilde{M}_{\B c})$-commutative diagram 
\[
\begin{tikzcd}[column sep=small]
&\mJ_I([\Sigma])\arrow[two heads]{r}{\simeq} & \mC_{\B c}(\mT_I)\arrow[hookleftarrow]{r}{\simeq} & \mC_{\B c}^\circ(\mT_I)\\
\Stab_h(S)/\Iso_h(J_S)\ar[r] & \Symp_h/ \Iso_h(J_S) \arrow[two heads]{r} \ar[u,hook]& \Symp_h/\Stab_h(S)\arrow[hook]{u}[swap]{\simeq} \arrow[equal]{r} & \Symp_h/\Stab_h(S)\arrow[hook]{u}[swap]{\simeq}
\end{tikzcd}
\]
together with Proposition~\ref{prop:SymplecticStabilisers}.
\end{proof}

Using the same techniques, it is easy to show that the strata of the spaces $\mJ_{\B c}(\Sigma)$, $\mA(n,\Sigma)$, and $\mA(n,[\Sigma])$ are co-oriented submanifolds of the same codimensions than the strata of $\mJ_{\B c}([\Sigma])$. By Propositions~\ref{prop:equivalence S and [S]} and~\ref{prop:equivalence G and D actions}, all the actions on these stratified spaces defined in Section~\ref{section:equivalence homotopy orbits} share similar properties.

For convenience, we collect the main results of this section in a single statement.
\begin{proposition}\label{prop:Symp action on J} Let $\B c$ be an admissible capacity and let $\mG_{\B c}:=\Symp(\widetilde{M}_{\B c})$.
\begin{enumerate}[itemsep=1em]
\item The stratum $\mJ_0\subset\mJ_{\B c}([\Sigma])$ is open, dense, and connected. It is homotopy equivalent to the orbit of a complex structure $J_0$ under the action of $\mG_{\B c}$. The homotopy fiber $\mF_{\ev}$ of the evaluation map $\ev:\mG_{\B c}\to\mJ_0$ at $J_0$ is homotopy equivalent to $\Iso_h(\widetilde{\om}_{\B c},J_0)\simeq\Aut_h(J_0)$, that is, $\mF_{\ev}\simeq\U(2)$ if $n=1$, $\mF_{\ev}\simeq\B T^2$ if $n=2,3$, and $\mF_{\ev}\simeq\mathbb{1}$ if $n=4$.

\item Whenever nonempty, the stratum $\mJ_{ijk}\subset \mJ_{\B c}([\Sigma])$ is a co-oriented submanifold of codimension $2$. It is homotopy equivalent to the orbit of a complex structure $J_{ijk}$ under the action of $\mG_{\B c}$. The homotopy fiber of the evaluation map $\ev:\mG_{\B c}\to\mJ_{ijk}$ at $J_{ijk}$ is homotopy equivalent to $\Iso_h(\widetilde{\om}_{\B c},J_{ijk})\simeq\Aut_h(J_{ijk})\simeq\B S^1_{ijk}$. 

\item Whenever nonempty, the stratum $\mJ_{1234}\subset \mJ_{\B c}([\Sigma])$ is a co-oriented manifold of codimension $4$. It contains a complex structure $J_{1234}$ with automorphism group $\Iso_h(\widetilde{\om}_{\B c},J_{1234})\simeq\Aut_h(J_{ijk})\simeq\B S^1_{1234}$. The homotopy fiber of the evaluation map $\ev:\mG_{\B c}\to\mJ_{1234}$ at $J_{1234}$ is homotopy equivalent to $\B S^1_{1234}\times \B F_2$.
\end{enumerate}
Moreover, similar statements hold for the action of $\mG_{\B c}(\Sigma)$ on $\mJ_{\B c}(\Sigma)$, for the action of $\mD_h(\Sigma)$ on $\mA(n,\Sigma)$, and for the action of $\mD_h$ on $\mA(n,[\Sigma])$.
\end{proposition}

\subsection{Integrable structures and isotropy representations}\label{section:integrable structures 8}
In this section we determine the isotropy representations associated to the action of the symplectomorphism group $\mG_{\B c}$ on $\mJ_{\B c}([\Sigma])$. To this end, we study the restrictions of the actions to the subspace $\mJ_{\B c}^{int}([\Sigma])$ made of integrable complex structures, and show that the isotropy representations of the restricted action can be described through complex deformation theory. We then show that the homotopy orbits of the action of $\mG_{\B c}$ on $\mJ_{\B c}^{int}([\Sigma])$ coincide with the homotopy orbits of the whole space $\mJ_{\B c}([\Sigma])$. Our discussion closely follows the exposition in~\cite{AbGrKi}, where similar considerations are applied to the study of symplectomorphism groups of ruled $4$-manifolds.%

Given an integrable structure $J$ we write $H^{0,q}_{J}(M)$ for the $q^{\text{th}}$ Dolbeault cohomology group with coefficients in the sheaf of germs of holomorphic functions, and $H^{0,q}_{J}(TM)$ for the $q^{\text{th}}$ Dolbeault cohomology group with coefficients in the sheaf of germs of holomorphic vector fields. 
\begin{theorem}[{\cite[Theorem 2.3]{AbGrKi}}]\label{thm:integrable deformations}
Let $(M,\om)$ is a symplectic $4$-manifold and suppose $J\in\mJ^{int}(\om)$
\index{J@$\mJ^{int}(\om)$ -- space of all compatible integrable complex structures on $(M,\omega)$}
is an integrable complex structure for which the cohomology groups $H^{0,2}_{J}(M)$ and $H^{0,2}_{J}(TM)$ are zero. Then $\mJ^{int}(\om)$ is a submanifold of $\mJ(\om)$ in the neighborhood of $J$. Moreover, the moduli space of infinitesimal compatible deformations of $J$ in $\mJ^{int}(\om)$ coincides with the moduli space of infinitesimal deformations of $J$ in the set of all integrable structures, that is, it is given by $H^{0,1}_{J}(TM)$. Finally, the tangent space of $\mJ^{int}(\om)$ at $J$ is naturally identified with the direct sum
$T_{J}\big((\Diff_{[\om]}(M)\cdot J)\cap \mJ(\om)\big)\oplus H^{0,1}_{J}(TM)$.\qed
\end{theorem}

The next theorem gives a sufficient condition for the $\Symp(M,\om)$-orbit of an integrable structure $J$ to be homotopy equivalent to the partial orbit $(\Diff_{[\om]}(M)\cdot J)\cap \mJ(\om)$ of Theorem~\ref{thm:integrable deformations}.

\begin{theorem}[{\cite[Corollary 2.6]{AbGrKi}}]\label{thm:homotopy equivalence integrable}
Let $(M,\om,J)$ be a K\"ahler $4$-manifold. Suppose the inclusion $\Iso(\om,J)\into \Aut_{[\om]}(J)$ is a weak homotopy equivalence. Then the inclusion of the $\Symp(M,\om)$-orbit $\Symp(M, \om)/ \Iso(\om, J) \into (\Diff_{[\om]}(M)\cdot J)\cap \mJ(\om)$ is also a weak homotopy equivalence. \qed
\end{theorem}

Each stratum of $\mJ_{\B c}([\Sigma])$ is characterized by the existence of a $J$-holomorphic sphere $u:S^2\to (\widetilde{M}_n,J)$ representing a given homology class $A$. We define the universal moduli space
\[\mM(A,\mJ_{\B c})=\{(u,J)\in C^\infty(S^2,\widetilde{M}_n)\times \mJ_{\B c}([\Sigma])~|~[u]=A,~u~\text{is simple and}~J\text{-holomorphic}\}.\]
This is a Fréchet manifold whose image under the projection $\pi_A:\mM(A,\mJ_{\B c})\to\mJ_{\B c}([\Sigma])$ is precisely the stratum $\mJ_A$ defined by the class $A$. The next theorem gives sufficient conditions under which the projection $\pi_A$ is tranversal to the inclusion $\mJ_{\B c}^{int}([\Sigma])\into \mJ_{\B c}([\Sigma])$. Although the original statement found in~\cite{AbGrKi} is slightly weaker, the given proof establishes the following stronger form.

\begin{theorem}[{\cite[Theorem 2.9]{AbGrKi}}]\label{thmC}
Let $(M,\om,J)$ be a K\"ahler $4$-manifold for which the cohomology groups $H^{0,2}_{J}(M)$ and $H^{0,2}_{J}(TM)$ are zero. Suppose that $(u,J)\in\mM(A,\mJ(\om))$ is such that $u^{*}:H^{0,1}_{J}(TM)\to H^{0,1}(u^{*}(TM))$ is surjective. Then the projection $\pi_A:\mM(A,\mJ(\om))\to\mJ(\om)$ is tranversal at $(u,J)$ to the inclusion $\mJ^{int}(\om)\into \mJ(\om)$ and the infinitesimal complement to the image $\mJ_{A}$ of $\pi_A$ in a neighborhood of $J$ can be identified with $H^{0,1}_{J}(TM)/\ker(u^*)$. \qed 
\end{theorem}

\begin{lemma}\label{lemma:vanishing+surjectivity}
Let $(\widetilde{M}_n,J)$ be the complex blow-up of $\cp^2$ at $n$ distinct points. 
\begin{enumerate}
\item The cohomology groups $H^{0,2}_{J}(\widetilde{M}_n)$ and $H^{0,2}_{J}(T\widetilde{M}_n)$ are zero. 
\item Let $u:S^2\to (\widetilde{M}_n,J)$ be a $J$-holomorphic map representing any of the classes $L_{ijk}$ or $L_{1234}$. Then $u^{*}:H^{0,1}_{J}(TM)\to H^{0,1}(u^{*}(TM))$ is surjective.
\end{enumerate}
\end{lemma}
\begin{proof}
The first statement holds for any rational surfaces since $H^{0,2}_{J}(M)$ and $H^{0,2}_{J}(TM)$ are invariant under pointwise blow-ups, and since $H^{0,2}_{J}(\cp^2)=0$ and $H^{0,2}_{J}(T\cp^2)=0$, see~\cite{Ko2005} and~\cite[Exercise 10.5]{Ha-GTM257}. For the second statement, let $A$ be one of the classes $L_{ijk}$ or $L_{1234}$. We first note that there always exist $J$-holomorphic representatives of a class $F$ satisfying $c_1(F)=2$, $F\cdot F=0$, and $F\cdot A=1$. Note that representatives of $F$ foliate an open and dense subset of $\widetilde{M}_n$. Granted this, the proof of~\cite[Lemma 7.11]{AP13} applies mutatis mutandis to our case.
\end{proof}

\begin{proposition}\label{prop:properties integrable strata} 
Let $(\widetilde{M}_{\B c},\widetilde{\om}_{\B c})$ be the symplectic blow-up of $\cp^2$ at $n\leq 4$ balls of capacity $\B c$. 
\begin{enumerate}
\item For each integrable complex structure $J\in\mJ_{\B c}([\Sigma])$, there is a homotopy equivalence
\[\Symp(\widetilde{M}_{\B c},\widetilde{\om}_{\B c})/\Iso(\widetilde{\om}_{\B c},J_I)\into (\Diff_{[\om]}(M)\cdot J)\cap \mJ_{\B c}([\Sigma]).\]
\item The inclusion $\mJ_{\B c}^{int}([\Sigma])\into \mJ_{\B c}([\Sigma])$ is transverse to each stratum. In particular, the stratification of $\mJ_{\B c}([\Sigma])$ induces a stratification of $\mJ_{\B c}^{int}([\Sigma])$ with strata of the same codimensions.
\item The isotropy representation along the orbit of an an integrable structure $J$ is isomorphic to the action of $\Iso(\widetilde{\om}_{\B c},J_I)$ on $H^{0,1}_{J}(TM)$.
\item For each stratum $\mJ_I$ with normal bundle $\mN_I$, the isotropy representation on the normal fiber $N_J$ at an integrable $J\in\mJ_I([\Sigma])$ is isomorphic to the action of $\Iso(\widetilde{\om}_{\B c},J_I)$ on $H^{0,1}_{J}(TM)/\ker(u^*)$.
\end{enumerate}
\end{proposition}
\begin{proof} 
The first statement follows from Proposition~\ref{prop:Symp action on J} together with Theorem~\ref{thm:homotopy equivalence integrable}. The other statements follows directly from Lemma~\ref{lemma:vanishing+surjectivity}, Theorem~\ref{thmC}, and Theorem~\ref{thm:integrable deformations}. 
\end{proof}

Among all strata in $\mJ_{\B c}$, the only stratum that is not homotopy equivalent to a single orbit under the action of the symplectomorphism group $\mG_{\B c}$ is the codimension $4$ stratum $\mJ_{1234}$. Its orbit structure can be understood in the following way. Consider the map $\chi:\mJ_{1234}\to\mM_{0,4}$ that assigns to $J$ the cross-ratio $\chi(p_4,p_1,p_2,p_3)$, where $p_i$ is the intersection of the $J$-holomorphic representatives $C_{1234}$ and $\Sigma_i$ of the classes $L_{1234}$ and $E_i$. Since there is a unique holomorphic parametrization $u:S^2\to C_{1234}$ that sends $(0,1,\infty)$ to $(p_1,p_2,p_3)$, this map is well defined. Moreover, since $\phi\circ u$ is $\phi^*J$-holomorphic, $\chi$ is invariant under the action of $\mG_{\B c}$. 

\begin{proposition}\label{prop:structure of the stratum J_1234}
The cross-ratio defines a fibration
\[\mJ_{\chi} \to \mJ_{1234} \xrightarrow{\chi} \mM_{0,4}\]
whose fiber is homotopy equivalent to the orbit under $\mG_{\B c}$ of an integrable complex structure.
\end{proposition}
\begin{proof}
Let $\Symp(S^2,k)$ be the group of symplectomorphisms of $S^2$ fixing a configuration of $k$ points. Given $w\neq 0,1,\infty$, there is an evaluation  fibration with local sections 
\[\Symp(S^2,4)\to\Symp(S^2,3)\to\mM_{0,4}\]
Fix a retraction of  $\Symp(S^2,3)\simeq\{\id\}$. Then any local section $s_z$ is given by the flow at time $t=1$ of some unique Hamiltonian on $S^2$, say $s_z(\xi)=\phi^{z,\xi}$. Given $J$ such that $\chi(J)=z$, and extending this Hamiltonian near the $J$-holomorphic curve $C_{1234}$ defines a local section of $\chi$ near $z$, namely, $S_z(\xi)=(\phi^{z,\xi})^*J$. This shows that $\chi$ is a fibration.

Let $C$ be a standard configuration of type $\mT_{1234}$ and consider the diagram of fibrations
\begin{equation}
\begin{tikzcd}
\mJ_{\chi}(C) \arrow[hook]{r}\ar[d] & \mJ_{\chi} \arrow[hook]{d}\arrow[twoheadrightarrow]{r}& \mC(\mT_{1234})\arrow[equal]{d}\simeq \mG_{\B c}/(\B S^1\times\B F_2)\\
*\simeq\mJ(C) \arrow[twoheadrightarrow]{d}{\chi} \arrow[hook]{r}& \mJ_{1234} \arrow[twoheadrightarrow]{d}{\chi}\arrow[twoheadrightarrow]{r}& \mC(\mT_{1234})\ar[d]\simeq \mG_{\B c}/(\B S^1\times\B F_2)\\
\mM_{0,4} \arrow[equal]{r} & \mM_{0,4} \arrow{r} & *
\end{tikzcd}
\end{equation}
The homotopy fiber of the leftmost vertical fibration is $\B F_2$, so that $\mJ_{\chi}(C) \simeq \B F_2$. On the other hand, the orbit $\mG_{\B c}/\Iso$ is the total space of the fibration 
\[
\B F_2 \to\mG_{\B c}/\Iso \to \mG_{\B c}/(\B S^1\times\B F_2) \simeq \mC(\mT_{1234}).
\]
Consequently, the inclusion $\mG_{\B c}/\Iso  \hookrightarrow \mJ_{\chi}$ is a homotopy equivalence as claimed. 
\end{proof}

\begin{proposition}\label{prop:equivalence 1234}
For every admissible capacity $\B c$, and for every stratum $\mJ_I([\Sigma])$, the inclusion $\mJ_{I}^{int}([\Sigma])\into \mJ_{I}([\Sigma])$ is a homotopy equivalence.
\end{proposition}
\begin{proof}
For the strata $\mJ_0$ and $\mJ_{ijk}$, Proposition~\ref{prop:Symp action on J}~(1-2) and Proposition~\ref{prop:properties integrable strata}~(1) imply that for each compatible almost complex structure $J$ in the stratum $\mJ_I$ we have equivalences
\[\mJ_I\simeq \Symp(\widetilde{M}_{\B c},\widetilde{\om}_{\B c})/\Iso(\widetilde{\om}_{\B c},J_I)
\simeq(\Diff_{[\om]}(M)\cdot J)\cap \mJ_{\B c}([\Sigma])\]
On the other hand, any complex integrable structure in $\mJ^{int}_I$ is obtained by blowing up a configuration of $n$ points in $\cp^2$ that are either in $F_0$ or in $F_{ijk}$. Since $\PGL(3)$ acts transitively on these sets, $\mD_h$ acts transitively on $\mJ^{int}_I$, so that $(\Diff_{[\om]}(M)\cdot J)\cap \mJ_{\B c}([\Sigma])=\mJ^{int}_I$.

For the stratum $\mJ_{1234}$, consider the ladder of fibrations
\begin{equation}
\begin{tikzcd}
\mJ_{\chi} \ar[r] & \mJ_{1234} \arrow{r}{\chi}& \mM_{0,4} \\
\mJ_{\chi}^{int} \ar[r] \ar[u, hookrightarrow] & \mJ_{1234}^{int} \ar[u, hookrightarrow] \arrow{r}{\chi}& \mM_{0,4} \ar[u,equal]
\end{tikzcd}
\end{equation}
By Proposition~\ref{prop:structure of the stratum J_1234}, each fiber is homotopy equivalent to an orbit. Consequently, the inclusion $\mJ_{\chi}^{int}\into \mJ_{\chi}$ is an equivalence and it follows that $\mJ_{1234}^{int}\into \mJ_{1234}$ is also a homotopy equivalence.
\end{proof}

Recall that given the stratification of $\mJ_{\B c}([\Sigma])$, we recursively defined $\mJ_0\subset\mJ_1\subset\cdots$ by setting $\mJ_0$ equal to the open stratum, and by successively adding strata of increasing codimensions. In the discussion below, $\mJ_i$ stands for a union of strata, while we keep using $\mJ_I$ when we refer to a single stratum.

\begin{proposition}\label{prop:equivalence with integrable}
For every admissible capacity $\B c$, the inclusion $\mJ_{\B c}^{int}([\Sigma])\into \mJ_{\B c}([\Sigma])$ is a homotopy equivalence.
\end{proposition}
\begin{proof}
Both spaces are partitioned into the same number of strata, namely, $\mJ_{\B c}^{int}([\Sigma])=\bigsqcup_I \mJ_I^{int}$ and $\mJ_{\B c}([\Sigma])=\bigsqcup_I \mJ_I$, the corresponding strata $\mJ_I^{int}$ and $\mJ_I$ are homotopy equivalent, their normal bundles $\mN_I^{int}$, $\mN_I$, are homotopy equivalent, and the gluing data are isomorphic. The statement then follows inductively from applying~\cite[Theorem p.80]{May-Course} to the homotopy equivalent excisive triads $\big(\mJ_i^{int}([\Sigma]),\, \mN_i^{int},\, \mJ_{i-1}^{int}([\Sigma])\big)\to \big(\mJ_i([\Sigma]),\, \mN_i,\, \mJ_{i-1}([\Sigma])\big)$ obtained by adding one stratum at a time. For more details see~\cite[Theorem 1.1]{AbGrKi}.
\end{proof}

\subsection{Existence of homotopy pushout squares}
Let $J$ be an integrable structure that belongs to the stratum of highest codimension $\mJ_I^{int}\subset\mJ_{\B c}^{int}([\Sigma])$. Although the orbit $\mO_J:=(\mG_{\B c}\cdot J)$ is a smooth $\mG_{\B c}$-invariant submanifold of finite codimension, it is not known whether one can find a $\mG_{\B c}$-invariant tubular neighborhood. However, the action of $\mG_{\B c}$ on $\mJ_{\B c}^{int}([\Sigma])$ admits slices. As explained in~\cite[Appendix C]{AbGrKi}, this implies that for any tubular neighborhood $\mN_J$ of $\mO_J$, and for any compact set $K\in\mG_{\B c}$, there is a smaller tubular neighborhood $\mN_J(K)$ sent by the action of $K$ into $\mN_J$. This is enough to define a $A_\infty$-action of $\mG_{\B c}$ on $\mN_J$ which is equivalent to the left $\mG_{\B c}$-action on a standard tube $\mG_{\B c}\times_{H_J} N_J$, where $N_J\simeq H^{0,1}_J(T\widetilde{M}_n)$ is the normal fiber at $J$ endowed with the isotropy representation of $H_J=\Iso(\om_{\B c},J)$. We summarize the previous discussion in the next proposition.

\begin{proposition}\label{prop:tubular nbhood}
Let $J\in\mJ_{\B c}^{int}([\Sigma])$ be an integrable structure and let $H_J=\Iso(\om_{\B c},J)$. Suppose its $\mG_{\B c}$-orbit $\mO_J$ is of codimension $(k+1)$. Then there exists a tubular neighborhood $\mN_J$ of $\mO_J$ such that the projection $(\mN_J\setminus\mO_J)\to\mO_J$ has the weak homotopy type of the projection $\mG_{\B c}\times_{H_J} S^k \to\mG_{\B c}/H_J$ of a standard tube. Moreover, one can define a $A_\infty$-action of $\mG_{\B c}$ on $\mN_J$ which is weakly homotopy equivalent to the left $\mG_{\B c}$-action on $\mG_{\B c}\times_{H_J} S^k \to\mG_{\B c}/H_J$.
\end{proposition}
\begin{proof}
The first part of the statement is analogous to Proposition~C.6 in~\cite[Appendix C]{AbGrKi}. The second part follows from formula~(A.1) in the same reference.
\end{proof}

\begin{proposition}\label{prop:existence pushouts}
Let $J\in\mJ_{\B c}^{int}([\Sigma])$ be an integrable structure that belongs to a stratum $\mJ_I^{int}$ of highest codimension. Let $S_J$ be the corresponding $J$-holomorphic configuration of type $\mT_I$. Let $u:S^2\to S_J$ be a $J$-holomorphic parametrization of the component in the homology class ($L_{ijk}$ or $L_{1234}$) characterizing the stratum $\mJ_I$.
\begin{enumerate}
\item If the stratum $\mJ_I$ is of codimension $2$, the homotopy orbit $\big(\mJ_{\B c}^{int}([\Sigma])\big)_{h\mG_{\B c}}\simeq \OP{B}\mG_{\B c}(\Sigma)$ is equivalent to the homotopy pushout
\begin{equation*}
\begin{tikzcd}
*\simeq\big(S^{1}\big)_{h\Iso(\widetilde{\om}_{\B c},J)} \arrow{r}{} \arrow{d}{} & \OP{B}\Stab(S_J)\simeq\B S^1 \arrow{d}{} \\
\big(\mJ_{\B c}^{int}([\Sigma])\setminus\mJ_I^{int}\big)_{h\mG_{\B c}}\arrow{r}{} & \OP{B}\mG_{\B c}(\Sigma)
\end{tikzcd}
\end{equation*}
where $\big(S^{1}\big)_{h\Iso(\widetilde{\om}_{\B c},J)}$ is the homotopy orbit of the isotropy representation of $\Iso(\widetilde{\om}_{\B c},J)$ on the unit sphere in $H^{0,1}_J(T\widetilde{M}_n)\simeq \C$. 
\item If the stratum $\mJ_I$ is of codimension $4$, the homotopy orbit $\big(\mJ_{\B c}^{int}([\Sigma])\big)_{h\mG_{\B c}}\simeq \OP{B}\mG_{\B c}(\Sigma)$ is equivalent to the homotopy pushout
\begin{equation*}
\begin{tikzcd}
\mM_{0,4}\times\big(S^{3}\big)_{h\Iso(\widetilde{\om}_{\B c},J)} \arrow{r}{} \arrow{d}{} & \OP{B}\Stab(S_J)\simeq \mM_{0,4}\times\B S^1 \arrow{d}{} \\
\big(\mJ_{\B c}^{int}([\Sigma])\setminus\mJ_I^{int}\big)_{h\mG_{\B c}}\arrow{r}{} & \OP{B}\mG_{\B c}(\Sigma)
\end{tikzcd}
\end{equation*}
where $\big(S^{3}\big)_{h\Iso(\widetilde{\om}_{\B c},J)}$ is the homotopy orbit of the isotropy representation of $\Iso(\widetilde{\om}_{\B c},J)$ on the unit sphere in $H^{0,1}_J(T\widetilde{M}_n)/\ker(u^*)\simeq \C^2$.
\end{enumerate}
\end{proposition}
\begin{proof}
Suppose the stratum $\mJ_I$ is of codimension $2$. Let $\mN_I$ be a tubular neighborhood of $\mJ_I^{int}$ and consider the the pushout diagram describing the union $\mJ_{\B c}^{int}([\Sigma])=(\mJ_{\B c}^{int}([\Sigma])\setminus \mJ_I^{int})\cup \mN_I$, namely,
\[
\begin{tikzcd}
\mN_{I}\setminus\mJ_{I}^{int} \arrow{r}{} \arrow{d}{} & \mN_{I} \arrow{d}{} \\
\mJ_{\B c}^{int}([\Sigma])\setminus \mJ_I^{int} \arrow{r}{} & \mJ_{\B c}^{int}([\Sigma])
\end{tikzcd}
\]
By Proposition~\ref{prop:Symp action on J}, each stratum $\mJ_I=\mJ_{ijk}^{int}$ of codimension $2$ is homotopy equivalent to an orbit under the action of $\mG_{\B c}$, and $\Iso(\widetilde{\om}_{\B c},J)\simeq \Stab(S_I)$. By Proposition~\ref{prop:tubular nbhood}, the upper map $\mN_{I}\setminus\mJ_{I}^{int} \to \mN_{I}$ is equivalent to the projection $\mG_{\B c}\times_{H_J} S^1 \to\mG_{\B c}/H_J$, where $H_J=\Iso(\widetilde{\om}_{\B c},J)\simeq \B S^1$ acts with weight one. Applying the Borel construction relative to the action of $\mG_{\B c}$ proves the statement.

Now consider the only stratum of codimension $4$, i.e. $\mJ_I^{int}=\mJ_{1234}^{int}$. Applying the Borel construction relative to the action of $\mG_{\B c}$ to the diagram of inclusions, we get
\[
\begin{tikzcd}
\mN_{I}\setminus\mJ_{I}^{int} \arrow{r}{} \arrow{d}{} & \mN_{I} \arrow{d}{} \\
\mJ_{\B c}^{int}([\Sigma])\setminus \mJ_I^{int} \arrow{r}{} & \mJ_{\B c}^{int}([\Sigma])
\end{tikzcd}
\implies
\begin{tikzcd}
\big(\mN_{I}\setminus\mJ_{I}^{int}\big)_{h\mG_{\B c}} \arrow{r}{} \arrow{d}{} & \big(\mN_{I}\big)_{h\mG_{\B c}}\simeq \mM_{0,4}\times \B S^1\arrow{d}{} \\
\big(\mJ_{\B c}^{int}([\Sigma])\setminus \mJ_I^{int}\big)_{h\mG_{\B c}} \arrow{r}{} & \big(\mJ_{\B c}^{int}([\Sigma])\big)_{h\mG_{\B c}}\simeq \OP{B}\mG_{\B c}(\Sigma)
\end{tikzcd}
\]
To better understand the upper map $\big(\mN_{I}\setminus\mJ_{I}^{int}\big)_{h\mG_{\B c}} \to \big(\mN_{I}\big)_{h\mG_{\B c}}\simeq \mM_{0,4}\times \B S^1$, consider the fibration 
\begin{equation}\label{eq:fiber over M(0,4)}
Y\to\mN_I\setminus\mJ_I\xrightarrow{\chi} \mM_{0,4}
\end{equation}
obtained by composing the projection $\mN_I\setminus\mJ_I\to \mJ_I$ with the cross-ratio map $\chi$. Its fiber $Y$ is itself a fibration over $\mJ_\chi$ with fiber homotopy equivalent to $S^3$. Since $\mJ_\chi$ is homotopy equivalent to an orbit $\mG_{\B c}/{\B S^1}$, the homotopy orbit $\big(Y\big)_{h\mG_{\B c}}$ is equivalent to $\big(S^3\big)_{h\B S^1}$, and since the action of $\B S^1$ on $S^3$ is free, we get $\big(S^3\big)_{h\B S^1}\simeq S^2$. On the other hand, the action of $\mG_{\B c}$ on $\mM_{0,4}$ being trivial, taking the homotopy orbits of the inclusion $Y\to\mN_I\setminus\mJ_I$ in the fibration~\eqref{eq:fiber over M(0,4)} gives another fibration 
\[S^2\simeq \big(S^3\big)_{h\B S^1} \to \big(\mN_{I}\setminus\mJ_{I}^{int}\big)_{h\mG_{\B c}} \to \mM_{0,4}.\]
Since $S^2$ is connected, and since $\mM_{0,4}\simeq S^1\vee S^1$, this last fibration must be trivial. This shows that the map $\big(\mN_{I}\setminus\mJ_{I}^{int}\big)_{h\mG_{\B c}} \to \big(\mN_{I}\big)_{h\mG_{\B c}}$ is homotopy equivalent to the product map
\[\mM_{0,4}\times S^2\simeq\mM_{0,4}\times\big(S^{3}\big)_{h\Iso(\widetilde{\om}_{\B c},J)} \to \mM_{0,4}\times\B S^1.\]
\end{proof}

\section{Homotopy orbits of the \texorpdfstring{$\Diff_{h}(\widetilde{M}_n)$}{Diff} action on \texorpdfstring{$\mA(n,[\Sigma])$}{A(n,S)}}\label{section: homotopy orbits Diff}

\subsection{Homotopy orbits for actions on almost complex structures}
We are now in a position to determine the homotopy orbits associated to the actions of diffeomorphism groups on spaces of almost complex structures. We state the results for the action of $\mD_h=\Diff_h(\widetilde{M}_n)$ on $\mA(n,[\Sigma])$ but, by Proposition~\ref{prop:equivalence G and D actions}, the same statements hold for the action of $\mD_h(\Sigma)=\Diff_h(\widetilde{M}_n,\Sigma)$ on $\mA(n,\Sigma)$.

\subsubsection{Homotopy orbits for $n=1$ and $n=2$}
For any admissible capacity $\B c$, the group $\mD_h$ acts on the unique stratum $\mA_0([\Sigma])$. From Proposition~\ref{prop:Symp action on J} (1) and Proposition~\ref{prop:equivalence G and D actions}, the homotopy orbit is
\[\mA_0([\Sigma])_{h\mD_h}\simeq\OP{B}\mG_{\B c}\simeq\OP{B}\Aut_h(J_0)\simeq \OP{B}\Iso(\widetilde{\om}_{\B c},J_0),\] 
that is, $\BU(2)$ in the case $n=1$, and $\OP{B}\B T^2$ in the case $n=2$.

\subsubsection{Homotopy orbits for $n=3$}
The group $\mD_h$ acts on $\mA(3,[\Sigma])=\mA_0([\Sigma])\sqcup\mA_{123}([\Sigma])$ preserving the stratification. By the Stability Theorem~\ref{thm:StabilityOfEmbeddings} and the description of chambers given in Section~\ref{section:stability chambers n=3}, we have equality $\mA_{\B c}([\Sigma])=\mA_0([\Sigma])$ whenever the capacities satisfy $c_1+c_2+c_3\geq 1$. By Proposition~\ref{prop:Symp action on J}, the stratum $\mA_0([\Sigma])$ is homotopy equivalent to a single orbit with stabilizer equivalent to a torus $\B T^2$. Consequently, the homotopy orbit $\mA_0([\Sigma])_{h\mD_h}$ is homotopy equivalent to $\OP{B}\B T^2$. On the other hand, we know from Proposition~\ref{prop:equivalence G and D actions} that the homotopy orbit $\mA_{\B c}([\Sigma])_{h\mD_h}=\mA_0([\Sigma])_{h\mD_h}$ is equivalent to $\BSymp(\widetilde{M}_{\B c},\Sigma)$. Since $\mA_0([\Sigma])$ and $\mD_h$ are independent of $\B c$, the homotopy orbit $\mA_0([\Sigma])_{h\mD_h}$ of the open stratum is \emph{always} homotopy equivalent to $\OP{B}\B T^2$.

When $c_1+c_2+c_3<1$, we have equality $\mA_{\B c}([\Sigma])=\mA_0([\Sigma])\sqcup \mA_{123}([\Sigma])=\mA(3,[\Sigma])$. By Proposition~\ref{prop:Symp action on J}, the stratum $\mA_{123}([\Sigma])$ is homotopy equivalent to a single orbit with stabilizer equivalent to a circle $\B S^1$. By Proposition~\ref{prop:existence pushouts}, applying the Borel construction to the diagram of inclusions gives
\[
\begin{tikzcd}
\mN(\mA_{123})\setminus\mA_{123} \arrow{r}{} \arrow{d}{} & \mN(\mA_{123}) \arrow{d}{} \\
\mA_{0}([\Sigma]) \arrow{r}{} & \mA(3,\Sigma)
\end{tikzcd}
\implies
\begin{tikzcd}
*\simeq\big(S^{1}\big)_{h\Iso(\widetilde{\om}_{\B c},J)}\arrow{r}{} \arrow{d}{} & \OP{B}\B S^1 \arrow{d}{} \\
\mA_0([\Sigma])_{h\mD_h} \simeq \OP{B}\B T^2 \arrow{r}{} & \mA(3,\Sigma)_{h\mD_h}\simeq \OP{B}\B T^2 \vee \OP{B}\B S^1.
\end{tikzcd}
\]
Again, we know from Proposition~\ref{prop:equivalence G and D actions} that the homotopy orbit $\mA_{\B c}([\Sigma])_{h\mD_h}=\mA(3,[\Sigma])_{h\mD_h}$ is equivalent to $\BSymp(\widetilde{M}_{\B c},\Sigma)$ whenever $c_1+c_2+c_3<1$. Consequently,
\begin{equation}\label{eq:precise statement n=3}
\OP{B}\mG_{\B c}=\BSymp(\widetilde{M}_{\B c},\Sigma)\simeq
\begin{cases}
\OP{B}\B T^2 & \text{if~}c_1+c_2+c_3\geq 1\\
\OP{B}\B T^2 \vee \OP{B}\B S^1 & \text{if~} c_1+c_2+c_3<1
\end{cases}
\end{equation}

\subsubsection{Homotopy orbits for $n=4$}
The group $\mD_h$ acts on the chain of inclusions
\[\mA_0([\Sigma])\subset\mA_1([\Sigma])~ \ldots ~\subset\mA_5([\Sigma])=\mA(4,[\Sigma]).\]
For $1\leq r\leq 4$, Proposition~\ref{prop:Symp action on J}, Proposition~\ref{prop:equivalence G and D actions}, and Proposition~\ref{prop:existence pushouts} give a sequence of pushout diagrams of inclusions and of homotopy orbits
\begin{equation}\label{eq:pushout diagram A-1}
\begin{tikzcd}[cramped,column sep=small]
\mN(\mA_{ijk})\setminus\mA_{ijk} \arrow{r}{} \arrow{d}{} & \mN(\mA_{ijk}) \arrow{d}{} \\
\mA_{r-1}([\Sigma]) \arrow{r}{} & \mA_r([\Sigma])
\end{tikzcd}
\implies
\begin{tikzcd}[cramped,column sep=small]
*\simeq\big(S^{1}\big)_{h\Iso} \arrow{r}{} \arrow{d}{} & \OP{B}\B S^1_{ijk} \arrow{d}{} \\
\mA_{r-1}([\Sigma])_{h\mD_h} \arrow{r}{} & \mA_{r}([\Sigma])_{h\mD_h}\simeq \mA_{r-1}([\Sigma])_{h\mD_h}\vee \OP{B}\B S^1_{ijk}
\end{tikzcd}
\end{equation}
By induction, we have $(\mA_{r})_{h\mD_h}\simeq \OP{B}\B S^1 \vee\cdots\vee \OP{B}\B S^1$ ($1\leq r\leq 4$ times). 

Adding the last stratum $\mA_{1234}$, and recalling that $\mM_{0,4}\simeq\OP{B}\B F_2$, Proposition~\ref{prop:existence pushouts} gives one more pair of diagrams of inclusions and homotopy orbits
\begin{equation}\label{eq:pushout diagram A-2}
\begin{tikzcd}[column sep=small]
\mN(\mA_{1234})\setminus\mA_{1234} \arrow{r}{} \arrow{d}{} & \mN(\mA_{1234}) \arrow{d}{} \\
\mA_{4}([\Sigma]) \arrow{r}{} & \mA_5([\Sigma])=\mA(4,\Sigma)
\end{tikzcd}
\implies
\begin{tikzcd}[column sep=small]
\OP{B}\B F_2\times(S^3)_{hS^1_{1234}} \arrow{r}{} \arrow{d}{} & \OP{B}\B F_2\times \OP{B}\B S^1_{1234} \arrow{d}{} \\
\mA_{4}([\Sigma])_{h\mD_h} \arrow{r}{} & \mA_5([\Sigma])_{h\mD_h}
\end{tikzcd}
\end{equation}
where $(S^3)_{hS^1_{1234}}\simeq S^2$.

By Proposition~\ref{prop:equivalence G and D actions}, the homotopy orbit $\mA_{\B c}([\Sigma])_{h\mD_h}$ is $\OP{B}\Symp(\widetilde{M}_{\B c},\Sigma)$, and by the Stability Theorem~\ref{thm:StabilityOfEmbeddings}, it only depends on the stability chamber $\B c$ belongs to. The description of the stability chambers given in Section~\ref{section:stability chambers n=4} shows that for $\B c$ in the chamber $C_r$, the space $\mA_{\B c}([\Sigma])$ is the union of strata $\mA_{r}([\Sigma])$.
Consequently, $\mA_{r}([\Sigma])_{h\mD_h}=\OP{B}\Symp(\widetilde{M}_{\B c},\Sigma)$ whenever $\B c$ is in the stability chamber $C_r$. Therefore,
\begin{equation}\label{eq:precise statement n=4}
\OP{B}\mG_{\B c}\simeq
\begin{cases}
* & \text{if~}c_2+c_3+c_4\geq 1\\
\OP{B} \B S^1   & \text{if~}  c_2 +c_3 +c_4 <  1 \text{~and~} c_1 +c_3 +c_4 \geq 1; \\
\OP{B} \B S^1 \vee B \B S^1  & \text{if~} c_1 +c_3 +c_4 <  1  \text{~and~} c_1 +c_2 +c_4 \geq 1; \\
\OP{B} \B S^1 \vee B \B S^1 \vee B \B S^1  & \text{if~} c_1 +c_2 +c_4 <  1 \text{~and~} c_1 +c_2 +c_3 \geq 1; \\
\OP{B} \B S^1 \vee \OP{B} \B S^1 \vee \OP{B} \B S^1 \vee \OP{B} \B S^1 & \text{if~}  c_1 +c_2 +c_3 <  1  \text{~and~} c_1 +c_2 +c_3 +c_4\geq 1.
\end{cases}
\end{equation}

\subsection{Comparison with configuration spaces of points}\label{section:proof of main thm}
We are finally in a position to prove the main geometric result of this paper. We only consider the case of $4$ balls as the proofs for the other cases are similar and simpler.  
\begin{proof}[Proof of Theorem~\ref{claim}]
Comparing the pushout diagrams~\eqref{eq:pushout diagram F-1} and \eqref{eq:pushout diagram F-2} describing the homotopy orbits of $\PGL(3)$ acting on $\Conf_n(\cp^2)$ with the diagrams~\eqref{eq:pushout diagram A-1} and \eqref{eq:pushout diagram A-2} describing the homotopy orbits of $\mD_h$ acting on $\mA(3,[\Sigma])$, we see that the top rows are always equivalent. It remains to show that the leftmost downward arrows induced by the inclusion of a deleted neighborhood of a stratum into the union of the other strata are also equivalent. This amounts to showing that, when adding a stratum $F_I$ or $\mA_I$, the gluing data for the inclusion
\[\mN(F_I)\setminus F_I\into F_{i-1}\]
are equivalent to the gluing data of the inclusion
\[\mN(\mA_{I})\setminus\mA_{I}\into \mA_{i-1}([\Sigma]).\]
By Proposition~\ref{prop:properties integrable strata}, the isotropy representation on the normal fiber of $\mN(\mA_{I})\setminus\mA_{I}$ at $J_I\in\mA_{I}$ is isomorphic to the action of $\Iso(\widetilde{\om},J_I)$ on the space of infinitesimal deformations $H^{0,1}_{J_I}(T\widetilde{M}_4)$. Suppose $J_I$ is obtained from the Fubini-Study structure $J_{FS}$ by blowing-up $\cp^2$ at a configuration of $4$ distinct points $\B p\in F_I$. As explained in~\cite[Exercise 10.5]{Ha-GTM257}, since the Fubini-Study structure is rigid, all such deformations correspond to blow-ups at other configurations in the normal fiber $U$ of $F_I$ over $\B p$. Under this identification, the infinitesimal action of $\Iso(\widetilde{\om},J_I)$ on $T_{\B p}U$ is, by construction, isomorphic to the complex representation $H^{0,1}_{J_I}(T\widetilde{M}_4)$. This shows that the gluing data for the strata in $\mA(4,[\Sigma])$ and in $\Conf_4(\cp^2)$ are isomorphic.
\end{proof}

\begin{remark}\label{remark:other proofs}
In a number of cases, the homotopy type of the symplectic stabilizer $\Symp(\widetilde{M}_{\B c},\Sigma)$ can be computed by combining the Stability Theorem, the computation of the stability chambers of Section~\ref{section: stability chambers}, and Lemma~\ref{lemma:Thickening}, with a few results on symplectomorphism groups of some monotone symplectic manifolds. For instance, in the cases $n=1$ and $n=2$, since there is a unique stability chamber, the homotopy type of embedding spaces is independent of the choice of capacities. Consequently, 
\[\IEmb_{n}(\B c,\cp^2)\simeq \varinjlim \IEmb_{n}(\B c,\cp^2) \simeq \Conf_{n}(\cp^2)\]
which agree with the results obtained in~\cite{P08i}. Similarly, in the cases $n=3$ and $n=4$, taking the limit as the capacities approach zero shows that $\IEmb_{n}(\B c,\cp^2) \simeq \Conf_{n}(\cp^2)$ for all capacities $\B c$ in the chamber $\sum_i c_i<1$ corresponding to "small" balls. For the other extremal chambers corresponding to "big" balls ($c_1+c_2+c_3\geq 1$ in the case $n=3$, $c_2+c_3+c_4\geq 1$ in the case $n=4$), we can choose the capacities to be equal to $c_i=1/3$. The symplectic blow-up is then monotone, so that we have a homotopy equivalence of the symplectic stabilizer $\Symp(\widetilde{M}_{\B c},\Sigma)$ with the symplectomorphism group $\Symp_h(\widetilde{M}_{\B c})$. It is shown in~\cite{Evans} that the latter group is contractible in the case $n=4$, and that it is homotopy equivalent to $\B T^2$ in the case $n=3$, in accordance with \eqref{eq:precise statement n=3} and~\eqref{eq:precise statement n=4}.
\end{remark}

\section{A rational model for \texorpdfstring{$\IEmb_{n}(\B c,\cp^2)$}{Im Emb (mBk,CP2)}}\label{section rational model}
The purpose of this section is to compute an algebraic model for the configuration space
$\IEmb_{n}(\B c,\cp^2)$ of $n$ embedded symplectic balls in $\cp^2$. It is done by
computing a relative model of the following fibration 
\[ \Symp(M,\iota(\mB_{\B c})) \to  \Symp_0(M,\omega) \to \IEmb_{n}(\B c,M). \]
Since the isotropy subgroup $\Symp(M,\iota(\mB_{\B c}))$ is homotopy equivalent to $\Symp(\widetilde{M_{\B c}},\Sigma)=\mG_{\B c}(\Sigma)$ by~\eqref{homotopy equivalence ball-fiber}, we will use the simpler notation $\mG_{\B c}(\Sigma)$ for $\Symp(M,\iota(\mB_{\B c}))$ in what follows.  
The above fibration induces the following one
\begin{equation}
\begin{tikzcd}\label{Eq:IEmb}
\Symp(\cp^2) \arrow[r] & \IEmb_{n}(\B c,\cp^2)\arrow[r] & \BG_{\B c}(\Sigma)
\end{tikzcd}
\end{equation}
which is the pull-back of the universal fibration
$$
\Symp(\cp^2) \to \ESymp(\cp^2)\to \BSymp(\cp^2)
$$
with respect to the map $\OP{B}i\colon \BG_{\B c}(\Sigma)\to \BSymp(\cp^2)$ induced by the inclusion
of the isotropy subgroup. Since $\Symp(\cp^2) \simeq {\rm PSU}(3)$ the relative model of the
above universal fibration is of the form
$$
\Lambda (\overline{\beta},\overline{\gamma})\to
\left( \Lambda (\overline{\beta},\overline{\gamma})\otimes \Lambda(\beta,\gamma),d_E \right)\to
\Lambda(\beta,\gamma)
$$
where $|\beta|=3,\ |\gamma|=5$ and $d_E\beta=\overline{\beta}$, $d_E\gamma=\overline{\gamma}$.
It follows from \cite[Theorem 2.70]{FOT08} that the relative model of the fibration \eqref{Eq:IEmb}
is given by
\begin{equation}
{\rm M}(\BG_{\B c}(\Sigma))\to
\left( {\rm M}(\BG_{\B c}(\Sigma))\otimes \Lambda(\beta,\gamma),D \right)\to
\Lambda(\beta,\gamma),
\label{Eq:MIEmb}
\end{equation}
where $D(\beta) = \OP{B}i^*(\overline{\beta})$ and $D(\gamma) = \OP{B}i^*(\overline{\gamma})$.
So we need to compute the minimal model of the classifying space of the isotropy subgroup and the map
$\OP{B}i^*\colon \rm{M}(\BSymp(\cp^2))\to \rm{M}(\BG_{\B c}(\Sigma))$ induced
on the minimal models by the inclusion. We will first compute the map induced
on the cohomology and then deduce the map on the minimal models.

Any circle action $\B S^1\to \Symp(\cp^2)$ factors as follows
\[
\B S^1\to \B T^2 \to \PSU(3)\to \Symp(\cp^2),
\]
where $\B T^2\to \PSU(3)$ is a maximal torus
\cite{Wil87}.
Moreover, the inclusion ${\rm PSU}(3)\to {\rm Symp}(\cp^2)$ is a homotopy equivalence.
Thus after a possible conjugation by an element of ${\rm PSU}(3)$ we can assume
that the maximal torus
consists of elements represented by matrices of the form
\[
\begin{pmatrix}
e^{is} & 0 & 0\\
0      & e^{it} & 0\\
0      & 0      & e^{-i(s+t)}
\end{pmatrix}
\in \SU(3).
\]
Since $\PSU(3)=\SU(3)/ \Z/3 \Z$ and since we are interested in real or rational
homotopy, we can replace $\PSU(3)$ by $\SU(3)$ in all considerations.

It is well known that the cohomology of the classifying space of a compact Lie group is 
isomorphic to the subalgebra of the cohomology of the maximal torus invariant with respect
to the action of the Weyl group. Thus, in our case, we have
\[
H^*(\BSU(3)) \cong H^*(\OP{B}\B T^2)^{S_3}\cong \Lambda(t_1,t_2)^{S_3}
\]
where $S_3$ denotes the symmetric group and $\deg(t_i)=2$.

The action of the symmetric group can be described as follows. The group $H^2(\OP{B}\B T^2)$ is canonically isomorphic to the dual of the Lie algebra of $\B T^2$. Let $t_1,t_2\in\mathfrak{t}^*$ be
the dual basis induced by the splitting $\B T^2=S^1\times S^1$.  Let
\begin{equation*}
e_1 = [1,0] = t_1,  \quad 
e_2 =[ 0,1] = t_2 \quad \mbox{and} \quad 
e_3 =[ -1,-1]
\end{equation*}
be the outward pointing normals to the facets of the polytope representing $\cp^2$ equipped with the standard toric action. The Weyl group $S_3$ acts on the dual Lie algebra $\R^2\cong H^2(\OP{B} \B T^2)$ as the permutation group of $\{e_1, e_2, e_3\}$. Since the cohomology $H^*(\OP{B}\B T^2)$ of the classifying space is generated by its cohomology in degree $2$, the action extends naturally to the
whole algebra. The invariant subalgebra is a free algebra generated by
\begin{align*}
\sigma_2 &= -e_1e_2 - e_2e_3 -e_3e_1 = t_1^2+t_2^2 +t_1t_2  \\
\sigma_3 &= -e_1e_2e_3 = t_1^2t_2+t_1t_2^2 .
\end{align*}
We get that
\[
H^*(\BSU(3)) = \Lambda (\sigma_2,\sigma_3).
\]

Any homomorphism $\B S^1\to \B T^2$ is of the form $e^{iz} \mapsto \left( e^{aiz},e^{biz} \right)$,
where $a,b\in \Z$ are integers. Such a homomorphism is injective if and only if
$\gcd(a,b)=1$. It follows that on the cohomology the circle action induces the homomorphism
$\Lambda(t_1,t_2)=H^*(\OP{B}\B T^2) \to H^*(\OP{B}\B S^1)=\Lambda(t)$ given by
\[
t_1\mapsto at \quad \text{and} \quad t_2\mapsto bt.
\]
Any homomorphism $\B S^1\to G$ to a compact Lie group $G$ factors, after a possible conjugation by an element of
$G$, through a homomorphism to the maximal torus. Since conjugation by an element of the group
acts trivially on cohomology we get that
$H^*(\OP{B}\SU(3))\to H^*(\OP{B}\B S^1)$ is given by
\begin{equation}
\sigma_2 \mapsto \left( a^2+b^2 +ab\right)t^2\quad\text{and}\quad
\sigma_3 \mapsto (a^2b+ab^2)t^3
\label{Eq:HBSU->HBS}
\end{equation}

Suppose that $H\subseteq \Symp(\cp^2)$ is a subgroup generated by $n$ circle actions
and such that it is an amalgamated free product of tori, that is a pushout of tori over a point as in the previous section. Then its classifying space $\OP{B}H$ is weak homotopy equivalent to the wedge product of the classifying spaces of the tori. 
Therefore, by \cite[Example 2.47]{FOT08}, an algebraic model for $\OP{B}H$ is generated
by elements $T_1,\ldots,T_n$ of degree $2$ corresponding to the circle actions, such that $T_iT_j=0$ if 
the generators $T_i$ and $T_j$ correspond to circle actions which are not in the same maximal torus. 

It follows that the minimal model of $\BG_{\B c}(\Sigma)$ is of the form
\begin{equation}
\left( \Lambda(T_1,\ldots,T_k) \otimes \Lambda W, d_{B}\right),
\label{Eq:MMStab}
\end{equation}
where $|T_i|=2$, the generators of the algebra $\Lambda W$ have degree $\geq 3$,  $d_{B}T_i=0$ and ${d_{B}}{|_W}$ is injective, due to the formality of the space
\cite[Exercise 2.3]{FOT08}.

Notice that $H^2(H)\cong \R^n$ is generated by $T_i$'s. Moreover, each inclusion
$\B S^1\to H$ induces the coordinate function $H^2(\OP{B}H)\cong \R^n\to \R = H^2(\OP{B}\B S^1)$.
Thus in order to compute the induced homomorphism
\[
H^2(\BSymp(\cp^2))=H^2(\BSU(3))\to H^2(\OP{B}H)
\]
it is enough to compute it for the generating circle actions as follows.

Suppose that the $i$-th circle action $\B S^1\to \Symp(\cp^2)$ is up to a conjugation a
homomorphism $\B S^1\to \B T^2$ given by the integers $(a_i,b_i)$. Then
$H^2(\OP{B} \B T^2)\to H^2(\OP{B}\B S^1)$ is given by
\[
t_1\mapsto a_iT_i\quad\text{and}\quad t_2\mapsto b_iT_i
\]
It follows that $(\OP{B}i)^*\colon H^*(\BSU(3))\cong\Lambda(\sigma_2,\sigma_3)\to H^*(\OP{B}H)$
is defined by
\begin{align}
\sigma_2=t_1^2 + t_2^2+t_1t_2 
&\mapsto 
\left( \sum_{i=1}^n a_iT_i \right)^2 + \left( \sum_{i=1}^n b_iT_i \right)^2 +\left( \sum_{i=1}^n a_iT_i \right) \left( \sum_{i=1}^n b_iT_i \right) \label{Eq:sigma2}\\
\sigma_3=t_1t^2_2 + t_1^2t_2
&\mapsto \left( \sum_{i=1}^n a_iT_i \right)
 \left( \sum_{i=1}^n b_iT_i \right)^2 +\left( \sum_{i=1}^n a_iT_i \right)^2
 \left( \sum_{i=1}^n b_iT_i \right)\label{Eq:sigma3} 
\end{align}

Since $\BSU(3)$ is rationally equivalent to the product of Eilenberg-Maclane spaces
$K(\B Q,4)\times K(\B Q,6)$, the set of homotopy classes of maps $\BSU(3)\to \OP{B}H$ is
in bijective correspondence with $H^4(\BSU(3);\B Q)\times H^6(\BSU(3);\B Q)$.
Moreover, we have the following commutative diagram of bijections
$$
\begin{tikzcd}
\left[ \OP{B}H,\BSU(3) \right]\ar[r,leftrightarrow]\ar[d,leftrightarrow]
& H^4(\OP{B}H;\B Q)\times H^6(\OP{B}H,\B Q)\ar[d,leftrightarrow]\\
\left[ \Lambda(\sigma_2,\sigma_3),{\rm M}(\OP{B}H) \right]\ar[r,leftrightarrow] & {\rm Hom}(H^*(\BSU(3);\B Q),H^*(\OP{B}H,\B Q)),
\end{tikzcd}
$$
where $\left[ \Lambda(\sigma_2,\sigma_3),{\rm M}(\OP{B}H) \right]$ denotes the space
of homotopy classes of morphism between minimal models of
$\BSU(3)\simeq \BSymp(\cp^2)$ and $\OP{B}H$.  It follows that the map
$\OP{B}i^*\colon \Lambda(\sigma_2,\sigma_3)\to {\rm M}(\OP{B}H)$ induced on the minimal
models is, up to a homotopy, given by formulas \eqref{Eq:sigma2} and
\eqref{Eq:sigma3} after suitable cancellations are made.
As explained above, it follows from \cite[Theorem 2.70]{FOT08} that the differential in
the relative model \eqref{Eq:MIEmb}
is given by $D(\beta) = \OP{B}i^*(\sigma_2)$ and $D(\gamma) = \OP{B}i^*(\sigma_3)$.
Consequently the relative model \eqref{Eq:MIEmb} is the minimal model of
$\IEmb_{n}(\B c,\cp^2)$.  More precisely we get the following result.

\begin{theorem}\label{T:model}
The minimal model of the embedding space $\IEmb_{n}(\B c,\cp^2)$ is
given by
\[
(\Lambda(T_1,\ldots,T_n) \otimes \Lambda W \otimes \Lambda(\beta,\gamma),d),
\]
where $|T_i|=2$, $|\beta|=3$, $|\gamma|=5$, the generators of the algebra $\Lambda W$ have degree $\geq 3$, $d T_i =0$, $d_{|W}$ is injective and  $d_{|\Lambda(\beta,\gamma)}$ is
given by the following formulas depending on the configurations:
\begin{enumerate}
\item For a configuration of $3$ balls:
\begin{enumerate}
\item if { $c_1 +c_2 +c_3 \geq 1$} then
\begin{align*}
d(\beta) &=
\left(a_1T_1+a_2T_2 \right)^2 + \left(b_1T_1+b_2T_2 \right)^2
+\left(a_1T_1+a_2T_2 \right) \left(b_1T_1+b_2T_2 \right),\\
d(\gamma)
&=
\left(a_1T_1+a_2T_2 \right)
 \left( b_1T_1+b_2T_2 \right)^2 +\left(a_1T_1+a_2T_2 \right)^2
 \left( b_1T_1+b_2T_2 \right).
\end{align*}
\item if { $c_1 +c_2 +c_3 < 1$} then
\begin{align*}
d(\beta) &=
\left(a_1T_1+a_2T_2 \right)^2 + \left(b_1T_1+b_2T_2 \right)^2
+\left(a_1T_1+a_2T_2 \right) \left(b_1T_1+b_2T_2 \right) + (a_3^2+a_3b_3+b_3^2)T_3^2,\\
d(\gamma)
&=
\left(a_1T_1+a_2T_2 \right)
 \left( b_1T_1+b_2T_2 \right)^2 +\left(a_1T_1+a_2T_2 \right)^2
 \left( b_1T_1+b_2T_2 \right) + (a_3^2b_3+a_3b_3^2)T_3^3.
\end{align*}
\end{enumerate}

\item For a configuration of $4$ balls:
\begin{align*}
d(\beta)&=
\sum_{i=1}^n (a_i^2+a_ib_i+b_i^2)T_i^2,\\
d(\gamma)&=
\sum_{i=1}^n (a_i^2b_i+a_ib_i^2)T_i^3,
\end{align*}
where $n=1,2,3,4$ depending on the capacities.
\end{enumerate}
\qed
\end{theorem}
In our concrete situation, in order to compute the model it is enough to find
the integers $(a_i,b_i)$ for each circle action from its geometric properties.

\section{The rational cohomology ring of \texorpdfstring{$\IEmb_{n}(\B c,\cp^2)$}{IEmb(mBk,CP2)}}\label{section cohomology ring}

The main goal of this section is to prove Theorem \ref{cohomology ring 4 balls}, that is, to compute the rational cohomology ring  of the space  of unparametrized balls $\IEmb_{4}(\B c, \cp^2)$, using the computation of its rational model obtained in the previous section. We also show that, using the same approach, we recover the cohomology ring of $\IEmb_{3}(\B c, \cp^2)$ in both cases of big and smalls balls, when this space is homotopy equivalent to the flag manifold $\PU(3)/\B T^2$ and the configuration space $\Conf_3(\cp^2)$, respectively. In the latter case, the cohomology ring has been obtained earlier by Ashraf and Berceanu in \cite{AB14} through methods completely different from ours. More precisely, they use an algebraic model of the configuration space obtained independently by Kriz \cite{Kr94}  and Totaro \cite{Tot96} that we explain in section \ref{section:small4balls}.
The case of $3$ balls serves also to illustrate our methods in a simpler context before addressing the $4$ balls case. 

\subsection{The rational cohomology ring of \texorpdfstring{$\IEmb_{3}(\B c, \cp^2)$}{Im Emb (mBk,CP2)}} \label{section:cohomology3balls} In this section let $\B c = (c_1,c_2,c_3)$.
We need to consider two cases: 

\subsubsection{The case $c_1 +c_2 +c_3 \geq  1$} It follows from \eqref{eq:precise statement n=3} that  $\mG_{\B c}(\Sigma)\simeq \B T^2$ and from Theorem \ref{3balls} that $\IEmb_{3}(\B c, M)$ is homotopy equivalent to the flag manifold $\PU(3)/ \B T^2 \simeq \U(3)/\B T^3$. We construct an algebraic model for the fibration, 
\begin{equation}\label{unparametrized 3balls}
 \Symp(\cp^2)\to\IEmb_{3}(\B c, \cp^2) \to \BG_{\B c}(\Sigma),
\end{equation}
using Theorem \ref{T:model}. Note that the stabilizer $\B T^2$ is the standard toric action on $\cp^2$ equipped with the monotone symplectic form, so the coefficients corresponding to this action are $a_1=1, a_2=0,b_1=0$ and $b_2=1$. Therefore we obtain 
 \[ \OP M(\IEmb_{3}(\B c, \cp^2)) = (\Lambda(T_1,T_2)\otimes \Lambda(\beta, \gamma),d) \]
 where 
 \[ d(T_i) =0, \quad d(\beta)=T_1^2+ T_2^2 + T_1T_2, \quad \mbox{and} \quad d(\gamma)=T_1T_2^2 + T_1^2T_2. \]
Then, computing the cohomology  of this model, it follows that the cohomology groups in degrees $1$ and $3$ are trivial ($\beta$ is not closed) and in degree $2$ clearly there are $2$ generators, namely $T_1$ and $T_2$. In degree 4, since the class $T_1^2+ T_2^2+ T_1T_2$ is exact it follows that the rank of $H^4(\IEmb_{3}(\B c, \cp^2))$ is $2$ and we can choose as generators of this cohomology group, for example, $T_2^2$ and $T_1T_2$. Then, the group in degree $5$ vanishes, because the class $\gamma$ is not closed as well as the classes $T_1\beta$ and $T_2\beta$ and there is no linear combination of these that gives a closed class. Indeed, we have $d(T_1\beta)= T_1^3 + T_1T_2^2 + T_1^2T_2$ and $d(T_2\beta)=  T_1T_2^2 +T_1^2T_2 + T_2^3$. Moreover, these together with the fact that the class $T_1T_2^2 + T_1^2T_2$ is exact implies that in degree 6 the classes $T_1^3$ and $T_2^3$ are also exact and the rank of the cohomology group is $1$. We can choose, for example,  $T_1T_2^2$ as a generator of this group. Finally, it is not hard to verify that the cohomology groups of degree $k$, with $k \geq 7$  are trivial, so the cohomology ring is given by 
\[ H^*(\IEmb_{3}(\B c, \cp^2), \Q) = \Lambda(T_1,T_2)/ (T_1^2+T_2^2 + T_1T_2, T_1^3), \]
where $|T_i|=2$, $i=1,2$. Moreover, this ring coincides with the  standard presentation of the rational cohomology ring of the flag manifold, $H^*(\U(3)/\B T^3)$.

\subsubsection{The case $c_1 +c_2 +c_3 <  1$}  In this situation, it follows from \eqref{eq:precise statement n=3} that $\BG_{\B c}(\Sigma)\simeq \OP{B} \B T^2 \vee \OP{B} \B S^1$  where the torus is the same as in the previous case and the circle  does not commute with any of the circles contained in the torus. It follows that 
\[ \OP M(\BG_{\B c}(\Sigma))= (\OP M (\OP{B} \B T^2) \oplus \OP M (\OP{B}\B S^1), 0) = \Lambda(T_1,T_2,T_3) / (T_1T_3,T_2T_3). \]
and the integers $(a_i,b_i)$,  with $i=1,2$, in Theorem \ref{T:model} are the same as in the previous case. For the third action, as we are going to see, we do not need to specify the values of the coefficients, so let us denote them simply by $(a,b)$.  It follows that an algebraic model of the fibration \eqref{unparametrized 3balls} is given by 
\[
\OP M (\IEmb_{3}(\B c, \cp^2)) =(\Lambda(T_1,T_2,T_3)/(T_1T_3,T_2T_3)\otimes \Lambda(\beta, \gamma),d)
\]
where
\[
d(T_i) =0, \quad d(\beta)=T_1^2+ T_2^2 + T_1T_2 +(a^2+ab +b^2) T_3^2 \quad \mbox{and} \] \[
d(\gamma)=T_1T_2^2 + T_1^2T_2 + (a^2b+ab^2)T_3^3.
\]
Note that $a^2+b^2+ab$ does not vanish unless $a=b=0$. Next we compute the cohomology of this model. Clearly it vanishes in degree $1$ and there are 3 generators in degree $2$, namely $T_1,T_2,T_3$. Since the class $\beta$ is not closed it follows that the cohomology in degree $3$ is trivial. In degree $4$,  the class  $T_1^2+ T_2^2 + T_1T_2 +(a^2+ab +b^2) T_3^2$ is exact and $T_1T_3=T_2T_3 =0$, so the rank of this cohomology group is $3$. Its generators can be $T_2^2,T_3^2$ and $T_1T_2$, for example. Then, since 
\begin{align*}
&d(T_1\beta) =  T_1^3  + T_1T_2^2 + T_1^2T_2, \\  
 &  d(T_2\beta) =  T_1^2T_2 + T_2^3 + T_1T_2^2, \\
  & d(T_3\beta)  =   (a^2+ab +b^2)T_3^3 \quad \quad \mbox{and} \\
   &d(\gamma)  =  T_1T_2^2 + T_1^2T_2 + (a^2b+ab^2)T_3^3
\end{align*}
it is easy to check that there are no nontrivial closed classes in degree $5$. Moreover, these differentials show that the classes $T_1^3, T_2^3,T_3^3$ and $T_1T_2^2 + T_1^2T_2$ are exact, so the rank of the cohomology group of degree $6$ is $1$ and it can be generated by $T_1T_2^2$. Now consider the class 
\[\eta: = (a^2+ab +b^2)T_3 \gamma - (a^2b+ab^2)T_3^2 \beta.\]
It is clear that $d \eta =0$, so this class is closed and gives a generator in degree $7$. On the other hand, note that the differentials 
\begin{align*}
     d(T_1 \gamma) & = T_1^3T_2 + T_1^2 T_2^2,  & &d(T_1^2 \beta) = T_1^4+ T_1^2 T_2^2 +  T_1^3 T_2 \\
     d(T_2 \gamma)  & = T_1^2T_2^2+T_1 T_2^3,  &  &d(T_2^2 \beta) =  T_1^2 T_2^2 +T_2^4 + T_1T_2^3 \\
     d(T_3 \gamma)  & = (a^2b+ab^2)T_3^4, & & d(T_1 T_2\beta) = T_1^3T_2+ T_1T_2^3 + T_1^2 T_2^2
\end{align*}
together with the relations $T_1T_3=T_2T_3=0$ and 
\[
\quad d(\beta\gamma)= (T_1^2+ T_2^2 + T_1T_2 +(a^2+ab +b^2) T_3^2) \gamma -\beta (T_1T_2^2 + T_1^2T_2 + (a^2b+ab^2)T_3^3)
\]
imply that the cohomology in degree $8$ is trivial. Finally, it is clear that $d(T_3 \eta)=0$, so the class $T_3 \eta$ is closed and it is not difficult to see that is not exact, so it represents a generator in degree $9$. Moreover, one can  show that there are no more nontrivial classes in this degree. In particular, clearly $\eta T_1 = \eta T_2 =0$. Since $d(T_3\beta\gamma)=\eta T_3^2$ it follows that the class $\eta T_3^2$ vanishes in cohomology. Moreover, all cohomology groups of degree $k$, for $k \geq 10$, vanish.  We conclude that 
\begin{align*}
H^*(\IEmb_{3}(\B c, \cp^2); \Q) = \Lambda(T_1,T_2,T_3,\eta) /  (T_1^2+ & T_2^2 + T_1T_2 +  (a^2+ab +b^2) T_3^2,\\
 & T_1T_3,\, T_2T_3,\,  T_1^3,\, \eta T_1,\, \eta T_2),  
\end{align*}
where $|T_i|=2$, $i=1,2,3$ and $|\eta|=7 $.

\subsubsection{Another approach}
It is interesting to compare our result with the one obtained by S. Ashraf and B. Berceanu in \cite[Theorem 1.3]{AB14} using a completely different approach to the problem. They showed that the cohomology ring of the space $\Conf_3(\cp^2)$ is given by 
\begin{equation}\label{CohomologyConf}
H^*(\Conf_3(\cp^2); \Q)= \Lambda(\alpha_1,\alpha_2,\alpha_3,\zeta)/ (\alpha_i^2+\alpha_j^2 + \alpha_i\alpha_j, \, \alpha_1^3, \, \zeta (\alpha_i-\alpha_j), \, i \neq j),
\end{equation}
where $|\alpha_i|=2$, $i=1,2,3$ and $|\zeta|=7.$ We confirm that we indeed obtain the same result by giving a ring  isomorphism,
$H^*(\Conf_3(\cp^2); \Q)  \to  H^*(\IEmb_{3}(\B c, \cp^2); \Q)$, defined by 
\begin{eqnarray*}
\alpha_1 & \mapsto & T_1 +c\, T_3, \\
\alpha_2 & \mapsto & T_2 + c\, T_3, \\
\alpha_3 & \mapsto  & -T_1 -T_2 + c\, T_3, \\
\zeta & \mapsto & \eta,
\end{eqnarray*}
where $c = \displaystyle {\sqrt{\frac{a^2+ab +b^2}{3}}}$.

Indeed, computing the image of the relations in \eqref{CohomologyConf} by this map we obtain 
\begin{eqnarray*}
\alpha_1^2+\alpha_2^2 + \alpha_1\alpha_2 & \mapsto & T_1^2+T_2^2 + T_1T_2 +3c^2 T_3^2 + 3cT_1T_3+ 3cT_2T_3, \\
\alpha_1^2+\alpha_3^2 + \alpha_1\alpha_3 & \mapsto & T_1^2+T_2^2 + T_1T_2 +3c^2 T_3^2 - 3cT_2T_3, \\
\alpha_2^2+\alpha_3^2 + \alpha_2\alpha_3 & \mapsto & T_1^2+T_2^2 + T_1T_2 +3c^2 T_3^2 - 3cT_1T_3, \\
\alpha_1^3 & \mapsto & T_1^3 +3cT_1^2T_3 + 3c^2T_1T_3^2 + c^3T_3^3, \\ 
\zeta (\alpha_1-\alpha_2) & \mapsto &  \eta(T_1-T_2),\\
\zeta (\alpha_1-\alpha_3) & \mapsto &  \eta(2T_1+T_2),\\
\zeta (\alpha_2-\alpha_3) & \mapsto &  \eta(T_1+2T_2),
\end{eqnarray*}
which imply the relations in  $H^*(\IEmb_{3}(\B c, \cp^2); \Q)$.

\subsection{The rational cohomology ring of \texorpdfstring{$\IEmb_{4}(\B c, \cp^2)$}{Im Emb (mBk,CP2)}}\label{section:cohomology4balls}
In this section let $\B c=(c_1,c_2,c_3,c_4)$.
\begin{proof}[Proof of Theorem \ref{cohomology ring 4 balls}]
Recall from the previous section that the fibration 
\begin{equation}\label{unparametrized 4balls}
\Symp(\cp^2)\to\IEmb_{4}(\B c, \cp^2) \to \BG_{\B c}(\Sigma),
\end{equation}
can be used to compute a rational model of $\IEmb_{4}(\B c, \cp^2)$. 
Since the homotopy type of $\BG_{\B c}(\Sigma)$ depends on the sizes of the capacities $c_i$, it follows that the algebraic model of the fibration depends also on these capacities. We need to consider five different cases:
\begin{itemize}
\item if $c_2 +c_3 +c_4 \geq 1$ then $\mG_{\B c}(\Sigma)$ is contractible so $\IEmb_{4}(\B c, \cp^2)  \simeq \Symp(\cp^2) \simeq \PU(3)$. Therefore
\[
H^*(\IEmb_{4}(\B c, \cp^2), \Q) = \Lambda(\beta, \eta),
\]
where $\Lambda(\beta,\eta)$ denotes an exterior algebra on generators $\beta$ of degree $3$ and $\eta$ of degree $5$;
\item if $c_2 +c_3 +c_4 <  1$ and $c_1 +c_3 +c_4 \geq 1$ then it follows from \eqref{eq:precise statement n=4} that $\BG_{\B c}(\Sigma)$ is weak homotopy equivalent to $ B\B S^1$. Then Theorem \ref{T:model} implies that a rational model for $\IEmb_{4}(\B c, \cp^2)$ is given by 
\[\OP M(\IEmb_{4}(\B c, \cp^2)) = (\Lambda(T)\otimes \Lambda(\beta, \gamma),d)\]
where $|T|=2$, $|\beta|=3$ and $|\gamma|=5$.  Moreover, the differential satisfies 
\[
d(T)=0,  \quad d(\beta) = (a^2+b^2+ab)T^2 \quad \mbox{and} \quad d(\gamma)=(a^2b+ab^2)T^3,
\]
where the circle action is given by the integers $(a,b)$.  The cohomology of this algebraic model is the rational cohomology of the space, so in degree $2$ there is one generator, namely  $T$,  and the cohomology groups of degrees $3$ and $4$ vanish since $\beta$ is not closed ($a^2+b^2+ab$ does not vanish unless $a=b=0$).  Then in degree $5$ there is one generator since the class $\eta:=(ab^2+ba^2)T\beta-(a^2+b^2+ab)\gamma$ is closed, as one can easily verify. It is clear that the cohomology in degree $6$ is trivial ($T^3$ is exact) and $d(T\eta)=0$, so the rank of $H^7(\IEmb_{4}(\B c, \cp^2)$ is $1$. Moreover, since $d(\gamma \beta)= T^2 \eta$ it follows that $T^2 \eta$ is exact and $H^k(\IEmb_{4}(\B c, \cp^2), \Q)=0$ for $k \geq 8$. Hence
\[H^*(\IEmb_{4}(\B c, \cp^2), \Q) = \Lambda(T,\eta)/ (T^2)\]
where  $|T|=2$ and $|\eta|=5$;
\item if $c_1 +c_3 +c_4 <  1$ and $c_1 +c_2 +c_4 \geq 1$ then $\BG_{\B c}(\Sigma)\simeq B\B S^1 \vee B\B S^1$ (see \eqref{eq:precise statement n=4}). Therefore, by Example 2.47 in \cite{FOT08}, its algebraic model is given by 
\[
\OP M(\BG_{\B c}(\Sigma))= (\OP M (B\B S^1) \oplus \OP M (B\B S^1), 0) = \Lambda(T_1,T_2) / (T_1T_2),
\]
where $|T_i|=2$, $i=1,2$. Assume the two circle actions are given by the integers $(a_i,b_i)$, $i=1,2$. Let $m_i = a_i^2+b_i^2+a_ib_i $  and $n_i= a_i^2b_i+a_ib_i^2$. Note that $m_i\neq 0$.  Using Theorem \ref{T:model} again one obtains 
\[
\OP M(\IEmb_{4}(\B c, \cp^2)) = (\Lambda(T_1,T_2)/ (T_1T_2)\otimes \Lambda(\beta, \gamma),d)
\]
where 
\begin{equation*}
d(T_i)=0, \quad 
d(\beta)  = m_1T_1^2 + m_2T_2^2  \quad \mbox{and} \quad
d(\gamma)  = n_1T_1^3 + n_2T_2^3. 
\end{equation*}
The cohomology of this algebraic model clearly contains $2$ generators in degree $2$, namely $T_1$ and $T_2$ and it vanishes in degree $3$, because $\beta$ is not closed. Then in degree $4$ there is only one class, since the class $m_1T_1^2 + m_2T_2^2$ is exact and $T_1 T_2=0$. Next, it is easy to check that the class $\eta := m_1m_2 \gamma -n_1m_2T_1\beta -n_2m_1T_2\beta$ is closed and since it is not exact it follows it represents a generator in degree 5. Therefore we obtain two classes in degree $7$ which are closed, namely  $T_1\eta$ and  $T_2\eta$. The cohomology groups in degrees $6$ and $8$ are trivial because on one hand $d(T_i \beta) = m_iT_i^3$, so the classes $T_i^3$ with $i=1,2$ are exact, and on the other hand $d(\gamma\beta)= (n_1T_1^3 + n_2T_2^3)\beta - \gamma (m_1T_1^2 + m_2T_2^2) $, so the class $\gamma\beta$ is not closed. The latter also implies that the rank of the cohomology group in degree $9$ is $1$, generated, for example, by the class $T_1^2\eta$. Since $ d(m_jT_i\beta \gamma) = m_jm_i\gamma T_i^3 -n_im_jT_i^4 \beta= \eta T_i^3$ for $i \neq j$ it follows that the remaining cohomology groups are trivial. Therefore the cohomology ring in this case is given by 
\[H^*(\IEmb_{4}(\B c, \cp^2), \Q) = \Lambda(T_1,T_2,\eta)/ (m_1T_1^2+m_2T_2^2,T_1T_2),\]
where $|T_1|=|T_2|=2$ and $|\eta|=5$;
\item if $c_1 +c_2 +c_4 <  1$ and $c_1 +c_2 +c_3 \geq 1$ then it follows from \eqref{eq:precise statement n=4} that in this case $\BG_{\B c}(\Sigma) \simeq B\B S^1 \vee B\B S^1 \vee B\B S^1$ and the computation of the cohomology ring is similar to the previous case. In this case the minimal model for $\BG_{\B c}(\Sigma)$ is given by 
\[
\quad \quad \quad   \OP M(\BG_{\B c}(\Sigma))  =  (\OP M (B\B S^1) \oplus \OP M (B\B S^1)\oplus \OP M (B\B S^1), 0) \\  =  \Lambda ( T_1,T_2,T_3) / ( T_iT_j, i \neq j )
\]
where $|T_i|=2$, $i=1,2,3$. It follows that a  rational model for the space of unparametrized balls is given by  \[ \OP M(\IEmb_{4}(\B c, \cp^2)) = (\Lambda(T_1,T_2,T_3)/ ( T_iT_j, i \neq j )\otimes \Lambda(\beta, \gamma),d) \]
where 
\begin{equation*}
d(T_i)=0, \quad 
d(\beta)  = \sum_{i=1}^{3}m_iT_i^2  \quad \mbox{and} \quad
d(\gamma)  = \sum_{i=1}^{3}n_iT_i^3. 
\end{equation*}
where $m_i$ and $n_i$ are defined above. The computation of the  cohomology ring of  this model is similar to the previous case, so we leave it to the interested reader. It gives the following ring
\[
H^*(\IEmb_{4}(\B c, \cp^2), \Q) = \Lambda(T_1,T_2,T_3,\eta)/ (m_1T_1^2+m_2T_2^2+m_3T_3^2,\ T_iT_j \ \mbox{if} \ i\neq j),
\]
where $|T_i|=2$, $i=1,2,3$,  and $|\eta|=5$;
\item finally, if $c_1 +c_2 +c_3 <  1$ and $c_1 +c_2 +c_3 +c_4 \geq 1$, the classifying space $\BG_{\B c}(\Sigma)$ is again a wedge of classifying spaces of circles. More precisely, $\BG_{\B c}(\Sigma) \simeq B \B S^1 *  B\B S^1 * B \B S^1* B\B S^1$. Therefore it should be clear that the rational cohomology ring in this case is given by 
\[
\quad \quad H^*(\IEmb_{4}(\B c, \cp^2), \Q) = \Lambda(T_1,T_2,T_3,T_4,\eta)/ \left(\sum_{i=1}^{4}m_iT_i^2, \ T_iT_j \ \mbox{if} \ i\neq j\right),
\]
where $|T_i|=2$, $i=1,2,3,4$,  and $|\eta|=5$. 
\end{itemize}
Note that in all cases the cohomology ring does not depend on the integers $(a_i,b_i)$ giving the circle actions and  setting $\alpha_i := \sqrt{m_i}T_i$ we obtain the cohomology ring presentations of Theorem \ref{cohomology ring 4 balls}. This completes the proof of the theorem. 
\end{proof}

\subsubsection{The space of embeddings \texorpdfstring{$\IEmb_{4}(\B c, \cp^2)$}{Im Emb (mBk,CP2)} in the case of small balls}\label{section:small4balls}

Note that in this case, by Theorem~\ref{homotopy type 4balls}, the space $\IEmb_{4}(\B c, \cp^2)$ is weakly homotopy equivalent to $\Conf_4(\cp^2)$. However, Theorem~\ref{cohomology ring 4 balls} does not give the cohomology ring of $\IEmb_{4}(\B c, \cp^2)$. The reason is that,  when $c_1+c_2+c_3+c_4<1$, it is not as easy to find the minimal model of $\OP{B}\mG_{\B c}(\Sigma)$
as in the previous cases. Although the classifying space of the stabilizer still has the homotopy type of a pushout, in this case the pushout diagram is not over a point, which prevents us to describe its model as before. Nevertheless,  from  Theorem \ref{homotopy type 4balls} we can obtain the rank of the rational cohomology groups $H^k(\OP{B}\mG_{\B c}(\Sigma);\Q)$, $k \geq 0$ as follows. 

\begin{proposition}
If $c_1+c_2+c_3+c_4<1$ the rational cohomology groups of the classifying space of the stabilizer $\OP{B}\mG_{\B c}(\Sigma)$ are given by 
\begin{equation}\label{rkcohomologygps}
H^{q}(\OP{B}\mG_{\B c}(\Sigma);\Q)\simeq
\begin{cases}
\Q &\text{if~} q=0\\
0 &\text{if~} q=1\\
\Q^{4} &\text{if~} q=2\\
0 &\text{if~} q=3\\
\Q^{5}&\text{if~} q=2k\geq 4\\
\Q^{2}&\text{if~} q=2k+1\geq 5.
\end{cases}
\end{equation}
\end{proposition}
\begin{proof}
The homotopy pushout decomposition of the classifying space of the stabilizer can be understood using a finite dimensional model given in terms of configuration spaces. 
From the equivalence of diagram~\eqref{eq:pushout diagram F-2} and  diagram~\eqref{eq:pushout diagram A-2} we obtain a Mayer-Vietoris sequence
\[0\to H^{1}(P_{5})\to H^{1}(P_{4})\oplus H^{1}(\OP{B} \B S^{1}\times\mM_{0,4})\to H^{1}(S^{2}\times \mM_{0,4})\to H^{2}(P_{5})\to \cdots\]
where $P_5\simeq \OP{B}\mG_{\B c}(\Sigma)$. Since $\mM_{0,4}\simeq \B S^{1}\vee \B S^{1}$, 
\[
H^{q}(\mM_{0,4})\simeq
\begin{cases}
\Q &\text{if~} q=0\\
\Q^{2} &\text{if~} q=1\\
0 &\text{if~} q\geq 2.
\end{cases}
\]
The map $S^{2}\times \mM_{0,4} \to \OP{B} \B S^{1}\times\mM_{0,4}$ induces a surjection in cohomology, so the Mayer-Vietoris sequence splits and we get 
\begin{align*}
0&\to& H^{1}(P_{5})&\to& H^{1}(P_{4})\oplus H^{1}(\OP{B} \B S^{1}\times\mM_{0,4})&\to& H^{1}(S^{2}\times \mM_{0,4})&\to& 0\\
& & 0 & & 0\oplus\Q^{2}& & \Q^{2}& & \\
0&\to& H^{2}(P_{5})&\to& H^{2}(P_{4})\oplus H^{2}(\OP{B} \B S^{1}\times\mM_{0,4})&\to& H^{2}(S^{2}\times \mM_{0,4})&\to& 0\\
 & &\Q^{4} & &\Q^{4}\oplus\Q & &\Q & & \\
0&\to& H^{3}(P_{5})&\to& H^{3}(P_{4})\oplus H^{3}(\OP{B} \B S^{1}\times\mM_{0,4})&\to& H^{3}(S^{2}\times \mM_{0,4})&\to& 0\\
& &0 & &0\oplus\Q^{2} & &\Q^{2} & & \\
0&\to& H^{4}(P_{5})&\to& H^{4}(P_{4})\oplus H^{4}(\OP{B} \B S^{1}\times\mM_{0,4})&\to& 0\\
& &\Q^{5}       & & \Q^{4}\oplus \Q                                & &  \\
0&\to& H^{5}(P_{5})&\to& H^{5}(P_{4})\oplus H^{5}(\OP{B} \B S^{1}\times\mM_{0,4})&\to& 0\\
& & \Q^{2} & & 0\oplus \Q^{2}& & \\
\end{align*}
Consequently, we obtain the desired cohomology groups.
\end{proof}

We now use an algebraic model of $\Conf_{4}(\cp^2)$ constructed independently by  Kriz~\cite{Kr94} and Totaro~\cite{Tot96} for configuration spaces to obtain the rational cohomology ring of $\OP{B}\mG_{\B c}(\Sigma)$. More precisely, for $M$  a smooth complex projective variety of complex dimension $m$, they  constructed a rational model $E(M,k)$ for $\Conf_{n}(M)$. Let us remind their construction. 
 
Let $\Omega \in H^{2m}(M)$ denote a fixed orientation class in $M$. For an arbitrary basis $\{a_i \}_{i=1,2, \hdots, q}$, in $H^*(M)$, take the dual basis $\{b_j \}_{j=1,2, \hdots, q}$, $(a_i \cup b_j = \delta_{ij}\Omega)$ and construct the {\it diagonal class} of $M$ by $\Delta= \sum_{i=0}^q a_i \otimes b_i \in H^*(M^2)$. 

For $a\neq b \in \{1, \hdots, k\}$, let $p_a^*\colon H^*(M) \to H^*(M^k)$ and 
$p^*_{ab}\colon H^*(M^2)\to H^*(M^k)$ be the pullbacks of the projection maps
\[ p_a:M^k \longrightarrow M, \quad p_a(x_1,\hdots,x_a, \hdots, x_n)=x_a,\]
and
\[
p_{ab}:M^k \longrightarrow M^2, \quad p_{ab}(x_1,\hdots,x_a, \hdots, x_b, \hdots, x_n)=(x_a,x_b),
\]
respectively. 

\begin{definition}[\cite{Kr94}]\label{Krizmodel}
Denote by $H^*(M^k)[G_{ab}]$ the algebra over $H^*(M^k)$ with generators $G_{ab}$, $ 1 \leq a \neq b \leq k$, of degree $2m-1$. The Kriz model $E(M,k)$ for  $\Conf_{n}(M)$ is the differential graded algebra (DGA)  given by the quotient of \[H^*(M^k)[G_{ab}],\]
modulo the following relations 
\begin{enumerate}
    \item $G_{ab}=G_{ba}$,
    \item $p_a^*(x)G_{ab}= p_b^*G_{ab}$ for $x \in H^*(M)$,
    \item $G_{ab}G_{bc}+ G_{bc}G_{ca} + G_{ca}G_{ab}=0$.
\end{enumerate}
The differential $d$ of degree $+1$ is given by 
\[d(p_a^*(x))=0 \]
and 
\[d(G_{ab})= p^*_{ab} (\Delta). \]
\end{definition}

Kriz proved 

\begin{theorem}[\cite{Kr94}]
Let $M$ be a complex projective manifold of dimension  $m$. Then the DGA $E(M,k)$ is a rational model, in the sense of Sullivan, of the configuration space $\Conf_{n}(M)$. 
\end{theorem}

We are interested in the model of $\Conf_{n}(\cp^2)$. The cohomology algebra of $M=\cp^2$ is given by $H^*(M, \Q)= \Q[x]/\langle x^3 \rangle$ where $\deg x= 2$, that is, $x^i$, with $i=0,1,2$ is a basis of $H^{2i}(M;\Q)$ and all other cohomology groups are zero.  First we construct the Kriz model $E(\cp^2,3)$. Using the K\"unneth formula we find the canonical basis of $H^{2i}(M^3)$: $x^j\otimes x^p \otimes x^q$, such that $i=j+p+q$. Let us denote the generators of $H^2(\cp^2)^{\otimes 3}$, namely $x\otimes 1\otimes 1$, $1\otimes x\otimes 1$ and $1\otimes 1\otimes x$, by $\alpha_1$, $\alpha_2$ and $\alpha_3$, respectively. Then we add the exterior part generated by $G_{12}, G_{13}, G_{23}$, where the generators have degree 3 and satisfy the following relations
\begin{enumerate}
    \item $G_{ab}=G_{ba}$,
    \item $\alpha^i_aG_{ab}=\alpha^i_bG_{ab}$,
    \item $G_{ab}G_{bc}+ G_{bc}G_{ca} + G_{ca}G_{ab}=0$,
\end{enumerate}
where $ 1 \leq a \neq b \neq c \leq 3$ and $i=1,2$, by Definition \ref{Krizmodel}. The differential is given by 
\[ d (\alpha_a) = 0  \] and 
\[ d (G_{ab})= p^*_{ab}(\Delta)= \alpha_a^2+ \alpha_a\alpha_b + \alpha_b^2,\] with $ 1 \leq a \neq b \leq 3$.

In fact, it was this model that Ashraf and Berceanu used in \cite{AB14} to obtain an explicit presentation of the cohomology ring of $\Conf_3(\cp^2)$. However, to our knowledge, no such presentation is known for $H^*(\Conf_4(\cp^2))$. On the other hand, using the algebraic model described above together with the software SageMath we can understand the rational cohomology ring $H^*(\Conf_4(\cp^2)$ and use it to compute the rational cohomology ring of the classifying space of the stabilizer $\mG_{\B c}(\Sigma)$.

\begin{theorem}\label{cohomology_stab_small_balls}
Consider $\cp^2$ equipped with its standard Fubini-Study symplectic form and let  $c_1,c_2,c_3, c_4 \in (0,1)$ such that $c_1+c_2+c_3+c_4< 1$. Then the  rational cohomology ring of the classifying space of the stabilizer $\mG_{\B c}(\Sigma)$ is given by 
\[ H^*(\BG_{\B c}(\Sigma);\Q) = \Lambda (\alpha_1,\alpha_2,\alpha_3,\alpha_4,\eta_1,\eta_2) / I,  \]
where $|\alpha_i|=2$, $|\eta_i|=5$ and $I$ is the ideal given by 
\[
I = \left( \begin{array}{c}
\alpha_1^2 -\alpha_2^2 +\alpha_1\alpha_i-\alpha_2\alpha_i,  \ \alpha_2^2 -\alpha_i^2 +\alpha_1\alpha_2-\alpha_1\alpha_i, \  i=3,4  \\
\alpha_3^2 -\alpha_4^2 +\alpha_1\alpha_3-\alpha_1\alpha_4, \\  \eta_i (\alpha_j -\alpha_k), \ i=1,2, \ j,k=1,2,3,4, \\
\eta_1 \eta_2
\end{array}
\right)
\]
\end{theorem}
\begin{remark}
It follows from Theorem \ref{cohomology_stab_small_balls} that the generators of the cohomology modules  of the classifying space of the stabilizer, $H^q(\BG_{\B c}(\Sigma);\Q)$, are given by 
\begin{itemize}
\item $\alpha_i$ \ if  \  $q=2$; 
\medskip
\item $\alpha_1^k, \alpha_1^{k-2}\, \alpha_2^2, \alpha_1^{k-1}\, \alpha_2,\alpha_1^{k-1}\, \alpha_3,\alpha_1^{k-1}\, \alpha_4$ \ if \ $q=2k\geq 4$;
\medskip
\item $\eta_1 \alpha_1^{k-2}, \eta_2 \, \alpha_1^{k-2}$ \ if \ $q=2k+1\geq 5$,
\end{itemize}
which agrees with the computation of the cohomology groups in \eqref{rkcohomologygps}. 
\end{remark}
\begin{proof}
The main idea to prove the theorem is to construct the algebraic model of the fibration 
\begin{equation}\label{main_fibration}
\IEmb_{4}(\B c, \cp^2)   \to \BG_{\B c}(\Sigma) \to \BSymp (\cp^2) \simeq  \BPU (3). 
\end{equation}
In order to simplify the notation,  let us denote the rational homotopy groups $\pi_*(\cdot) \otimes \Q$  by $\pi_*(\cdot)$.  

Recall that any fibration $V \into P \to U$ for which the theory of minimal models applies gives rise to a sequence 
\begin{equation}\label{model_fibration}
(\OP M(U),d_U ) \longrightarrow (\OP M(U) \otimes \OP M (V),d) \longrightarrow (\OP M (V),d_V ),  
\end{equation}
where the middle differential graded algebra is a model for the total space of the fibration. Let $d_{|U}$ and $d_{|V}$ denote the restriction of the differential $d$ to $U$ and $V$, respectively. The theory of minimal models implies that 
\begin{equation*}
d_{|U}  = d_U \quad \mbox{and} \quad d_{|V}= d_V + d'
\end{equation*}
where $d'$ is a perturbation with image not in $\OP M(V)$. Moreover, the linear part of $d'$  is dual to the boundary map $\partial_*\colon \pi_{*} (U) \to \pi_{*-1}(V)$. 

Note that in our case all the spaces are simply connected, so the theory of minimal models applies. The model for the base of the fibration is 
\[ \OP M(\BPU(3))=(\Lambda (Z,W),0)\]
where $|Z|=4$ and $|W|=6$.
Therefore \eqref{model_fibration} implies that 
\[
\OP M(\OP{B}\mG_{\B c}(\Sigma))= (\OP M(\IEmb_{4}(\B c, \cp^2)) \otimes \Lambda (Z,W), d)
\]
where $d \, Z= d \, W = 0$ and 
\[d_{|\IEmb} = d_{\IEmb}+d'\]
where the image of $d'$ is not in $M(\IEmb_{4}(\B c, \cp^2))$. Since the linear part of $d'$ is dual to the boundary map $\partial_*$, it follows from the long exact homotopy sequence of fibration~\eqref{main_fibration} that the image of the differential $d$ of elements of degree 3 and 5 might contain terms with $Z$ and $W$, respectively, depending  on $\partial_4\colon \pi_4(\BPU(3)) \to \pi_3(\IEmb)$ and $\partial_6\colon\pi_6(\BPU(3)) \to \pi_5(\IEmb)$ are trivial or not. 

\emph{Claim: $\partial_4$ and $\partial_6$ are not trivial.}

Note that if $\partial_4$ was trivial  then the long exact homotopy sequence 
\[
\dots \rightarrow \pi_4(\BPU(3)) \stackrel{\partial_4}{\longrightarrow} \pi_3(\IEmb) \longrightarrow \pi_3(\OP{B}\mG_{\B c}(\Sigma)) \longrightarrow \pi_3(\BPU(3)) \to \dots
\]
would imply there was an isomorphism between  $\pi_3(\IEmb)$ and $ \pi_3(\OP{B}\mG_{\B c}(\Sigma))$, because $\pi_3(\BPU(3))$ is trivial. 
Since $\IEmb_{4}(\B c, \cp^2) \simeq \Conf^4(\cp^2)$, we then implement the Kriz model in the software SageMath~\cite{SM} to obtain the minimal model of $\Conf_4(\cp^2)$ and therefore to obtain the rank of its rational homotopy groups in the relevant degrees for us. In particular, we obtain $\pi_3(\IEmb)=\Q^6$ which would imply that $\pi_2(\mG_{\B c}(\Sigma) = \pi_3(\OP{B}\mG_{\B c}(\Sigma))=\Q^6$.
We now use  the Serre-Leray spectral sequence of the universal fibration 
\begin{equation}\label{universal_fibration}
\mG_{\B c}(\Sigma) \to \OP{E}\mG_{\B c}(\Sigma) \to \OP{B}\mG_{\B c}(\Sigma)
\end{equation}
to prove that in fact $\pi_2(\Stab)$ cannot be isomorphic to $\Q^6$ and therefore $\partial_4$ is not trivial. 
More precisely, since $\mG_{\B c}(\Sigma)$ is an $H$-space, by the Cartan-Serre Theorem, the rational cohomology is a free algebra such that the number of generators of dimension $d$  is equal to the dimension of $\pi_d(\mG_{\B c}(\Sigma))$.  Moreover, the Kriz model implies that $\pi_2(\IEmb)=\Q^4$. Therefore it follows again from the long exact homotopy sequence 
\[ 
\dots \rightarrow \pi_3(\BPU(3)) \stackrel{\partial_3}{\longrightarrow} \pi_2(\IEmb) \longrightarrow \pi_2(\OP{B}\mG_{\B c}(\Sigma)) \longrightarrow \pi_2(\BPU(3)) \to \dots 
\]
that  $\pi_1(\mG_{\B c}(\Sigma)) =\pi_2(\OP{B}\mG_{\B c}(\Sigma)) = \pi_ 2(\IEmb) = \Q^4$. Note that these $4$ generators correspond to the $4$ circles appearing in the pushout diagram which yields the homotopy type of $\Stab$.   Then, from the computation of the cohomology groups of $H^*(\OP{B}\mG_{\B c}(\Sigma))$ in \eqref{rkcohomologygps}, it follows that the $E_2$ page of the spectral sequence is given as in Figure \ref{E2page} below.

\begin{figure}[H]
\begin{tikzpicture}
  \matrix (m) [matrix of math nodes,
    nodes in empty cells,nodes={minimum width=7ex,
    minimum height=7ex,outer sep=-7pt},
    column sep=1ex,row sep=0.1ex]{
                &  \hdots    &     &          &      &         &        & \\
          2     &  \Q^{11}   &  0  &  \hdots  &   0  & \hdots  &        & \\
          1     &  \Q^4      &  0  &  \Q^{16} &   0  & \Q^{20} & \hdots & \\
          0     &  \Q        &  0  &  \Q^4    &   0  & \Q^5    &  \Q^2  & \\
    \quad\strut &   0        &  1  &  2       &  3   &  4      &  5     & \strut \\};
\draw[-stealth] (m-3-2) -- (m-4-4);
\draw[-stealth] (m-2-2) -- (m-3-4);
\draw[-stealth] (m-3-4) -- (m-4-6);
\draw[thick] (m-1-1.east) -- (m-5-1.east) ;
\draw[thick] (m-5-1.north) -- (m-5-8.north) ;
\end{tikzpicture}
\caption{$E_2$-page of the spectral sequence of fibration \eqref{universal_fibration}.}\label{E2page}
\end{figure}
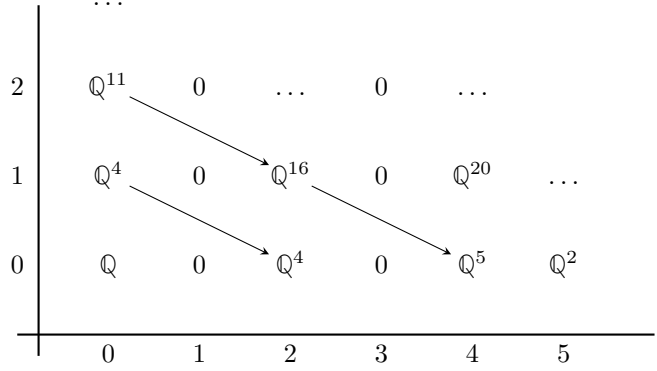
\noindent where the column $E_2^{p,0}$ contains the cohomology groups $H^p(\mG_{\B c}(\Sigma))$ while the row $E_2^{0,q}$ contains the cohomology groups $H^q(\OP{B}\mG_{\B c}(\Sigma))$. The $4$ generators of degree $1$ in $H^1(\mG_{\B c}(\Sigma))$ give rise to $6$ elements in $H^2(\mG_{\B c}(\Sigma))$, so there are only 5 new generators $H^2(\mG_{\B c}(\Sigma))$. If instead there were $6$ new generators then one element clearly would survive to the $E_\infty$-page of the spectral sequence, which is not possible. 
So we conclude that $\pi_2(\mG_{\B c}(\Sigma)) = \pi_3(\OP{B}\mG_{\B c}(\Sigma))=\Q^5$ and $\partial_4$ is not trivial. A similar argument shows that $\partial_6$ is not trivial as well. 

Next, using again the software SageMath we can check there is only one way of defining the differential $d_{|\IEmb}$ on elements of degree $3$ such that their image contain a non-zero term with $Z$. The analog holds for the differential of elements of degree $5$ whose image contain a non-zero term with $W$. Moreover, we can also verify that there are no non-linear terms in the image of the differential involving the generators $Z$ and $W$. Finally, this software gives the generators of the cohomology ring  $H^*(\OP{B}\mG_{\B c}(\Sigma);\Q)$ in degrees $2$ and $5$ and their relations. In order to confirm that these give the complete description of the cohomology ring we use the Leray-Serre spectral sequence of the fibration 
\begin{equation}\label{main_fibration2}
 \OP{PU(3)} \to \IEmb_{4}(\B c, \cp^2)  \to \OP{B}\mG_{\B c}(\Sigma)
\end{equation}
which will give in the $E_\infty$-page the rational cohomology groups of $\IEmb_{4}(\B c, \cp^2)$. 
These are also obtained from the Kriz model by the software SageMath
\begin{equation}\label{Emb_coho_modules}
H^{q}(\IEmb_{4}(\B c, \cp^2);\Q )\simeq
\begin{cases}
\Q &\text{if~} q=0,11\\
\Q^{4} &\text{if~} q=2,4,9\\
\Q^{2}&\text{if~} q=5,10,12\\
\Q^{6}&\text{if~} q=7 \\
0  & \text{otherwise.}
\end{cases}
\end{equation}
In the $E_2$ page of the spectral sequence the column $E_2^{p,0}$ contains $H^*(\OP{PU(3)};\Q)=\Lambda(\beta,\gamma)$ where $|\beta|=3$ and $|\gamma|=5$. The row $E_2^{0,q}$ contains the cohomology algebra $H^*(\OP{B}\mG_{\B c}(\Sigma);\Q)$. Since the cohomology of $\IEmb_{4}(\B c, \cp^2)$ in dimension $3$ is trivial, it follows that $d_k \beta$ is non-zero, for some $k$. The only possibility is $k=4$. Moreover, since $H^6(\IEmb_{4}(\B c, \cp^2);\Q)$ is trivial we conclude that $4$ elements in $H^6(\OP{B}\mG_{\B c}(\Sigma);\Q)$ should be in the image of $d_4(\beta \alpha_i)$, $i=1,2,3,4$ and the $5$th element is the image of $d_6 \gamma$. Recall that the cohomology module $H^6(\OP{B}\mG_{\B c}(\Sigma);\Q)$ is generated by $\alpha_1^3, \alpha_1\alpha_2^2, \alpha_1^2 \alpha_2,\alpha_1^2 \alpha_3,\alpha_1^2 \alpha_4$. Therefore we should have $d_4 \beta=\alpha_1^2$ and $d_6 \gamma = \alpha_1\alpha_2^2$. Then $d_4(\beta \alpha_i)= \alpha_1^2\alpha_i$ and there are no elements in dimension $6$ left in the $E_k$-page for $k \geq 7$. Moreover, the relevant pages in the spectral sequence, that is, the pages in which the differential is not trivial are $E_4$ and $E_6$. The page $E_4$ contains all elements in $H^*(\OP{PU(3)}) \otimes H^*(\OP{B}\mG_{\B c}(\Sigma))$ and the non-zero terms in the $E_6$-page are the ones in Figure \ref{E6page}, 
\begin{figure}[H]
\begin{tikzpicture}
  \matrix (m) [matrix of math nodes,
    nodes in empty cells,nodes={minimum width=7ex,
    minimum height=7ex,outer sep=-7pt},
    column sep=1ex,row sep=0.1ex]{
                &  \hdots &    &     &    &      &      &    &      &\\
          5     &  \gamma  &  0  &  \alpha_i \gamma  & 0  & (*)\gamma &  \eta_i \gamma & \alpha_1 \alpha_2^2\gamma & \alpha_1\eta_i\gamma &\\
        \vdots &  0  &  0  &  0     & 0  &  0   &  0   &  0 &  0   & \\
          0     &  \Q &  0  &  \alpha_i  & 0  & (*) & \eta_i & \alpha_1 \alpha_2^2 & \alpha_1\eta_i  & \\
    \quad\strut &  0  &  1  &  2     &  3 &  4   &  5   & 6  &  7 & \strut \\};
\draw[-stealth] (m-2-2) -- (m-4-8);
\draw[thick] (m-1-1.east) -- (m-5-1.east) ;
\draw[thick] (m-5-1.north) -- (m-5-10.north);
\end{tikzpicture}
\caption{$E_6$-page of the spectral sequence of fibration \eqref{main_fibration2}.}\label{E6page}
\end{figure}
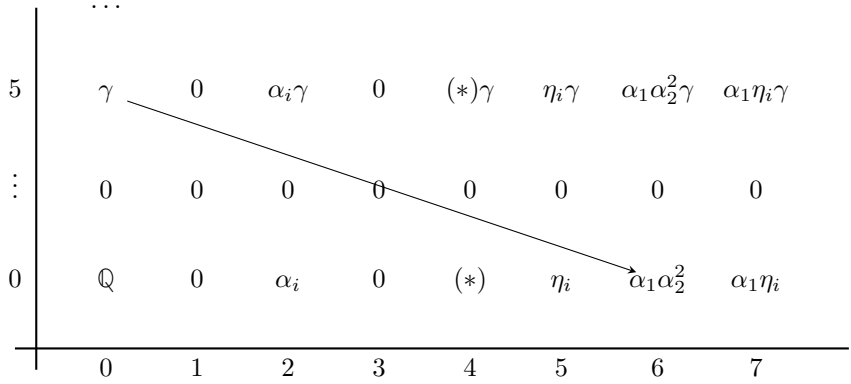
\noindent 
where (*) denotes the elements $\alpha_2^2, \alpha_1\alpha_2, \alpha_1\alpha_3, \alpha_1\alpha_4.$ Therefore in the $E_\infty$-page of the spectral sequence, in Figure \ref{E_infty}, we obtain the desired cohomology groups \eqref{Emb_coho_modules}. 
\begin{figure}[H]
\begin{tikzpicture}
  \matrix (m) [matrix of math nodes,
    nodes in empty cells,nodes={minimum width=7ex,
    minimum height=7ex,outer sep=-7pt},
    column sep=1ex,row sep=0.1ex]{
                &  \hdots &    &     &    &      &      &    &      &\\
          5     &  0  &  0  &  \Q^4  & 0  & \Q^4 &  \Q^2& \Q & \Q^2 & \\
         \vdots &  0  &  0  &  0     & 0  &  0   &  0   &  0 &  0   & \\
          0     &  \Q &  0  &  \Q^4  & 0  & \Q^4 & \Q^2 & 0  & \Q^2  & \\
    \quad\strut &  0  &  1  &  2     &  3 &  4   &  5   & 6  &  7 & \strut \\};
\draw[thick] (m-1-1.east) -- (m-5-1.east) ;
\draw[thick] (m-5-1.north) -- (m-5-10.north) ;
\end{tikzpicture}
\caption{$E_\infty$-page of the spectral sequence of fibration \eqref{main_fibration2}.}\label{E_infty}
\end{figure}
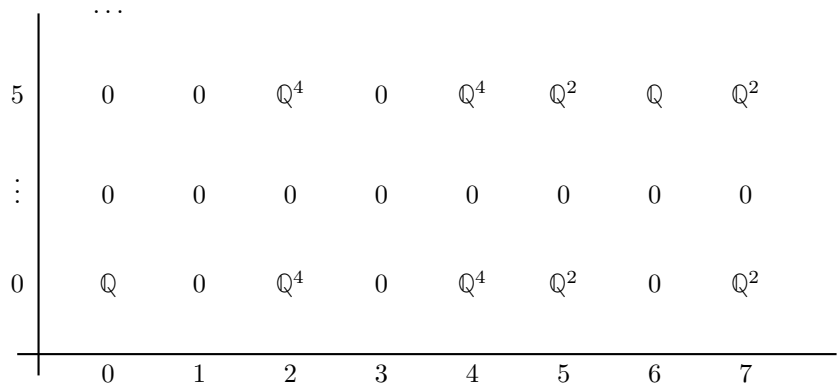
\end{proof}

\bibliographystyle{plain}

\begin{thebibliography}{AbGrKi}

\bibitem{AbGrKi} M. Abreu, G. Granja and N. Kitchloo, \emph{Compatible complex structures on symplectic rational ruled surfaces}, Duke Math. J. {\bf 148} (2009), 539--600.

\bibitem{AbMcD00}
M. Abreu and D. McDuff
{ \em Topology of symplectomorphism groups of rational ruled surfaces}, 
J. Amer. Math. Soc.13(2000), no.4, 971–1009.

\bibitem{AB14}
S. Ashraf and B. Berceanu. \newblock{\em Cohomology of 3-points configuration spaces of complex projective spaces.} \newblock{Advances in Geometry}, vol. 14  (2014), no. 4 pp. 691--718.

\bibitem{ALLP23}
S. Anjos, J. Li, T.-J. Li and M. Pinsonnault. 
\newblock {\em Stability of the symplectomorphism group of rational surfaces}
\newblock {to appear in Math. Ann.}

\bibitem{ALP09}  
S. Anjos, F.  Lalonde, and M. Pinsonnault.
\newblock {\em The homotopy type of the space of symplectic balls in rational ruled 4-manifolds.},
\newblock {Geom. Topol.}, {13 (2009)}, {no. 2},  {1177--1227}. 

\bibitem{AP13}
S. Anjos and M. Pinsonnault.
\newblock {\em The homotopy Lie algebra of symplectomorphism groups of 3-fold blow-ups of the projective plane.}
\newblock {Math. Z}, 275 (2013), no. 1-2, 245-292.

\bibitem{Ban78}
A. Banyaga,
{ \em Sur la structure du groupe des difféomorphismes qui préservent une forme symplectique}.  
Comment. Math. Helv.53(1978), no.2, 174–227.

\bibitem{Bi96} P. Biran, 
{\emph{Connectedness of spaces of symplectic embeddings,}} International Mathematics Research Notices, Volume 1996, Issue {\bf 10}, (1996), Pages 487--491. 

\bibitem{BLW14}
M. S. Borman, T.-J. Li and W. Wu,
{\em Spherical Lagrangians via ball packings and symplectic cutting.}
Selecta Math. (N.S.)20(2014), no.1, 261–283.

\bibitem{CM21} J. Chaidez and M. Munteanu, {\emph{Essential tori in spaces of symplectic embeddings}}, Algebr. Geom. Topol. {\bf 21} (2021), no. 5, 2489--2522. 

\bibitem{DoLiWu18} J.~G. Dorfmeister, T.-J. Li and W.  Wu, {\em Stability and existence of surfaces in symplectic 4-manifolds with b+=1},
J. Reine Angew. Math.742(2018), 115–155.

\bibitem{Evans} J. D. Evans. \newblock  {\emph{ Symplectic mapping class groups of some Stein and rational surfaces}}, \newblock  {J. Symplectic Geom. } {  9(2011), no. 1}, 45--82.

\bibitem{FT94}
Y. F\'{e}lix\ and\ J.-C. Thomas, {\em Homologie des espaces de lacets des espaces de configuration}, Ann. Inst. Fourier (Grenoble) {\bf 44} (1994), no.~2, 559--568. MR1296743

\bibitem{FOT08} 
Y. F\'{e}lix, J.~F. Oprea\ and\ D. Tanr\'{e}, {\it Algebraic models in geometry}, Oxford Graduate Texts in Mathematics, 17, Oxford Univ. Press, Oxford, 2008. MR2403898 

\bibitem{Gr85}  M. Gromov, {\emph{Pseudo holomorphic curves in symplectic manifolds}}, Invent.  Math., {\bf 82} (1985), 307--347.

\bibitem{Gua18}
R. Guadagni, \emph{Symplectic neighborhood of crossing divisors}, arXiv:1611.02363v2.


\bibitem{Ha-GTM257}
R. Hartshorne, \emph{Deformation theory}, Grad. Texts in Math., 257, Springer, New York, 2010. viii+234 pp.

\bibitem{Ke05}
J. K\c{e}dra, {\em Evaluation fibrations and topology of symplectomorphisms}, Proc. Amer. Math. Soc. {\bf 133} (2005), no.~1, 305--312. MR2086223

\bibitem{Ko2005}
K. Kodaira, {\em Complex Manifolds and Deformation of Complex Structures}, Classics Math., Springer, Berlin, 2005. MR 2109686

\bibitem{Kr94}
I. Kriz, {\emph{On the rational homotopy type of configuration spaces}}. Ann. of Math. (2) 139 (1994), no. 2, 227--237. 


\bibitem{LM96}
F. Lalonde and D. McDuff.
\newblock{\em $J$-curves and the classification of rational and ruled symplectic 4-manifolds.} {Contact and symplectic geometry (Cambridge, 1994), 3–42, 
Publ. Newton Inst., 8, Cambridge Univ. Press, Cambridge, 1996.}

\bibitem{LP04}
 {F. Lalonde and M. Pinsonnault},
 \newblock{\em The topology of the space of symplectic balls in rational 4-manifolds.}
\newblock{Duke Math. J.}, {122 (2004), no. 2}, {347-397}.

\bibitem{LLW15}
Li, J., Li, T.-J., Wu, W.,
{\em The symplectic mapping class group of $\mathbb{CP}^2\#n\overline{\mathbb {CP}}^2$ with $n \leq 4$}, Michigan Math. J., \textbf{64} (2015), no.2, 319–333.

\bibitem{LL95}T. J. Li and A. Liu, {\em Symplectic structure on ruled surfaces and a generalized adjunction formula}.
Math. Res. Lett. 2 (1995), no. 4, 453--471.

\bibitem{LL01}
{T-J. Li and  A-K. Liu},
 \newblock  {\em Uniqueness of symplectic canonical class, surface cone and symplectic cone of 4-manifolds with {$B^+=1$}.}
\newblock {J. Differential Geom.}, {58 (2001), no. 2}, {331--370}.

\bibitem{May-Course}
J. P. May, {\em A Concise Course in Algebraic Topology}, Chicago Lectures in Math., Univ. Chicago Press, Chicago, 1999. MR 1702278

\bibitem{McDOp15}
D. McDuff, E. Opshtein, {\em Nongeneric $J$-holomorphic curves and singular inflation}, 
Algebr. Geom. Topol.15 (2015), no.1, 231–286.

\bibitem{McD98}
D. McDuff. 
\newblock {\em From symplectic deformation to isotopy.}
\newblock {\it Topics in symplectic $4$-manifolds} (Irvine, CA, 1996), 85--99, First Int. Press Lect. Ser., I, Int. Press, Cambridge, MA, 1998.


\bibitem{McD09} D. McDuff, {\emph{Symplectic embeddings of 4-dimensional ellipsoids}}, J. Topol. {\bf 2} (2009), no. 1, 1--22.

\bibitem{MS17}
D. McDuff and D. Salamon.
\newblock {\em Introduction to Symplectic Topology}.
\newblock Oxford Mathematical Monographs, 3rd edition, 2017.

\bibitem{P08i}
M. Pinsonnault, 
\newblock{\em Symplectomorphism groups and embeddings of balls into rational ruled 4-manifolds.} \newblock{Compos. Math.}, {144 (2008)}, {no. 3}, {787--810}.

\bibitem{SM} SageMath-9-6, William Stein, Marc Culler, Nathan Dunfield, Matthhias Görner and others (2022).


\bibitem{Tot96} B. Totaro, {\emph{Configuration spaces of algebraic varieties}}, Topology {\bf 35} (1996), 1057--1067. 

\bibitem{Wil87} D. M. Wilczyński.
\newblock { \em Group actions on the complex projective plane}
\newblock{Trans. Am. Math. Soc.} {303 (1987)}, 707--731. 

\bibitem{Zhang17} W. Zhang, { \em The curve cone of almost complex 4-manifolds}, Proc. Lond. Math. Soc. (3) {\bf 115} (2017), no.~6, 1227--1275. MR3741851
\end{thebibliography}

\onehalfspacing
\printindex

\singlespacing
\end{document}